\documentclass{amsart}

\usepackage{amsmath,amsthm,amssymb,amsfonts,enumerate,color}
\usepackage{mathrsfs}
\usepackage{amstext,amsxtra}
\usepackage[margin=1.5in]{geometry}
\usepackage{graphicx}
\usepackage{float,appendix}

\usepackage{tikz}
\usetikzlibrary{matrix,arrows,calc,intersections,fit}
\usetikzlibrary{decorations.markings}
\usepackage{tikz-cd}
\usepackage{textcomp}

\usepackage[colorlinks,urlcolor=black,linkcolor=blue,citecolor=blue,hypertexnames=false]{hyperref}
\allowdisplaybreaks
\usepackage{pgf,tikz}
\usepackage{stmaryrd}
\usepackage{bm}

{\theoremstyle{plain}
\newtheorem{theorem}{Theorem}[section]
\newtheorem{proposition}[theorem]{Proposition}
\newtheorem{corollary}[theorem]{Corollary}
\newtheorem{lemma}[theorem]{Lemma}
\newtheorem{problem}[theorem]{Problem}}
{\theoremstyle{definition}

\newtheorem{definition}[theorem]{Definition}
}
{\theoremstyle{remark}
\newtheorem{remark}[theorem]{Remark}}

\newcommand{\mult}{\operatorname{mult}}
\newcommand{\aut}{\operatorname{Aut}}
\newcommand{\conj}{\operatorname{conj}}

\newcommand{\sym}{\operatorname{Sym}}
\newcommand{\symc}{\operatorname{Symc}}

\newcommand{\nsym}{\operatorname{Nsym}}
\newcommand{\val}{\operatorname{val}}
\newcommand{\vt}{\operatorname{Vert}}
\newcommand{\eg}{\operatorname{Edge}}

\newcommand{\zl}{\mathcal{Z}}
\newcommand{\ol}{\mathcal{O}}
\newcommand{\pl}{\mathcal{P}}

\newcommand{\rb}{\mathbb{R}}

\newcommand{~}{\quad}
\newcommand{\cb}{\mathbb{C}}
\newcommand{\nb}{\mathbb{N}}

\newcommand{\pb}{\mathbb{P}}

\newcommand{\undl}{\underline}

\newcommand{\w}{\omega}

\definecolor{cardinalred}{RGB}{140,21,21}
\definecolor{coolgray}{RGB}{77,79,83}
\definecolor{black}{RGB}{0,0,0}
\definecolor{beige}{RGB}{210,194,149}
\definecolor{darkbeige}{RGB}{179,153,93}
\definecolor{darkcardinal}{RGB}{94,48,50}
\definecolor{lightcardinal}{RGB}{141,60,30}
\definecolor{darkpurple}{RGB}{83,40,79}
\definecolor{darkcyan}{RGB}{0,124,146}
\definecolor{skyblue}{RGB}{0,152,219}
\definecolor{seablue}{RGB}{10,100,180}
\definecolor{darkblue}{RGB}{20,80,150}
\definecolor{treegreen}{RGB}{0,155,118}
\definecolor{darkorange}{RGB}{168,101,12}
\definecolor{beigegray}{RGB}{95,87,79}
\definecolor{boxgray}{RGB}{238,235,233}
\definecolor{footergray}{RGB}{199,209,197}

\begin{document}

\title[The uniform asymptotics]{The uniform asymptotics for real double Hurwitz numbers with triple ramification II: lower bounds and asymptotics}

\author{Yanqiao Ding}

\author{Kui Li}

\author{Huan Liu}

\author{Dongfeng Yan}

\address{School of Mathematics and Statistics, Zhengzhou University, Zhengzhou, 450001, China}

\email{yqding@zzu.edu.cn}


\email{likui@zzu.edu.cn}


\email{liuhuan@zzu.edu.cn}


\email{yandongfeng@zzu.edu.cn}



\date{\today}

\begin{abstract}
This is the second of two papers on the uniform asymptotics for real double Hurwitz numbers with triple ramification.
Using the modified tropical correspondence theorem established in the first paper of this series,
we introduce a combinatorial invariant that serves as a lower bound
for real double Hurwitz numbers with triple ramification. 
We derive a uniform lower bound for the large-degree and large-genus logarithmic asymptotics of
these combinatorial invariants.
This uniform lower bound yields the following results:
\begin{itemize}
\item[(1)] We establish a uniform lower bound for the large-degree and large-genus logarithmic asymptotics of 
real double Hurwitz numbers with triple ramification and their complex analogues.
In particular, we provide a partial answer to an open question proposed by Dubrovin, Yang and Zagier
on the uniform bound for simple Hurwitz numbers.
\item[(2)] We prove logarithmic equivalence between real double Hurwitz numbers with triple ramification
and their complex analogues as the degree tends to infinity and only simple branch points are added.
\item[(3)] As the genus tends to infinity and only simple branch points are added, 
we show that the logarithms of real double Hurwitz numbers with triple ramification and their complex analogues 
are of the same order.
\end{itemize}
\end{abstract}

\maketitle


\section{Introduction}
Hurwitz numbers count the number of ramified coverings of a Riemann surface 
by other Riemann surfaces, with prescribed ramification profiles over a given set of branch points.
These numbers have profound connections to combinatorics, the moduli space of algebraic curves, 
mathematical physics and so on \cite{okounkov-2000,elsv-2001, gjv-2005, op-2006,hurwitz-1891,bbm-2011, cjm-2010}.
Through a deep understanding of these interrelations, many algebraic structures of double Hurwitz numbers have been uncovered. 
For instance, the following results have been established: chamber structures and wall-crossing formulas 
for double Hurwitz numbers \cite{ssv-2008, cjm-2011};
relations linking double Hurwitz numbers to intersection numbers on moduli spaces \cite{dl-2022};
and the polynomiality structure of 
double Hurwitz numbers \cite{bdklm-2022}.

A series of remarkable works \cite{Aggarwal-2021,dgzz-2020,dgzz-2021,dgzz-2022,gy-2024}
has achieved substantial progress in understanding the large-genus asymptotics of 
intersection numbers over the moduli spaces of curves.
By contrast, the asymptotics of double Hurwitz numbers remain far from being fully understood.
In particular, the asymptotics of double Hurwitz numbers as both the degree and the genus tend to infinity, 
which are called {\it the uniform asymptotics}, have been rarely explored.

The asymptotics for simple Hurwitz numbers as the genus or degree tends to infinity 
separately have been established in \cite{hurwitz-1891,dyz-2017,yang-2025,cgmps-2007},
and the large-degree asymptotics of double Hurwitz numbers have been derived in \cite{li-2024}.
On page 616 of the paper \cite{dyz-2017}, Dubrovin, Yang and Zagier posed an open question
that is relevant to our results.
\begin{problem}[Dubrovin, Yang and Zagier]
\label{prob:1}
How can one find a uniform bound for simple Hurwitz numbers in terms of elementary functions?
\end{problem}

Budzinski and Louf \cite{bl-2021} obtained the uniform asymptotics,
up to subexponential factors, for the number of triangulations as both the size and the genus tend to infinity.
Tropical computations of Hurwitz numbers \cite{bbm-2011,cjm-2010,gpmr-2015,mr-2015}
provide a framework for studying the asymptotics of Hurwitz numbers.

The goal of this paper is to study the uniform asymptotics of real double Hurwitz numbers 
with triple ramification and their complex counterparts.
Real double Hurwitz numbers with triple ramification
enumerate the number of real ramified coverings
of $\cb\pb^1$ by real Riemann surfaces, where the coverings have arbitrary ramification types 
over $0$ and $\infty$, and triple or simple ramification over all other branch points.
These numbers depend on the distribution of the real branch points.
This phenomenon also arises in the
enumerative problem of counting real rational curves \cite{wel2005a,wel2005b,iks2014rw,iks2014,shustin2015}. 
Consequently,
it is crucial to find invariant lower bounds for the number of real solutions.
The fundamental question addressed in this study of real Hurwitz numbers is as follows. 
\begin{problem}
Does there exist an invariant that can yield precise estimates of real double Hurwitz numbers with triple ramification?
\end{problem}

Let $l(\lambda)$ denote the {\it length} of a partition $\lambda$, \textit{i.e.}, the number of entries in the partition $\lambda$.
The sum of all entries in the partition $\lambda$ is denoted by $|\lambda|$.
Given two partitions $\lambda$ and $\mu$ satisfying $|\lambda|=|\mu|$,
we choose two non-negative integers $s,t$ such that $2s+t=l(\lambda)+l(\mu)+2g-2$.
We use the notation $H^\rb_g(\lambda,\mu;\Lambda^-_{s,t},\Lambda^+_{s,t})$ to denote
the real double Hurwitz number with triple ramification defined in Definition \ref{def:RDH}.
Here, $\Lambda^-_{s,t}$ and $\Lambda^+_{s,t}$ denote the sequences of ramification profiles
over the negative and positive branch points, respectively.
In particular, $s$ is the number of triple branch points in $\rb\pb^1\setminus\{0,\infty\}$,
and $t$ is the number of simple branch points.
When $s=0$, the number $H^\rb_g(\lambda,\mu;\Lambda^-_{0,t},\Lambda^+_{0,t})$ is the ordinary
real double Hurwitz number.
In the case that $t=0$, the number $H^\rb_g(\lambda,\mu;\Lambda^-_{s,0},\Lambda^+_{s,0})$ is called the 
real double Hurwitz number with only triple ramification.

We denote by $|\Lambda^-_{s,t}|$ the number of partitions in the sequence $\Lambda^-_{s,t}$.
Note that the ordinary real double Hurwitz number $H^\rb_g(\lambda,\mu;\Lambda^-_{0,t},\Lambda^+_{0,t})$
depends only on the number $|\Lambda^-_{s,t}|$ \cite{mr-2015}.
However, the number $H^\rb_g(\lambda,\mu;\Lambda^-_{s,t},\Lambda^+_{s,t})$ 
depends not only on the numbers $|\Lambda^-_{s,t}|, |\Lambda^+_{s,t}|$,
but also on the arrangement of the ramification profiles in the sequences $\Lambda^-_{s,t},\Lambda^+_{s,t}$.
This constitutes the key difference between ordinary real double Hurwitz numbers
and their counterparts with triple ramification.

In \cite{mr-2015}, Markwig and Rau expressed 
real double Hurwitz numbers as a weighted count of coloured tropical covers, 
and distinguished positive and negative real branch points via two distinct colouring rules for even edges (see also \cite{gpmr-2015}).
Rau \cite{rau2019} introduced {\it zigzag covers}, in which even edges can be coloured according to arbitrary rules.
The number of zigzag covers provides an optimal lower bound for
$H^\rb_g(\lambda,\mu;\Lambda^-_{0,t},\Lambda^+_{0,t})$, 
and Rau \cite{rau2019} showed that the complex double Hurwitz numbers and their real counterparts are
logarithmically equivalent when the zigzag number is nonvanishing.
However, this framework is not enough for analyzing
real double Hurwitz numbers with triple ramification, for two reasons.
First, real double Hurwitz numbers with triple ramification 
depend on the arrangement of the branch points.
Second, according to Markwig and Rau's tropical computation \cite{mr-2015}, 
real double Hurwitz numbers with triple ramification
can be expressed as a weighted count of real tropical covers from real tropical curves containing 
$4$-valent vertices or genus one $2$-valent vertices.
The presence of these vertices complicates the characterization of the sign of a real branch point.
More exactly, we have to use two factors, the colouring of even edges and the edge weights, 
to distinguish between positive and negative real branch points.
This makes the analysis of uniform asymptotics for 
$H^\rb_g(\lambda,\mu;\Lambda^-_{s,t},\Lambda^+_{s,t})$ extremely challenging.

In the first paper \cite{dlly-23} of this series, 
we employ simple resolution to establish a modified tropical correspondence theorem.
In particular, we express the real double Hurwitz numbers with triple ramification
as a weighted count of coloured tropical covers, 
where positive and negative real branch points characterized by two colouring rules for even edges \cite{dlly-23}.
In this paper, we apply the modified tropical correspondence theorem \cite[Theorem 5.10]{dlly-23} to investigate
the lower bounds and asymptotics of real double Hurwitz numbers with triple ramification.
We first introduce a generalized version $Z_g(\lambda,\mu; \Lambda_{s,t})$ of Rau's zigzag number \cite{rau2019}, 
where $\Lambda_{s,t}=(\Lambda^-_{s,t},\Lambda^+_{s,t})$.
Note that $Z_g(\lambda,\mu; \Lambda_{s,t})$ depends on the distribution of triple and simple branch points.
We then derive a gluing formula (see Lemma \ref{lem:glue1}) for this combinatorial invariant 
to analyse its logarithmic asymptotics.
When a triple branch point passes through a simple branch point along the real line,
we analyse the wall-crossing phenomenon associated with generalized zigzag numbers and 
introduce {\it proper zigzag numbers} $Z_g(\lambda,\mu;s,t)$ (see Definition \ref{def:ref-zig-num}).
Finally, we prove that $Z_g(\lambda,\mu;s,t)$ is independent of 
the distribution of triple and simple branch points, and
serves as a lower bound for the generalized zigzag numbers $Z_g(\lambda,\mu; \Lambda_{s,t})$ (see Theorem \ref{thm:lower-bounds}).

Let $H^\cb_g(\lambda,\mu;s,t)$ denote the complex counterparts of $H^\rb_g(\lambda,\mu;\Lambda^-_{s,t},\Lambda^+_{s,t})$.
Under the assumption that $2s+t>0$ and $\{\lambda,\mu\}\not\subset\{(2k),(k,k)\}$, Proposition \ref{thm:low-bou1} yields that
$$
Z_g(\lambda,\mu;s,t)\leq H^\rb_g(\lambda,\mu;\Lambda^-_{s,t},\Lambda^+_{s,t})\leq H^\cb_g(\lambda,\mu;s,t).
$$
Thus, the proper zigzag number $Z_g(\lambda,\mu;s,t)$ is an invaraint lower bound for 
$H^\rb_g(\lambda,\mu;\Lambda^-_{s,t},\Lambda^+_{s,t})$.
In the following, we always assume that $2s+t>0$ and $\{\lambda,\mu\}\not\subset\{(2k),(k,k)\}$.
We define
$$
\begin{aligned}
z_{\lambda,\mu}(g,h,m)=Z_g((\lambda,1^{2c+h+2m-1}),(\mu,1^{2c+h+2m-1});s(h),t(m)),\\
h_{\lambda,\mu}^\cb(g,h,m)=H^\cb_g((\lambda,1^{2c+h+2m-1}),(\mu,1^{2c+h+2m-1});s(h),t(m)),
\end{aligned}
$$
where $s(h)=\frac{l(\lambda')+l(\mu')}{2}+c+h$, $t(m)=l(\lambda_e)+l(\mu_e)+2g-2+4m+2c$,
and $c$ is a positive constant depending on $\lambda$ and $\mu$.
Note that $h,m$ are used to indicate the number of triple branch points and simple branch points, respectively.
The notation $(1^m)$ stands for a partition consisting of $m$ ones.
Our first main result is given in the following theorem.

\begin{theorem}
\label{thm:main1}
Let $\lambda$ and $\mu$ be two partitions with $|\lambda|=|\mu|$.
Suppose that there exists an odd integer $o\neq1$ that appears at least twice in $\lambda$
and there exists an even integer $e\geq2o$ in $\mu$.
Then the uniform logarithmic asymptotics for proper zigzag numbers $z_{\lambda,\mu}(g,h,m)$ is bounded below by 
$\frac{h}{3}\ln h+2g\ln m+4m\ln m$ as $m,h,g\to\infty$, that is,
$$
\liminf_{g,m,h\to\infty}\frac{\ln z_{\lambda,\mu}(g,h,m)}{\frac{h}{3}\ln h+2g\ln m+4m\ln m}
\geq1.
$$
\end{theorem}

To consider the uniform asymptotics for real double Hurwitz numbers, we put
$$
\begin{aligned}
z_{\lambda,\mu}(g,m)&=Z_g((\lambda,1^{2m}),(\mu,1^{2m}); \Lambda_{0,t}),\\
z_{g,\lambda,\mu}(h)&=Z_g((\lambda,(1^h)),(\mu,(1^h));\Lambda_{s,0}),\\
h^\cb_{\lambda,\mu}(g,m)&=H_g^\cb((\lambda,1^{2m}),(\mu,1^{2m})),
\end{aligned}
$$
where $t=l(\lambda)+l(\mu)+2g-2+4m$, and $2s=l(\lambda)+l(\mu)+2g-2+2h$. 
Note that $Z_g((\lambda,1^{2m}),(\mu,1^{2m}); \Lambda_{0,t})$
is a variation of Rau's zigzag numbers (see Remark \ref{rmk:gzc-vs-zc}).
The following theorem is our second main result on uniform aysmptotics.

\begin{theorem}
\label{thm:main2}
Let $\lambda,\mu$ be two partitions with $|\lambda|=|\mu|$.
If the sum of the number of odd integers that appear odd number of times in $\lambda$
and the number of odd integers that appear odd number of times in $\mu$ is at most $2$,
the uniform logarithmic asymptotics for zigzag numbers $z_{\lambda,\mu}(g,m)$ is at least $2g\ln m+4m\ln m$
as $m,g\to\infty$, that is,
$$
\liminf_{g,m\to\infty}\frac{\ln z_{\lambda,\mu}(g,m)}{2g\ln m+4m\ln m}
\geq1.
$$
\end{theorem}

Since proper zigzag numbers (resp. zigzag numbers) serve as the lower bounds for 
real double Hurwitz numbers with triple ramification
(resp. real double Hurwitz numbers), Theorem \ref{thm:main1} and 
Theorem \ref{thm:main2} imply 
the uniform asymptotics for their complex counterparts.
\begin{corollary}
\label{coro:main1}
Let $\lambda,\mu$ be two partitions with $|\lambda|=|\mu|$.
Then the following statements hold.
\begin{enumerate}
\item[$(1)$] If $\lambda,\mu$ satisfy the conditions in Theorem $\ref{thm:main1}$,
the uniform logarithmic asymptotics for $h^\cb_{\lambda,\mu}(g,h,m)$ is at least $\frac{h}{3}\ln h+2g\ln m+4m\ln m$
as $m,h,g\to\infty$.
\item[$(2)$] If $\lambda,\mu$ satisfy the conditions in Theorem $\ref{thm:main2}$,
the uniform logarithmic asymptotics for $h^\cb_{\lambda,\mu}(g,m)$ is at least $2g\ln m+4m\ln m$
as $m,g\to\infty$.  
\end{enumerate}
\end{corollary}
From Corollary \ref{coro:main1}, we obtain a uniform bound for logarithmic asymptotics of simple Hurwitz numbers.
This provides a partial answer to the question posed by Dubrovin, Yang and Zagier (see Problem \ref{prob:1}).
Another interesting consequence of Theorem \ref{thm:main1} is the logarithmic equivalence between
real double Hurwitz numbers with triple ramification and their complex analogues,
as the degree tends to infinity and only simple branch points are added.

\begin{theorem}
\label{coro:main2}
Let $\lambda,\mu$ be two partitions with $|\lambda|=|\mu|$.
Suppose that $\lambda,\mu$ satisfy the conditions in Theorem $\ref{thm:main1}$.
Then there exists an integer $c_0$ depending on $\lambda$ and $\mu$
such that if $h>c_0$, $g\geq0$ are fixed, we have
$$
\ln z_{\lambda,\mu}(g,h,m)\sim 4m\ln m\sim\ln h^\cb_{\lambda,\mu}(g,h,m), \text{ as } m\to\infty.
$$
\end{theorem}

As the genus goes to infinity and only simple branch points are added, the following theorem
shows that the logarithms of real double Hurwitz numbers (with triple ramification) and their complex counterparts
are of the same order.
\begin{theorem}
\label{coro:main3}
Let $\lambda,\mu$ be two partitions with $|\lambda|=|\mu|$.
Then the following statements hold.
\begin{enumerate}
\item[$(1)$] Suppose that $\lambda,\mu$ satisfy the conditions in Theorem $\ref{thm:main1}$.
Then there exists an integer $c_0$ depending on $\lambda,\mu$
such that if $h>c_0$, $m>1$ are fixed, we have $\ln z_{\lambda,\mu}(g,h,m)$ and
$\ln h^\cb_{\lambda,\mu}(g,h,m)$ are of the same order as $g\to\infty$:
$$
\liminf_{g\to\infty}\frac{\ln z_{\lambda,\mu}(g,h,m)}{\ln h^\cb_{\lambda,\mu}(g,h,m)}\geq
\frac{\ln(m-1)}{ \ln(\frac{(|\lambda|+2c+h+2m-1)(|\lambda|+2c+h+2m-2)}{2})}.
$$
\item[$(2)$] Suppose that $\lambda,\mu$ satisfy the conditions in Theorem $\ref{thm:main2}$.
Then there exists an integer $N_0$ depending on $\lambda,\mu$
such that if $m>N_0$ is fixed, we have $\ln z_{\lambda,\mu}(g,m)$ and
$\ln h^\cb_{\lambda,\mu}(g,m)$ are of the same order as $g\to\infty$:
$$
\liminf_{g\to\infty}\frac{\ln z_{\lambda,\mu}(g,m)}{\ln h^\cb_{\lambda,\mu}(g,m)}\geq
\frac{\ln (m-N_0)}{\ln(\frac{(|\lambda|+2m)(|\lambda|+2m-1)}{2})}.
$$
\end{enumerate}
\end{theorem}

%


From the logarithmic equivalence
between real invariants and their complex counterparts
\cite{iks2004,iks2007,iz-2018,rau2019,shustin2015,d-2020,dh-2022},
it is natural to expect that the real double Hurwitz numbers
$H_g^\rb(\lambda,\mu;\Lambda_{s,0}^-,\Lambda_{s,0}^+)$ with only triple ramification
are logarithmically equivalent to their complex
counterparts $H_g^\cb(\lambda,\mu;s,0)$.
However, the logarithmic asymptotics
for $H_g^\cb(\lambda,\mu;s,0)$ remain unestablished when the degree tends to infinity and
only triple branch points are added.
For this reason, we cannot confirm whether our logarithmic growth estimate 
for real double Hurwitz numbers
with only triple ramification is optimal.
Nevertheless, the logarithmic growth rate of $2m\ln m$ \cite{dyz-2017,rau2019,d-2020}
for complex (and real) double Hurwitz numbers
$H^\cb_g((\lambda,1^m),(\mu,1^m); 0,t)$ as $m\to\infty$ and only simple branch points
are added
yields a strictly larger upper bound for the logarithmic growth
of real double Hurwitz numbers with only triple ramification.
It is very interesting to find the logarithmic asymptotics
for $H_g^\cb(\lambda,\mu;s,0)$ as the degree tends to infinity and
only triple branch points are added in the future.

\begin{problem}
What is the large-degree asymptotics for complex double Hurwitz numbers
with only triple ramification?
\end{problem}

In Section \ref{sec:2}, we review the definition of real double Hurwitz numbers with triple ramification
and the modified tropical correspondence theorem.
Section \ref{sec:4} presents the definition and properties of generalized zigzag numbers.
We establish a lower bound for the logarithmic asymptotics of real double Hurwitz numbers with only
triple ramification in Section \ref{sec:5}.
In Section \ref{sec:6}, we introduce the proper zigzag number and prove that it serves as a lower bound
for real double Hurwitz numbers with triple ramification.
The uniform asymptotics for real double Hurwitz numbers with triple ramification are derived in the section \ref{sec:7}.
The proofs of two uniform limits are provided in Appendix \ref{sec:a}.

\section{Preliminaries on the modified tropical correspondence theorem}
\label{sec:2}
In this section, we recall the definition of real double Hurwitz numbers with triple ramification and 
the modified tropical correspondence theorem \cite[Theorem 5.11]{dlly-23}.

\subsection{Real double Hurwitz numbers with triple ramification}

Fix integers $d\geq1$, $g\geq0$, and let $\lambda$, $\mu$ denote two partitions of $d$.
Choose non-negative integers $s,t$ satisfying $2s+t=l(\lambda)+l(\mu)+2g-2$,
an equality determined by the Riemann-Hurwitz formula. 
Let $\undl b=\{b_1,\ldots,b_{s+t}\}\subset\rb\pb^1\setminus\{0,\infty\}$ be a set
with elements ordered such that $b_1<b_2<\cdots<b_{s+t}$.
A map $r:\undl b\to\{2,3\}$ is called {\it an $(s,t)$-ramification function} on $\undl b$
if $|r^{-1}(3)|=s$ and $|r^{-1}(2)|=t$. A set $\undl b$ together with a ramification function $r$
on it are called an {\it $(s,t)$-branch set}.

\begin{definition}
\label{def:CDH}
The {\it complex double Hurwitz number with triple ramification} is defined as
$$
H^\cb_g(\lambda,\mu; s,t)=\sum_{[\pi]}
\frac{1}{|\aut^\cb(\pi)|},
$$
where the sum is taken over all isomorphism classes $[\pi]$ of
ramified coverings $\pi:C\to\cb\pb^1$ that satisfy the following conditions:
\begin{enumerate}
    \item $C$ is a connected Riemann surface of genus $g$;
    \item $\pi:C\to\cb\pb^1$ is a holomorphic map of degree $d$;
    \item $\pi$ has ramification profiles $\lambda$ and $\mu$ over $0$ and $\infty$, respectively;
    \item Let $\undl b$ be an $(s,t)$-branch set. Over points in $r^{-1}(3)$, $\pi$ has ramification profile $(3,1,\ldots,1)$, 
    while over points in $r^{-1}(2)$, it has simple ramification;
    \item $\pi$ is unramified at all other points.
\end{enumerate}
Here, an isomorphism of two such coverings $\pi_1:C_1\to\cb\pb^1$ and $\pi_2:C_2\to\cb\pb^1$
is an isomorphism $\varphi:C_1\to C_2$ of Riemann surfaces
satisfying $\pi_1=\pi_2\circ\varphi$.
The group $\aut^{\cb}(\pi)$ denotes the automorphism group of the ramified covering $\pi$.
\end{definition}

It is well-known that the number $H^\cb_g(\lambda,\mu; s,t)$ only depends on the partitions
$\lambda$, $\mu$ and the integers $s$, $t$ \cite{cm-2016,mr-2015}.
When $s=0$, the number $H^\cb_g(\lambda,\mu; 0,t)$ is the
ordinary complex double Hurwitz number, and we use the abbreviation
$H^\cb_g(\lambda,\mu; t)$ to denote $H^\cb_g(\lambda,\mu; 0,t)$.

An \textit{$(s,t)$-tuple} $\Lambda_{s,t}$ is a collection $(\Lambda_1,\ldots,\Lambda_{s+t})$
of partitions of $d$, where $s$ entries are $(3,1,\ldots,1)$ 
and the remaining $t$ are $(2,1,\ldots,1)$.
Let $\undl b$ be an $(s,t)$-branch set 
and $\Lambda_{s,t}$ an $(s,t)$-tuple.
$\Lambda_{s,t}$ is \textit{compatible} with $\undl b$,
if
\begin{itemize}
    \item $\Lambda_i=(3,1,\ldots,1)$ iff $r(b_i)=3$, and $\Lambda_j=(2,1,\ldots,1)$ iff $r(b_j)=2$.
\end{itemize}
For a compatible pair  $\undl b$ and  $\Lambda_{s,t}$, split $\Lambda_{s,t}$ into 
$\Lambda^-_{s,t}=(\Lambda_1,\ldots,\Lambda_{|\undl b\cap\rb^-|})$
and $\Lambda^+_{s,t}=(\Lambda_{1+|\undl b\cap\rb^-|},\ldots,\Lambda_{s+t})$.
This splitting $\Lambda_{s,t}=(\Lambda^-_{s,t}, \Lambda^+_{s,t})$ is called the {\it compatible signed splitting}
of $\undl b$.


\begin{definition}
\label{def:RDH}
Let $\Lambda_{s,t}=(\Lambda^-_{s,t}, \Lambda^+_{s,t})$ be a compatible signed splitting
of the $(s,t)$-branch set $\undl b$.
The {\it real double Hurwitz number with triple ramification} is the number
$$
H^\rb_g(\lambda,\mu;\Lambda^-_{s,t},\Lambda^+_{s,t})=\sum_{[(\pi,\tau)]}
\frac{1}{|\aut^\rb(\pi,\tau)|},
$$
where the sum is taken over all isomorphism classes $[(\pi,\tau)]$ of
ramified coverings $\pi:C\to\cb\pb^1$ that satisfy Conditions $(1)-(5)$
in Definition \ref{def:CDH} and the following supplementary requirement.
\begin{itemize}
    \item $\tau:C\to C$ is an anti-holomorphic involution on $C$
    satisfying $\pi\circ\tau=\conj\circ\pi$, with $\conj$ denoting the standard complex conjugation;
\end{itemize}
Here, an isomorphism of two real coverings
$(\pi_1:C_1\to\cb\pb^1,\tau_1)$
and $(\pi_2:C_2\to\cb\pb^1,\tau_2)$
is a Riemann surface isomorphism $\varphi:C_1\to C_2$ 
such that $\pi_1=\pi_2\circ\varphi$ and $\varphi\circ\tau_1=\tau_2\circ\varphi$.
The group $\aut^{\rb}(\pi,\tau)$ denotes the automorphism group of  the real ramified covering $(\pi,\tau)$.
\end{definition}

Note that
when the two partitions $\lambda$, $\mu$ are fixed, $H^\rb_g(\lambda,\mu; \Lambda^-_{s,t},\Lambda^+_{s,t})$
depends on the two sequences $\Lambda^-_{s,t}$ and $\Lambda^+_{s,t}$ \cite[Section 2]{mr-2015}.
In the case that $s=0$ and $|\undl b\cap\rb^{+}|=r\leq t$,
we use $H^\rb_g(\lambda,\mu;r)$ to denote the ordinary real double Hurwitz number
$H^\rb_g(\lambda,\mu;\Lambda^-_{0,t},\Lambda^+_{0,t})$.

\subsection{The modified tropical correspondence theorem}

A {\it tropical curve} $C$ is defined as a connected metric graph containing only $1$-valent and $3$-valent vertices.
The $3$-valent vertices of $C$ are called {\it inner vertices} of $C$, while
the $1$-valent vertices of $C$ are called {\it leaves}.
Denote by $\vt(C)$ the set of inner vertices and leaves of $C$.
An {\it inner edge} $e$ of $C$ is an edge with finite length $\ell(e)\in\rb$, whereas
edges incident to leaves are called {\it ends} and have infinite length.
Let $\eg(C)$ denote the set of edges of $C$.
An isomorphism $\varPhi:C_1\to C_2$ between two tropical curves $C_1$, $C_2$ is
an isometric homeomorphism $\varPhi:C_1^\circ\to C_2^\circ$, 
here $C_1^\circ$ and $C_2^\circ$ are the sub-graphs of
$C_1$ and $C_2$ obtained by deleting all $1$-valent vertices, respectively.
For any vertex $v\in C$, let $\val(v)$ denote its valence. The {\it genus} of $C$ is 
the first Betti number $b_1(C)$ of $C$.
We consider $T\pb^1=\rb\cup\{\pm\infty\}$ as the \textit{tropical projective line}.

Fix two integers $d\geq1$, $g\geq0$. Let $\lambda$, $\mu$ be two partitions of $d$,
and let $s,t$ be two non-negative integers such that $2s+t=l(\lambda)+l(\mu)+2g-2$.

\begin{definition}
\label{def:ram-split}
Let $\undl x'=\{x'_1,\ldots,x'_{2s}\}\subset\rb$ be ordered such that $x'_1<\cdots<x'_{2s}$.
For $i=1,\ldots,s$, set $x_i=(x'_{2i-1},x'_{2i})$ and let $\undl x=\{x_1,\ldots,x_s\}$ be the set of these pairs.
Let $\undl y=\{y_1,\ldots,y_{t}\}\subset\rb\setminus\undl x'$ be an ordered set with $y_1<\ldots<y_{t}$.
The disjoint union $\undl z=\undl x\sqcup\undl y$ is called an {\it $(s,t)$-distribution}, if
$[x'_{2i-1},x'_{2i}]\cap\undl y=\emptyset$ for all $i=1,\ldots,s$.
We write $x_i<y_j$ to mean $x'_{2i-1}<x'_{2i}<y_j$, and assume 
$\undl z=\{z_1,\ldots,z_{s+t}\}$ is ordered with $z_1<\cdots<z_{s+t}$, where $z_i\in\undl x\sqcup\undl y$.
\end{definition}

\begin{definition}
\label{def:d-tro-cover}
Let $\undl z=\undl x\sqcup\undl y$ be an $(s,t)$-distribution.
A \textit{tropical cover} $\varphi:C\to T\pb^1$ of type $(g,\lambda,\mu,\undl z)$
is a continuous map from a genus $g$ tropical curve $C$ satisfying these conditions:
\begin{enumerate}[(1)]
    \item $\undl x'\sqcup\undl y$ is the images of inner vertices of $C$ under $\varphi$, 
    and these vertices are called {\it inner vertices} of $T\pb^1$.
    \item $\varphi$ maps all leaves of $C$ onto $\{\pm\infty\}$.
    \item $\varphi$ is piecewise integer affine linear,
    with the slope $\w(e)\in\nb_{>0}$ of $\varphi$ on an edge $e$ defined as the {\it weight} of $e$.
    \item $\varphi$ satisfies the balancing condition at every inner vertex of $C$.
    \item The entries of $\lambda$ (resp. $\mu$) are weights of ends mapped by $\varphi$ to edges incident to $-\infty$ (resp. $+\infty$).
\end{enumerate}
\end{definition}

For any edge $e'$ of $T\pb^1$,
the balancing condition implies that the sum
$$
\deg(\varphi):=\sum_{
    \substack{e\text{ edge of } C\\
    e'\subset\varphi(e)}}\w(e)
$$
is independent of $e'$, and it is called the \textit{degree} of $\varphi$.

In section \ref{sec:5}--\ref{sec:7}, we frequently use a combinatorial description \cite[Remark 5.3]{rau2019} 
of tropical covers to study uniform asymptotic behaviors of real double Hurwitz numbers with triple ramification.
For the convenience of readers, we recall \cite[Remark 5.3]{rau2019} in Remark \ref{rmk-combin1}.
\begin{remark}[{\cite[Remark 5.3]{rau2019}}]
\label{rmk-combin1}
Let $C$ be a genus $g$ graph with only $1$-valent and $3$-valent vertices.
Suppose that $C$ is equipped with an orientation $\ol$ such that there is no oriented loops in $C$.
Assume that every edge of $C$ is weighted with a positive integer such that the balancing condition holds.
The orientation induces a partial order $\pl$ on inner vertices of $C$.
Let $\undl z\subset \rb$ be a set of $r$ points which satisfy $z_1<\cdots<z_r$,
where $r$ is the number of inner vertices of $C$.
Then any total order $v_1,\ldots,v_r$ on inner vertices of $C$ extending the partial order $\pl$
gives a unique tropical cover $\varphi:C\to T\pb^1$ satisfying the following conditions:
\begin{itemize}
\item the restricted orientation on every edge agrees with the orientation of $T\pb^1$ (from $-\infty$ to $+\infty$) under $\varphi$;
\item the given weight of every edge agrees with the one induced by $\varphi$;
\item $\varphi(v_i)=z_i$ for $i=1,\ldots, r$.
\end{itemize}

Vice versa, let $\varphi:C\to T\pb^1$ be a tropical cover of type $(g,\lambda,\mu,\undl z)$.
The tropical map $\varphi$ and the tropical projective line $T\pb^1$ induce an orientation $\ol$ on the tropical curve $C$:
each edge $e$ of $C$ is oriented from $v_1(e)$ to $v_2(e)$, where $v_1(e)$ and $v_2(e)$ are endpoints of $e$
with $\varphi(v_1(e))<\varphi(v_2(e))$. Obviously, $C$ is oriented without oriented loops.
Each edge $e$ of $C$ is weighted by the weight $\omega(e)$.
The map $\varphi$ and $\undl z$ determine a unique total order $v_1,\ldots,v_r$ on inner vertices of $C$
such that $\varphi(v_i)=z_i$, where $i=1,\ldots,r$.
The total order $v_1,\ldots,v_r$ extends the partial order $\pl$, which is induced by the orientation $\ol$.
\end{remark}

A \textit{symmetric cycle} (resp. \textit{symmetric fork}) of a
tropical cover $\varphi:C\to T\pb^1$ is a pair of inner edges (resp. ends) of the same weight
and are adjacent to the same two vertices (resp. one vertex).
A symmetric fork is \textit{inward} (resp. \textit{outward}) if the ends of it 
can be oriented pointing from leaves to inner vertices (resp. from inner vertices to leaves). 
Let $\sym(\varphi)$ denote the set of symmetric cycles and symmetric forks of $\varphi:C\to T\pb^1$.

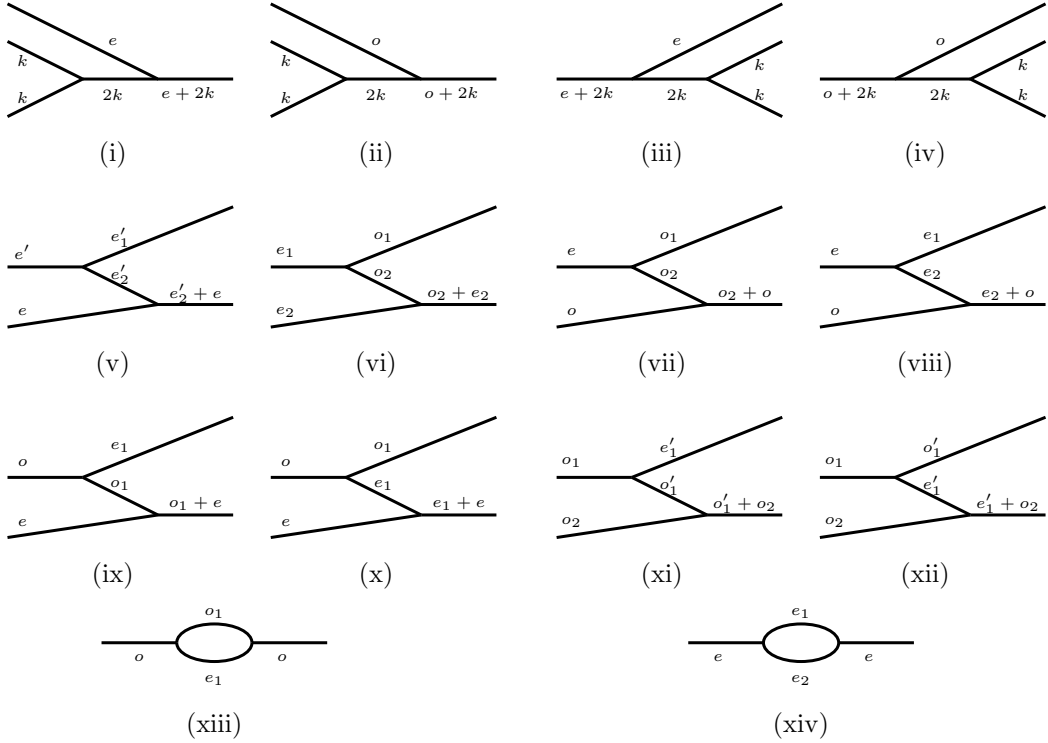
\begin{figure}[ht]
    \centering
    \begin{tikzpicture}
    \draw[line width=0.4mm] (0,-1)--(-1,-1)--(-3,0);
    \draw[line width=0.4mm] (-2,-1)--(-1,-1);
    \draw[line width=0.4mm] (-3,-1.5)--(-2,-1)--(-3,-0.5);
    \draw (-2.8,-0.75) node {\tiny $k$} (-2.8,-1.25) node {\tiny $k$} (-1.6,-1.2) node {\tiny $2k$} (-1.6,-0.5) node {\tiny $e$} (-0.6,-1.2) node {\tiny $e+2k$};
    \draw (-1.6,-2) node {(\textrm{i})};
    \draw[line width=0.4mm] (3.5,-1)--(2.5,-1)--(0.5,0);
    \draw[line width=0.4mm] (1.5,-1)--(2.5,-1);
    \draw[line width=0.4mm] (0.5,-1.5)--(1.5,-1)--(0.5,-0.5);
    \draw (0.7,-0.75) node {\tiny $k$} (0.7,-1.25) node {\tiny $k$} (1.9,-1.2) node {\tiny $2k$} (1.9,-0.5) node {\tiny $o$} (2.9,-1.2) node {\tiny $o+2k$};
    \draw (1.9,-2) node {(\textrm{ii}) };
    \draw[line width=0.4mm] (4.3,-1)--(5.3,-1)--(7.3,0);
    \draw[line width=0.4mm] (5.3,-1)--(6.3,-1);
    \draw[line width=0.4mm] (7.3,-1.5)--(6.3,-1)--(7.3,-0.5);
    \draw (7,-0.8) node {\tiny $k$} (7,-1.2) node {\tiny $k$} (5.9,-1.2) node {\tiny $2k$} (5.9,-0.5) node {\tiny $e$} (4.7,-1.2) node {\tiny $e+2k$};
    \draw (5.7,-2) node {(\textrm{iii})};
    \draw[line width=0.4mm] (7.8,-1)--(8.8,-1)--(10.8,0);
    \draw[line width=0.4mm] (8.8,-1)--(9.8,-1);
    \draw[line width=0.4mm] (10.8,-1.5)--(9.8,-1)--(10.8,-0.5);
    \draw (10.5,-0.8) node {\tiny $k$} (10.5,-1.2) node {\tiny $k$} (9.4,-1.2) node {\tiny $2k$} (9.4,-0.5) node {\tiny $o$} (8.2,-1.2) node {\tiny $o+2k$};
    \draw (9.2,-2) node {(\textrm{iv})};
    \draw[line width=0.4mm] (-3,-3.5)--(-2,-3.5)--(0,-2.7);
    \draw[line width=0.4mm] (-2,-3.5)--(-1,-4);
    \draw[line width=0.4mm] (-3,-4.3)--(-1,-4)--(0,-4);
    \draw (-2.8,-4.1) node{\tiny $e$} (-2.8,-3.3) node{\tiny $e'$}
    (-1.5,-3.1) node{\tiny $e_1'$} (-1.5,-3.6) node{\tiny $e_2'$} (-0.5,-3.85) node{\tiny $e_2'+e$}
    (-1.6,-4.8) node {(\textrm{v})};
    \draw[line width=0.4mm] (0.5,-3.5)--(1.5,-3.5);
    \draw[line width=0.4mm] (3.5,-2.7)--(1.5,-3.5)--(2.5,-4)--(3.5,-4);
    \draw[line width=0.4mm] (0.5,-4.3)--(2.5,-4);
    \draw (0.7,-4.1) node{\tiny $e_2$} (0.7,-3.3) node{\tiny $e_1$}
    (2,-3.1) node{\tiny $o_1$} (2,-3.6) node{\tiny $o_2$} (3,-3.85) node{\tiny $o_2+e_2$};
    \draw (1.9,-4.8) node {(\textrm{vi})};
    \draw[line width=0.4mm] (4.3,-3.5)--(5.3,-3.5);
    \draw[line width=0.4mm] (7.3,-2.7)--(5.3,-3.5)--(6.3,-4)--(4.3,-4.3);
    \draw[line width=0.4mm] (6.3,-4)--(7.3,-4);
    \draw (4.5,-4.1) node{\tiny $o$} (4.5,-3.3) node{\tiny $e$}
    (5.8,-3.1) node{\tiny $o_1$} (5.8,-3.6) node{\tiny $o_2$} (6.8,-3.85) node{\tiny $o_2+o$};
    \draw (5.7,-4.8) node {(\textrm{vii}) };
    \draw[line width=0.4mm] (7.8,-3.5)--(8.8,-3.5)--(10.8,-2.7);
    \draw[line width=0.4mm] (8.8,-3.5)--(9.8,-4);
    \draw[line width=0.4mm] (7.8,-4.3)--(9.8,-4)--(10.8,-4);
    \draw (8,-4.1) node{\tiny $o$} (8,-3.3) node{\tiny $e$}
    (9.3,-3.1) node{\tiny $e_1$} (9.3,-3.6) node{\tiny $e_2$} (10.3,-3.85) node{\tiny $e_2+o$};
    \draw (9.2,-4.8) node {(\textrm{viii})};
    \draw[line width=0.4mm] (-3,-6.3)--(-2,-6.3);
    \draw[line width=0.4mm] (-2,-6.3)--(-1,-6.8);
    \draw[line width=0.4mm] (-2,-6.3)--(0,-5.5);
    \draw[line width=0.4mm] (-3,-7.1)--(-1,-6.8);
    \draw[line width=0.4mm] (0,-6.8)--(-1,-6.8);
    \draw (-2.8,-6.9) node{\tiny $e$} (-2.8,-6.1) node{\tiny $o$}
    (-1.5,-5.9) node{\tiny $e_1$} (-1.5,-6.4) node{\tiny $o_1$} (-0.5,-6.65) node{\tiny $o_1+e$};
    \draw (-1.6,-7.6) node {(\textrm{ix})};
    \draw[line width=0.4mm] (0.5,-6.3)--(1.5,-6.3);
    \draw[line width=0.4mm] (1.5,-6.3)--(2.5,-6.8);
    \draw[line width=0.4mm] (1.5,-6.3)--(3.5,-5.5);
    \draw[line width=0.4mm] (0.5,-7.1)--(2.5,-6.8);
    \draw[line width=0.4mm] (3.5,-6.8)--(2.5,-6.8);
    \draw (0.7,-6.9) node{\tiny $e$} (0.7,-6.1) node{\tiny $o$}
    (2,-5.9) node{\tiny $o_1$} (2,-6.4) node{\tiny $e_1$} (3,-6.65) node{\tiny $e_1+e$};
    \draw (1.9,-7.6) node {(\textrm{x})};
    \draw[line width=0.4mm] (4.3,-6.3)--(5.3,-6.3);
    \draw[line width=0.4mm] (5.3,-6.3)--(6.3,-6.8);
    \draw[line width=0.4mm] (5.3,-6.3)--(7.3,-5.5);
    \draw[line width=0.4mm] (4.3,-7.1)--(6.3,-6.8);
    \draw[line width=0.4mm] (7.3,-6.8)--(6.3,-6.8);
    \draw (4.5,-6.9) node{\tiny $o_2$} (4.5,-6.1) node{\tiny $o_1$}
    (5.8,-5.9) node{\tiny $e_1'$} (5.8,-6.4) node{\tiny $o_1'$} (6.8,-6.65) node{\tiny $o_1'+o_2$};
    \draw (5.7,-7.6) node {(\textrm{xi})};
    \draw[line width=0.4mm] (7.8,-6.3)--(8.8,-6.3);
    \draw[line width=0.4mm] (8.8,-6.3)--(9.8,-6.8);
    \draw[line width=0.4mm] (8.8,-6.3)--(10.8,-5.5);
    \draw[line width=0.4mm] (7.8,-7.1)--(9.8,-6.8);
    \draw[line width=0.4mm] (10.8,-6.8)--(9.8,-6.8);
    \draw (8,-6.9) node{\tiny $o_2$} (8,-6.1) node{\tiny $o_1$}
    (9.3,-5.9) node{\tiny $o_1'$} (9.3,-6.4) node{\tiny $e_1'$} (10.3,-6.65) node{\tiny $e_1'+o_2$};
    \draw (9.2,-7.6) node {(\textrm{xii})};
    \draw[line width=0.4mm] (-0.75,-8.5)--(-1.75,-8.5);
    \draw[line width=0.4mm] (0.25,-8.5)--(1.25,-8.5);
    \draw[line width=0.4mm] (-0.25,-8.5) +(0:0.5 and 0.25) arc (0:180:0.5 and 0.25);
    \draw[line width=0.4mm] (-0.25,-8.5) +(0:0.5 and 0.25) arc (0:-180:0.5 and 0.25);
    \draw (-1.25,-8.7) node {\tiny $o$} (-0.25,-8.1) node {\tiny $o_1$} (-0.25,-9) node {\tiny $e_1$} (0.65,-8.7) node {\tiny $o$};
    \draw (-0.25,-9.6) node {(\textrm{xiii})};
    \draw[line width=0.4mm] (7.05,-8.5)--(6.05,-8.5);
    \draw[line width=0.4mm] (8.05,-8.5)--(9.05,-8.5);
    \draw[line width=0.4mm] (7.55,-8.5) ellipse (0.5 and 0.25);
    \draw (6.45,-8.7) node {\tiny $e$} (7.55,-8.1) node {\tiny $e_1$} (7.55,-9) node {\tiny $e_2$} (8.45,-8.7) node {\tiny $e$};
    \draw (7.55,-9.6) node {(\textrm{xiv})};
    \end{tikzpicture}
    \caption{Pairs of vertices of resolving tropical covers.}
    \label{fig:pair-EC}
\end{figure}
\begin{definition}
\label{def:resolving-cover}
A tropical cover $\varphi:C\to T\pb^1$ of type $(g,\lambda,\mu,\undl z)$
is a \textit{resolving tropical cover} if
it satisfies the following conditions:
\begin{enumerate}
    \item For each $i\in\{1,\ldots,s\}$, the vertex pair $x_i=(x_{2i-1}', x_{2i}')$
    is the image of two adjacent $3$-valent vertices in Figure \ref{fig:pair-EC}, up to reflection along a horizontal line.
    \item The pair of two flags with the same weight $k$ in the first row of Figure \ref{fig:pair-EC}
    is either a symmetric fork of $C$ or is contained in a symmetric cycle of $C$.
\end{enumerate}
\end{definition}

Let $\varphi$ be a resolving tropical cover of type $(g,\lambda,\mu,\undl z)$.
The edge connecting the two vertices in any figure of Figure \ref{fig:pair-EC}
is called the \textit{contractible edge} of that pair of vertices.
A symmetric cycle or fork is called {\it non-contractible} if it contains no contractible edges.
We use the following notation.
\begin{itemize}
    \item $\symc(\varphi)$: the set of symmetric cycles of $\varphi:C\to T\pb^1$.
    \item $E_c(\varphi)$: the set of even contractible edges in $C$.
    \item $\sym_3(\varphi)$ (resp. $\sym_2(\varphi)$): the set of non-contractible symmetric cycles/forks 
    adjacent to (resp. not adjacent to) a contractible edge.
    \item $\symc_c(\varphi)$: the set of contractible symmetric cycles.
    \item $\nsym_c(\varphi)$: the set of nonsymmetric
    cycles consisting of two even contractible edges.
\end{itemize}
Note that $\sym_3(\varphi)$ is the set of symmetric forks and symmetric cycles containing
the weight $k$ flag pair in the first row of Figure \ref{fig:pair-EC}.

\begin{definition}
\label{def:coloured-TC}
A \textit{coloured tropical cover} $(\varphi:C\to T\pb^1,\rho)$ consists of a tropical cover, 
a subset $I_\rho(\varphi)\subset \sym(\varphi)$, and an assignment of red or blue to each component
of the subgraph of even-weight edges in $C\setminus (I_\rho(\varphi))^\circ$.
\end{definition}

Let $(\varphi_1:C_1\to T\pb^1;\rho_1)$ and
$(\varphi_2:C_2\to T\pb^1;\rho_2)$ be coloured tropical covers.
A tropical curve isomorphism $\varPhi:C_1\to C_2$ that preserves colourings 
and satisfies $\varphi_2\circ\varPhi=\varphi_1$
is called an {\it isomorphism} of the coloured tropical covers
$(\varphi_1,\rho_1)$ and $(\varphi_2,\rho_2)$.
For a coloured tropical cover $(\varphi,\rho)$ of type
$(g,\lambda,\mu,\undl z)$, a {\it positive} (resp. {\it negative}) point is
the image of an inner vertex of $C$ shown in the left two (resp. right two) columns of Figure \ref{fig:coloured-vertices}, 
up to reflection along a vertical line.
\begin{figure}[ht]
    \centering
    \begin{tikzpicture}
    \draw[line width=0.3mm] (-3,0)--(-2,0)--(-1,0.5);
    \draw[line width=0.3mm,blue] (-2,0)--(-1,-0.5);
    \draw[line width=0.3mm,blue] (-0.5,0,0)--(0.5,0)--(1.5,0.5);
    \draw[line width=0.3mm,blue] (0.5,0)--(1.5,-0.5);
    \draw[line width=0.3mm,red] (-3,-1.3)--(-2,-1.3);
    \draw[line width=0.3mm] (-1,-0.8)--(-2,-1.3)--(-1,-1.8);
    \draw[line width=0.3mm,blue] (-0.5,-1.3)--(0.5,-1.3);
    \draw[line width=0.3mm,dotted] (1.5,-0.8)--(0.5,-1.3)--(1.5,-1.8);
    \draw[line width=0.3mm] (4,0)--(5,0)--(6,0.5);
    \draw[line width=0.3mm,red] (5,0)--(6,-0.5);
    \draw[line width=0.3mm,red] (6.5,0)--(7.5,0)--(8.5,0.5);
    \draw[line width=0.3mm,red] (7.5,0)--(8.5,-0.5);
    \draw[line width=0.3mm,blue] (4,-1.3)--(5,-1.3);
    \draw[line width=0.3mm] (6,-0.8)--(5,-1.3)--(6,-1.8);
    \draw[line width=0.3mm,red] (6.5,-1.3)--(7.5,-1.3);
    \draw[line width=0.3mm,dotted] (8.5,-0.8)--(7.5,-1.3)--(8.5,-1.8);
    \end{tikzpicture}
    \caption{Positive and negative vertices: even edges are coloured, odd edges are in black, and edges in $I_\rho$ are dotted.}
    \label{fig:coloured-vertices}
\end{figure}
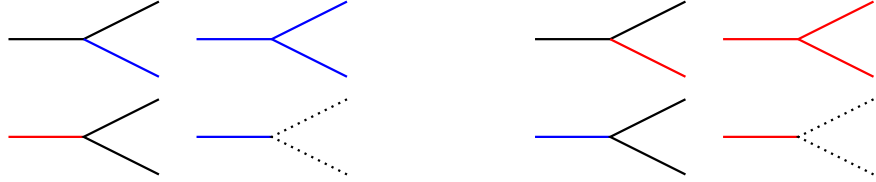

\begin{definition}
\label{def:E-real-TC1}
A colouring $\rho$ of a resolving tropical cover $\varphi:C\to T\pb^1$
is \textit{effective}, if the two vertices $x'_{2i-1},x'_{2i}$ in a pair $x_i$
have the same sign,
$\sym_3(\varphi)\subset I_\rho(\varphi)$ and $\symc_c(\varphi)\cap I_\rho(\varphi)=\emptyset$.
\end{definition}
From Definition \ref{def:E-real-TC1}, any pair of vertices $x'_{2i-1}, x'_{2i}$,
$i=1,\ldots,s$, are images of
two adjacent $3$-valent vertices of $C$ depicted in Figure \ref{fig:enhanced-vertices1},
up to reflection along a vertical line.
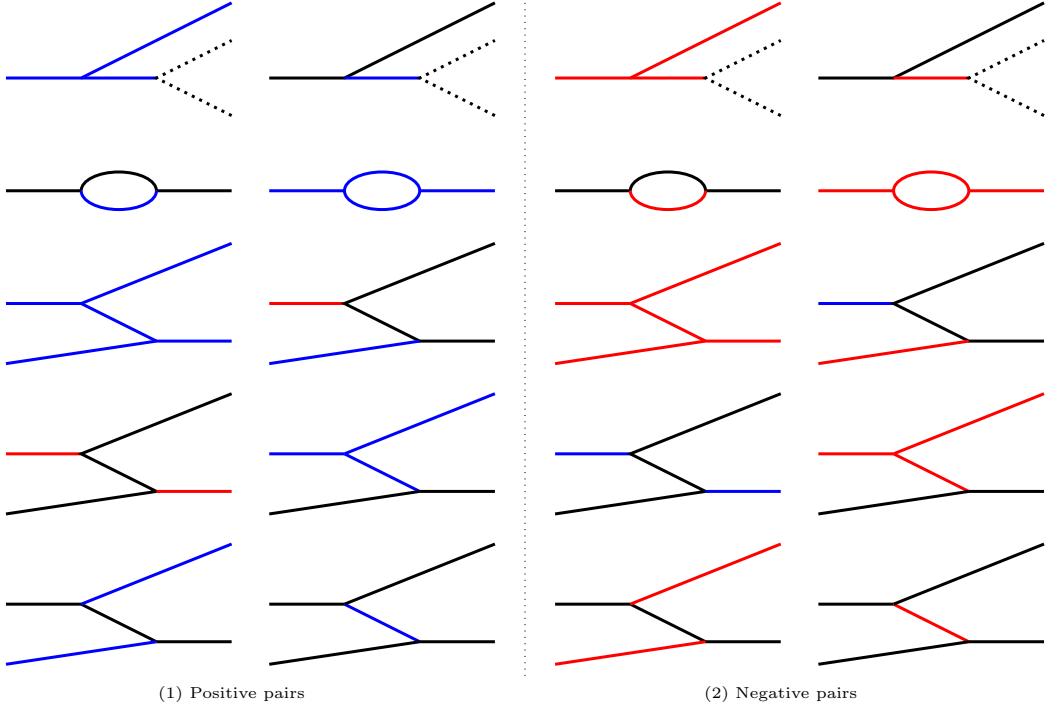
\begin{figure}[ht]
    \centering
    \begin{tikzpicture}
    \draw[line width=0.4mm,blue] (-3,-1)--(-2,-1)--(0,0);
    \draw[line width=0.4mm,blue] (-2,-1)--(-1,-1);
    \draw[line width=0.4mm,dotted] (0,-1.5)--(-1,-1)--(0,-0.5);
    \draw[line width=0.4mm] (0.5,-1)--(1.5,-1)--(3.5,0);
    \draw[line width=0.4mm,blue] (1.5,-1)--(2.5,-1);
    \draw[line width=0.4mm,dotted] (3.5,-1.5)--(2.5,-1)--(3.5,-0.5);
    \draw[line width=0.4mm] (-3,-2.5)--(-2,-2.5);
    \draw[line width=0.4mm] (-1,-2.5)--(0,-2.5);
    \draw[line width=0.4mm] (-1.5,-2.5) +(0:0.5 and 0.25) arc (0:180:0.5 and 0.25);
    \draw[line width=0.4mm,blue] (-1.5,-2.5) +(0:0.5 and 0.25) arc (0:-180:0.5 and 0.25);
    \draw[line width=0.4mm,blue] (0.5,-2.5)--(1.5,-2.5);
    \draw[line width=0.4mm,blue] (2.5,-2.5)--(3.5,-2.5);
    \draw[line width=0.4mm,blue] (2,-2.5) ellipse (0.5 and 0.25);
    \draw[line width=0.4mm,blue] (-3,-4)--(-2,-4)--(0,-3.2);
    \draw[line width=0.4mm,blue] (-2,-4)--(-1,-4.5);
    \draw[line width=0.4mm,blue] (-3,-4.8)--(-1,-4.5)--(0,-4.5);
    \draw[line width=0.4mm,red] (0.5,-4)--(1.5,-4);
    \draw[line width=0.4mm] (3.5,-3.2)--(1.5,-4)--(2.5,-4.5)--(3.5,-4.5);
    \draw[line width=0.4mm,blue] (0.5,-4.8)--(2.5,-4.5);
    \draw[line width=0.4mm,red] (-3,-6)--(-2,-6);
    \draw[line width=0.4mm] (0,-5.2)--(-2,-6)--(-1,-6.5)--(-3,-6.8);
    \draw[line width=0.4mm,red] (-1,-6.5)--(0,-6.5);
    \draw[line width=0.4mm,blue] (0.5,-6)--(1.5,-6);
    \draw[line width=0.4mm,blue] (1.5,-6)--(2.5,-6.5);
    \draw[line width=0.4mm,blue] (1.5,-6)--(3.5,-5.2);
    \draw[line width=0.4mm] (0.5,-6.8)--(2.5,-6.5);
    \draw[line width=0.4mm] (3.5,-6.5)--(2.5,-6.5);
    \draw[line width=0.4mm] (-3,-8)--(-2,-8);
    \draw[line width=0.4mm] (-2,-8)--(-1,-8.5);
    \draw[line width=0.4mm,blue] (-2,-8)--(0,-7.2);
    \draw[line width=0.4mm,blue] (-3,-8.8)--(-1,-8.5);
    \draw[line width=0.4mm] (0,-8.5)--(-1,-8.5);
    \draw[line width=0.4mm] (0.5,-8)--(1.5,-8);
    \draw[line width=0.4mm,blue] (1.5,-8)--(2.5,-8.5);
    \draw[line width=0.4mm] (1.5,-8)--(3.5,-7.2);
    \draw[line width=0.4mm] (0.5,-8.8)--(2.5,-8.5);
    \draw[line width=0.4mm] (3.5,-8.5)--(2.5,-8.5);
    \draw (0,-9.2) node{\tiny $(1)$ Positive pairs};
    \draw[dotted] (3.9,0)--(3.9,-9);
    \draw[line width=0.4mm,red] (4.3,-1)--(5.3,-1)--(7.3,0);
    \draw[line width=0.4mm,red] (5.3,-1)--(6.3,-1);
    \draw[line width=0.4mm,dotted] (7.3,-1.5)--(6.3,-1)--(7.3,-0.5);
    \draw[line width=0.4mm] (7.8,-1)--(8.8,-1)--(10.8,0);
    \draw[line width=0.4mm,red] (8.8,-1)--(9.8,-1);
    \draw[line width=0.4mm,dotted] (10.8,-1.5)--(9.8,-1)--(10.8,-0.5);
    \draw[line width=0.4mm] (4.3,-2.5)--(5.3,-2.5);
    \draw[line width=0.4mm] (6.3,-2.5)--(7.3,-2.5);
    \draw[line width=0.4mm] (5.8,-2.5) +(0:0.5 and 0.25) arc (0:180:0.5 and 0.25);
    \draw[line width=0.4mm,red] (5.8,-2.5) +(0:0.5 and 0.25) arc (0:-180:0.5 and 0.25);
    \draw[line width=0.4mm,red] (7.8,-2.5)--(8.8,-2.5);
    \draw[line width=0.4mm,red] (9.8,-2.5)--(10.8,-2.5);
    \draw[line width=0.4mm,red] (9.3,-2.5) ellipse (0.5 and 0.25);
    \draw[line width=0.4mm,red] (4.3,-4)--(5.3,-4)--(7.3,-3.2);
    \draw[line width=0.4mm,red] (5.3,-4)--(6.3,-4.5);
    \draw[line width=0.4mm,red] (4.3,-4.8)--(6.3,-4.5)--(7.3,-4.5);
    \draw[line width=0.4mm,blue] (7.8,-4)--(8.8,-4);
    \draw[line width=0.4mm] (10.8,-3.2)--(8.8,-4)--(9.8,-4.5)--(10.8,-4.5);
    \draw[line width=0.4mm,red] (7.8,-4.8)--(9.8,-4.5);
    \draw[line width=0.4mm,blue] (4.3,-6)--(5.3,-6);
    \draw[line width=0.4mm] (7.3,-5.2)--(5.3,-6)--(6.3,-6.5)--(4.3,-6.8);
    \draw[line width=0.4mm,blue] (6.3,-6.5)--(7.3,-6.5);
    \draw[line width=0.4mm,red] (7.8,-6)--(8.8,-6);
    \draw[line width=0.4mm,red] (8.8,-6)--(9.8,-6.5);
    \draw[line width=0.4mm,red] (8.8,-6)--(10.8,-5.2);
    \draw[line width=0.4mm] (7.8,-6.8)--(9.8,-6.5);
    \draw[line width=0.4mm] (10.8,-6.5)--(9.8,-6.5);
    \draw[line width=0.4mm] (4.3,-8)--(5.3,-8);
    \draw[line width=0.4mm] (5.3,-8)--(6.3,-8.5);
    \draw[line width=0.4mm,red] (5.3,-8)--(7.3,-7.2);
    \draw[line width=0.4mm,red] (4.3,-8.8)--(6.3,-8.5);
    \draw[line width=0.4mm] (7.3,-8.5)--(6.3,-8.5);
    \draw[line width=0.4mm] (7.8,-8)--(8.8,-8);
    \draw[line width=0.4mm,red] (8.8,-8)--(9.8,-8.5);
    \draw[line width=0.4mm] (8.8,-8)--(10.8,-7.2);
    \draw[line width=0.4mm] (7.8,-8.8)--(9.8,-8.5);
    \draw[line width=0.4mm] (10.8,-8.5)--(9.8,-8.5);
    \draw (7.3,-9.2) node{\tiny $(2)$ Negative pairs};
    \end{tikzpicture}
    \caption{Signed pairs: even edges are drawn in colours, odd edges in black. Dotted edges are
    the symmetric cycles or forks contained in $I_\rho$.}
    \label{fig:enhanced-vertices1}
\end{figure}
Let $\undl z^-$ (resp. $\undl z^+$) denote the set of negative (resp. positive) points or pairs of $(\varphi,\rho)$.
There is a natural signed splitting of $\undl z=\undl x\sqcup\undl y=\undl z^-\sqcup\undl z^+$
into positive and negative points or pairs.

The multiplicity of an effectively coloured
resolving tropical cover $(\varphi,\rho)$ is defined as
\begin{equation}
\label{eq:mult-1}
\mult^\rb(\varphi,\rho)=\frac{2^{|E(I_\rho)\setminus E_c(\varphi)|+|\symc(\varphi)\cap\sym_3(\varphi)|+|\nsym_c(\varphi)|}}{2^{|\sym_2(\varphi)|}}\prod_{e\in \symc(\varphi)\cap I_\rho}\omega(e),
\end{equation}
where $\omega(e)$ is the weight of an edge in the dotted symmetric cycle $e$,
and $E(I_\rho)$ is the set of inner edges with even weights in $C\setminus I_\rho^\circ$.

\begin{remark}
\label{rem:resolving-mult}
From equation (\ref{eq:mult-1}), the multiplicity of an effectively coloured
resolving tropical cover $(\varphi,\rho)$ only depends on the set $I_\rho$,
not on the colouring of even inner edges.
\end{remark}


\begin{definition}
\label{def:compatible-signed}
Let $\Lambda_{s,t}=(\Lambda_1,\ldots,\Lambda_{s+t})$ be an $(s,t)$-tuple with signed splitting 
$\Lambda_{s,t}=(\Lambda_{s,t}^-,\Lambda_{s,t}^+)$.
Let $\undl z=\undl x\sqcup\undl y$ is an $(s,t)$-distribution, 
where $\undl x$ is a set of $s$ pairs, and $\undl y$ is a set of $t$ points.
Let $\undl z^-=\{z^-_1,\ldots,z^-_{|\undl z^-|}\}$ and $\undl z^+=\{z^+_1,\ldots,z^+_{|\undl z^+|}\}$
be a signed splitting of $\undl z$.
The splitting $\undl z=\undl z^-\sqcup\undl z^+$
is \textit{compatible} with the splitting $(\Lambda_{s,t}^-,\Lambda_{s,t}^+)$ of $\Lambda_{s,t}$,
if $|\undl z^-|=|\Lambda^-_{s,t}|$ and the following holds.
\begin{itemize}
    \item $\Lambda_i^*=(3,1,\ldots,1)$ iff $z_i^*\in\undl x$,
     and $\Lambda_j^*=(2,1,\ldots,1)$ iff $z_j^*\in\undl y$, where $*=\pm$.
\end{itemize}
\end{definition}


\begin{theorem}[{\cite[Theorem 5.11]{dlly-23}}]
\label{thm:mr}
Let $g\geq0$, $d\geq1$, $s\geq0$ and $t\geq0$ be four integers,
and suppose that $\lambda$, $\mu$ are two partitions of $d$
such that $2s+t=l(\lambda)+l(\mu)+2g-2$.
Let $\tilde{\undl z}=\undl x\sqcup\undl y$ be an $(s,t)$-distribution with a splitting
$\tilde{\undl z}^-\sqcup\tilde{\undl z}^+$ compatible with a signed splitting $(\Lambda^-_{s,t}, \Lambda^+_{s,t})$
of an $(s,t)$-tuple $\Lambda_{s,t}$.
Then, we have
\begin{equation}\label{eq:mr-resolving}
H^\rb_g(\lambda,\mu;\Lambda^-_{s,t}, \Lambda^+_{s,t})=\sum_{[(\varphi,\rho)]}\mult^\rb(\varphi,\rho),
\end{equation}
where the sum is taken over isomorphism classes $[(\varphi,\rho)]$
of effectively coloured resolving tropical covers of type $(g,\lambda,\mu,\tilde{\undl z})$,
whose positive and negative points or pairs of points reproduce
the splitting $\tilde{\undl z}^+\sqcup\tilde{\undl z}^-$.
\end{theorem}

\section{Generalized zigzag numbers}
\label{sec:4}
In this section we introduce the generalized zigzag number that is a generalization of the zigzag number proposed by Rau in \cite{rau2019}, 
then we discuss the properties of this generalization.

Let $g\geq0$, $d\geq1$ be two integers,
and suppose that $\lambda$, $\mu$ are two partitions of $d$.
Let $s,t$ be two non-negative integers such that $2s+t=l(\lambda)+l(\mu)+2g-2$.
Let $\undl z=\undl x\sqcup\undl y$ be an $(s,t)$-distribution.
We assume $2s+t>0$ and $\{\lambda,\mu\}\not\subset\{(2k),(k,k)\}$ from now on (see \cite[Remark 2.5]{rau2019}).

\begin{figure}[ht]
    \centering
    \begin{tikzpicture}
    \draw[line width=0.3mm] (-3,0)--(-1.5,0);
    \draw[line width=0.3mm] (-0.3,0) arc[start angle=0, end angle=360, x radius=0.6, y radius=0.4];
    \draw[line width=0.3mm] (-0.3,0)--(1,0);
    \draw[line width=0.3mm] (2.2,0) arc[start angle=0, end angle=360, x radius=0.6, y radius=0.4];
    \draw[line width=0.3mm] (2.2,0)--(3.7,0);
    \draw[line width=0.3mm,gray,dotted]  (3.7,0.3)--(3.7,-0.3) -- (4,-0.3)--(4,0.3)--(3.7,0.3);
    \draw[line width=0.3mm,gray] (4.2,0) node{\tiny$S$};
    \draw (-3.3,0) node{\tiny $2o$} (-0.3,0.3) node{\tiny$o$} (2.2,0.3) node{\tiny$o$} (-0.3,-0.3) node{\tiny$o$} (2.2,-0.3) node{\tiny$o$};
    \draw[line width=0.3mm]  (-3,1.8)--(-2.5, 1.5) -- (-3,1.2);
    \draw[line width=0.3mm] (-2.5,1.5)--(-1.5,1.5);
    \draw[line width=0.3mm] (-0.3,1.5) arc[start angle=0, end angle=360, x radius=0.6, y radius=0.4];
    \draw[line width=0.3mm] (-0.3,1.5)--(1,1.5);
    \draw[line width=0.3mm] (2.2,1.5) arc[start angle=0, end angle=360, x radius=0.6, y radius=0.4];
    \draw[line width=0.3mm] (2.2,1.5)--(3.7,1.5);
    \draw[line width=0.3mm,gray,dotted]  (3.7,1.8)--(3.7,1.2) -- (4,1.2)--(4,1.8)--(3.7,1.8);
    \draw[line width=0.3mm] (-3.3,1.8) node{\tiny$o$} (-0.3,1.8) node{\tiny$o$} (2.2,1.8) node{\tiny$o$};
    \draw[line width=0.3mm] (-3.3,1.2) node{\tiny$o$} (-0.3,1.2) node{\tiny$o$} (2.2,1.2) node{\tiny$o$};
    \draw[line width=0.3mm,gray] (4.2,1.5) node{\tiny$S$};
    \draw[line width=0.3mm]  (-3,-1.2)--(-2.5,-1.5)--(3.7, -1.5);
    \draw[line width=0.3mm]  (-3,-1.8)--(-2.5,-1.5);
    \draw[line width=0.3mm,gray,dotted]  (3.7,-1.8)--(3.7,-1.2) -- (4,-1.2)--(4,-1.8)--(3.7,-1.8);
   \draw (-3.3,-1.2) node{\tiny$e$}(-3.3,-1.8) node{\tiny$e$};
    \draw[line width=0.3mm,gray](4.2,-1.5) node{\tiny$S$};
    \draw[line width=0.3mm]  (-3,-3)--(3.7, -3);
    \draw[line width=0.3mm,gray,dotted]  (3.7,-3.3)--(3.7,-2.7) -- (4,-2.7)--(4,-3.3)--(3.7,-3.3);
    \draw[line width=0.3mm,gray](4.2,-3) node{\tiny$S$};
   \draw (-3.3,-3) node{\tiny$2e$};
    \end{tikzpicture}
    \caption{\small Tails for generalized zigzag covers. The number of symmetric cycles can be arbitrary.}
    \label{fig:gzc-tail}
\end{figure}
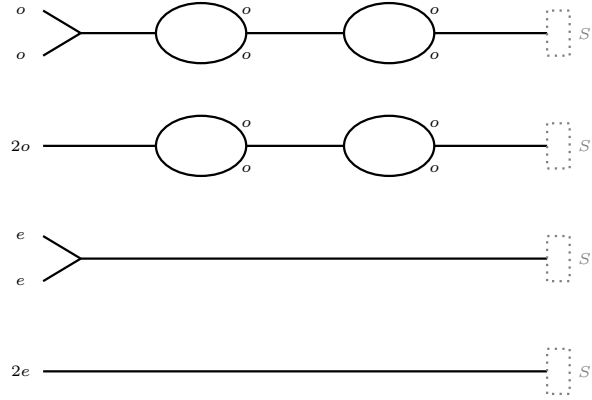
\begin{definition}
\label{def:gen-zig-cov}
A resolving tropical cover $\varphi:C\to T\pb^1$ of type $(g,\lambda,\mu,\undl x\sqcup\undl y)$
is a \textit{generalized zigzag cover}, if there exists a non-empty connected subgraph $S\subset C\setminus\sym_2(\varphi)$
such that:
\begin{enumerate}[$(1)$]
    \item all contractible edges of $C$ are contained in $S$ and $S$ contains at least one inner vertex;
    \item except for even contractible edges, there is no even edge in $S$;
    \item the connected components of $C\setminus S$ are tails depicted in Figure \ref{fig:gzc-tail}.
Moreover, the cycles and forks in Figure \ref{fig:gzc-tail} are all symmetric, and the symmetric cycles are of odd weight.
\end{enumerate}
\end{definition}

A tail in Figure \ref{fig:gzc-tail} of a generalized zigzag cover $\varphi$ is called an \textit{inward tail} (resp. \textit{outward tail}),
if the tail can be oriented pointing from leaves to inner vertices (resp. from inner vertices to leaves) by $\varphi$.
\begin{remark}
\label{rmk:gzc-vs-zc}
Let $s=0$ ({\it i.e.} $\undl x=\emptyset$),  
and $\varphi:C\to T\pb^1$ be a generalized zigzag cover of type $(g,\lambda,\mu,\undl x\sqcup\undl y)$.
Then $\varphi:C\to T\pb^1$ is a variation of the zigzag cover introduced by Rau in \cite[Definition 4.4]{rau2019}. 
Note that in \cite[Definition 4.4]{rau2019}, there is only symmetric fork of odd weight in a zigzag cover.
\end{remark}

\begin{proposition}
\label{prop:colour-GZC}
Let $\varphi:C\to T\pb^1$ be a generalized zigzag cover of type
$(g,\lambda,\mu,\undl z)$.
For any signed splitting $\undl z=\undl z^+\sqcup\undl z^-$ of the
$(s,t)$-distribution $\undl z=\undl x\sqcup\undl y$,
there is at least one effective colouring $\rho$ of the resolving tropical cover $\varphi$
such that $(\varphi,\rho)$ is an effectively coloured resolving tropical cover whose
positive and negative points or pairs of points reproduce the
splitting $\undl z=\undl z^+\sqcup\undl z^-$.
\end{proposition}

\begin{proof}
Let $\undl z=\undl z^+\sqcup\undl z^-$ be a signed splitting of the
$(s,t)$-distribution $\undl z=\undl x\sqcup\undl y$.
Then the sign of any pair of vertices $x_i=(x_{2i-1}',x_{2i}')$ is fixed.
The way to colour the edges adjacent to $\varphi^{-1}(x_i)$ is determined by
the local picture of these edges in Figure \ref{fig:enhanced-vertices1}.
Because of the conditions (1) and (2) in Definition \ref{def:gen-zig-cov}, all edges adjacent to $\varphi^{-1}(\undl x)$ are coloured according to Figure \ref{fig:enhanced-vertices1}.
The way to extend the colour of edges to all edges in $S$
is described in the proof of \cite[Proposition 4.8]{rau2019}.
Any vertex $\varphi^{-1}(y_i)$, where $y_i\in\undl y$, is adjacent to an even edge
which is not contractible.
From Definition \ref{def:gen-zig-cov}(2), if the vertex $\varphi^{-1}(y_i)\in S$, it must
be attached with a tail depicted in Figure \ref{fig:gzc-tail}.
When signs of vertices in $\varphi^{-1}(\undl y)$ are fixed,
the way to extend the colour of edges to a tail depicted in Figure \ref{fig:gzc-tail}
is described in the proof of \cite[Proposition 4.8]{rau2019}.
Moreover, if all symmetric forks in tails depicted in Figure \ref{fig:gzc-tail} are of odd weights,
the extension is unique.
\end{proof}

\begin{proposition}
\label{prop:mult-gen-zig-cov}
Let $\varphi:C\to T\pb^1$ be a generalized zigzag cover of type
$(g,\lambda,\mu,\undl z)$,
and $\rho$ be an effective colouring of $\varphi$.
Then the statements in the below are true.
\begin{enumerate}
\item[$(a)$] The multiplicity of $(\varphi,\rho)$ is
\begin{equation}\label{eq:mult-gzc}
\mult^\rb(\varphi,\rho)=2^{|\symc(\varphi)\cap\sym_3(\varphi)|}\prod_{e\in \symc(\varphi)\cap I_\rho}\omega(e).
\end{equation}
\item[$(b)$] If the multiplicity $\mult^\rb(\varphi,\rho)$ is odd,
the multiplicity $\mult^\rb(\varphi,\rho')$ is odd for any effective colouring $\rho'$ of $\varphi$.
\end{enumerate}
\end{proposition}

\begin{proof}
\begin{enumerate}[$(a)$]
\item By Definition \ref{def:gen-zig-cov},  every even inner edge in $E(I_\rho)\setminus E_c(\varphi)$
is contained in a tail depicted in Figure \ref{fig:gzc-tail},
so the number $|E(I_\rho)\setminus E_c(\varphi)|$ is equal to the number $|\sym_2(\varphi)|$.
We show that the set $\nsym_c(\varphi)=\emptyset$ in the below.
If $s=|\undl x|=0$, the set $\nsym_c(\varphi)=\emptyset$ by definition.
In the case $s=|\undl x|=1$, if the set $\nsym_c(\varphi)\neq\emptyset$,
the two even edges adjacent to the contractible non-symmetric cycle are not in $S$.
The two even deges must be contained in tails depicted in Figure \ref{fig:gzc-tail},
so we have $\{\lambda,\mu\}\subset\{(2k),(k,k)\}$.
This contradicts to our assumption $2s+t>0$ and $\{\lambda,\mu\}\not\subset\{(2k),(k,k)\}$.
Hence, in the case $s=|\undl x|=1$ the set $\nsym_c(\varphi)=\emptyset$.
Now we consider the case that $s=|\undl x|>1$.
If the set $\nsym_c(\varphi)\neq\emptyset$, any cycle $Cy\in\nsym_c(\varphi)$
is contained in $S$. Since the subgraph $S$ is connected, at least one of the two edges
adjacent to $Cy$ is an inner edge contained in $S$. This is a contradiction.
Therefore, $\nsym_c(\varphi)=\emptyset$.
By equation (\ref{eq:mult-1}), the multiplicity of $(\varphi,\rho)$ is given in equation (\ref{eq:mult-gzc}).
\item If $\mult^\rb(\varphi,\rho)$ is odd, we obtain
$\symc(\varphi)\cap\sym_3(\varphi)=\emptyset$ and 
the weight of every symmetric cycle in $\symc(\varphi)\cap I_{\rho}$ is odd.
Note that the set $\symc(\varphi)\cap\sym_3(\varphi)$ does not depend on the colouring $\rho$.
Let $\rho'$ be any effective colouring of $\varphi$.
The multiplicity $\mult^\rb(\varphi,\rho')$ only depends on the set $I_{\rho'}$.
Since $\symc(\varphi)\cap\sym_3(\varphi)=\emptyset$,
symmetric cycles in $\symc(\varphi)\cap I_{\rho'}$ can only be chosen from $\sym_2(\varphi)$.
By Definition \ref{def:gen-zig-cov}(3) and the statement in $(a)$, the multiplicity $\mult^\rb(\varphi,\rho')$ is odd.
\end{enumerate}
\end{proof}

\begin{proposition}\label{prop:odd-mult}
Let $(\varphi,\rho)$ be an effectively coloured resolving tropical cover. Then $\mult^\rb(\varphi,\rho)$ is an integer.
Moreover, if $\mult^\rb(\varphi,\rho)$ is odd,
the resolving tropical cover $\varphi$ has to be a generalized zigzag cover.
\end{proposition}

\begin{proof}
We use the idea of the proof of \cite[Lemma 4.2]{rau2019} to prove this proposition.
Let $(\varphi:C\to T\pb^1,\rho)$ be an effectively coloured resolving tropical cover.
Replace all symmetric forks and symmetric cycles in the set $\sym_2(\varphi)$ with even edges of appropriate weights,
then we obtain a new tropical curve $C_1$.
Since $2s+t>0$ and $\{\lambda,\mu\}\not\subset\{(2k),(k,k)\}$,
the tropical curve $C_1$ has at least one inner vertex. The number of even inner edges in $C_1$ is non-negative.
We can recover $C$ from $C_1$ by inserting symmetric cycles/forks.
When a symmetric cycle or symmetric fork is inserted in $C_1$, at least one new even inner edge is created,
so the number $|\sym_2(\varphi)|\leq |E(I_{\rho})\setminus E_c(\varphi)^\circ|$.
Hence, the multiplicity $\mult^\rb(\varphi,\rho)$ is an integer.
An even inner edge in $C\setminus(E(I_{\rho})\setminus E_c(\varphi)^\circ)$ has to be
an even contractible edge or an even edge in a symmtric cycle in $I_{\rho}$.
The multiplicity of $(\varphi,\rho)$ is odd if and only if
$|\sym_2(\varphi)|=|E(I_{\rho})\setminus E_c(\varphi)^\circ|$,  $|\symc(\varphi)\cap\sym_3(\varphi)|+|\nsym(\varphi)|=0$
and the weight of every symmetric cycle in $\symc(\varphi)\cap I_{\rho}$ is odd.
Now suppose that $\mult^\rb(\varphi,\rho)$ is odd.
If the set $E(I_{\rho})\setminus E_c(\varphi)^\circ$ contains an even cycle ( {\it i.e.} two even edges with the same two endpoints),
we have $|\sym_2(\varphi)|< |E(I_{\rho})\setminus E_c(\varphi)^\circ|$.
This is a contradiction, so in the tropical curve $C$ there is no even cycle in $E(I_{\rho})\setminus E_c(\varphi)^\circ$.
Since the tropical curve $C$ has no even symmetric cycle in $I_{\rho}$ and
$|\sym_2(\varphi)|=|E(I_{\rho})\setminus E_c(\varphi)^\circ|$,
the tropical curve $C_1$ obtained in the above has no even inner edges except even contractible inner edges.
Hence, all elements in $\sym_2(\varphi)$ are contained in tails of the types depicted in Figure \ref{fig:gzc-tail}.
Let $S$ be the subgraph of $C_1$ which is obtained by removing all even ends of $C_1$, then $S$ is also a subgraph of $C$.
It is easy to see that $S$ satisfies the three conditions in Definition \ref{def:gen-zig-cov}, so $\varphi$ is a generalized zigzag cover.
\end{proof}

\begin{remark}
\label{rmk:mult-parity}
Let $\varphi:C\to T\pb^1$ be a generalized zigzag cover of type $(g,\lambda,\mu,\undl z)$,
and $\rho$ be an effective colouring of $\varphi$.
From Proposition \ref{prop:mult-gen-zig-cov}, the parity of the multiplicity $\mult^\rb(\varphi,\rho)$ does
not depend on the colouring $\rho$.
In particular, if $s=0$ ({\it i.e.} $\undl x=\emptyset$), equation (\ref{eq:mult-gzc})
implies that $\mult^\rb(\varphi,\rho)$ is odd (see also \cite[Proposition 4.7]{rau2019}).
Since the multiplicity used in \cite[Definition 3.3]{rau2019} is different from the multiplicity
in \cite[Corollary 5.9]{mr-2015} (see \cite[Remark 3.5]{rau2019}),
in the case $s=0$ there are generalized zigzag covers having EVEN multiplicity according to 
the multiplicity of \cite[Definition 3.3]{rau2019}.
That are generalized zigzag covers having the third type of tails depicted in Figure \ref{fig:gzc-tail}.
If $t=0$ ({\it i.e.} $\undl y=\emptyset$), equation (\ref{eq:mult-gzc}) coincides with
the pure combinatorial computation carried out
in a previous version \cite[Proposition 4.4]{dlly-2023} of this manuscript.
\end{remark}

\begin{lemma}\label{lem:glue1}
Let $o$ be an odd integer, and let $\lambda_1,\lambda_2,\mu_1,\mu_2$ be four partitions such that 
$|\lambda_1|=|\mu_1|+o$ and $|\lambda_2|+o=|\mu_2|$.
Suppose that there exist $n_1$ generalized zigzag covers $\varphi_1:C_1\to T\pb^1$
of type $(g_1,\lambda_1,(\mu_1,o),\undl x_1\sqcup\undl y_1)$,
whose non-empty connected subgraph $S_1$ contains an outward end $l_1$ weighted by $o\in(\mu_1,o)$ and that is not in any symmetric fork.
Assume further that there are $n_2$ generalized zigzag covers $\varphi_2:C_2\to T\pb^1$
of type $(g_2,(\lambda_2,o),\mu_2,\undl x_2\sqcup\undl y_2)$)
whose non-empty connected subgraph $S_2$ contains an inward end $l_2$ weighted by $o\in(\lambda_2,o)$ and that is not in any symmetric fork.
Let $a$ be the largest number in the set $\undl x_1\sqcup\undl y_1$,
and let $b$ be the smallest number in the set $\undl x_2\sqcup\undl y_2$.
We denote by $\undl x_2'\sqcup\undl y_2'$ the set obtained by adding $|a-b|+1$ to each number in $\undl x_2\sqcup\undl y_2$.
Then, there are $n_1\cdot n_2$ generalized zigzag covers $\varphi:C\to T\pb^1$ of type
$(g_1+g_2,(\lambda_1,\lambda_2),(\mu_1,\mu_2),\undl x\sqcup\undl y)$ obtained by gluing $\varphi_1$ with $\varphi_2$ along $l_1$ and $l_2$.
Here, $\undl x=\undl x_1\cup\undl x_2'$ and $\undl y=\undl y_1\cup\undl y_2'$.
\end{lemma}

\begin{proof}
Let $\varphi_1:C_1\to T\pb^1$ (resp. $\varphi_2:C_2\to T\pb^1$) be a generalized zigzag cover of type 
$(g_1,\lambda_1,(\mu_1,o),\undl x_1\sqcup\undl y_1)$ (resp. $(g_2,(\lambda_2,o),\mu_2,\undl x_2\sqcup\undl y_2)$)
satisfying the assumption in Lemma \ref{lem:glue1}.
According to Remark \ref{rmk-combin1}, the graph $C_1$ (resp. $C_2$) is oriented and weighted by the map $\varphi_1$ (resp. $\varphi_2$).
Furthermore, the map $\varphi_1$ (resp. $\varphi_2$) induces a total order on inner vertices of $C_1$ (resp. $C_2$), denoted by $O(C_1)$ (resp. $O(C_2)$).

We glue the end $l_1\subset C_1$ with the end $l_2\subset C_2$ to obtain a new graph $C$.
In $C_1$ (resp. $C_2$) the end $l_1$ (resp. $l_2$) is oriented outward (resp. inward),
meaning it points from its inner vertex to its leaf (resp. from its leaf to its inner vertex).
The gluing procedure respects the orientations and weight functions on $C_1$ and $C_2$.
Specifically,
an edge $E$ in the new graph $C$ is oriented and weighted by the orientation and weight of $C_1$, if $E$ is contained in $C_1$. Otherwise, $E$ adopts the orientation and weight of $C_2$.
Consequently, the graph $C$ is oriented and weighted. 
We choose $O(C_1),O(C_2)$ as a total order on inner vertices of $C$.
This total order is compatible with the orientation, meaning
it extends the partial order of inner vertices of $C$ induced by the orientation on $C$.
From \cite[Remark 5.3]{rau2019}, this total order determines a unique tropical cover
$\varphi:C\to T\pb^1$ of type $(g_1+g_2,(\lambda_1,\lambda_2),(\mu_1,\mu_2),\undl x\sqcup\undl y)$,
where $\undl x=\undl x_1\cup\undl x_2'$ and $\undl y=\undl y_1\cup\undl y_2'$.
By the construction, we know that the tropical cover $\varphi:C\to T\pb^1$ is a generalized zigzag cover.

Let $\varphi_1':C_1'\to T\pb^1$ be a generalized zigzag cover of type 
$(g_1,\lambda_1,(\mu_1,o),\undl x_1\sqcup\undl y_1)$,
and let $\varphi_2':C_2'\to T\pb^1$ be a generalized zigzag cover of type 
$(g_2,(\lambda_2,o),\mu_1,\undl x_2\sqcup\undl y_2)$.
The aforementioned construction yields a generalized zigzag cover $\varphi':C'\to T\pb^1$
of type $(g_1+g_2,(\lambda_1,\lambda_2),(\mu_1,\mu_2),\undl x\sqcup\undl y)$ by gluing $\varphi_1'$ with $\varphi_2'$.
If $\varphi':C'\to T\pb^1$ is isomorphic to the generalized zigzag cover $\varphi:C\to T\pb^1$,
then $\varphi_1':C_1'\to T\pb^1$ is isomorphic to $\varphi_1:C_1\to T\pb^1$,
and $\varphi_2':C_2'\to T\pb^1$ is isomorphic to $\varphi_2:C_2\to T\pb^1$.
Hence, the statement is proved.
\end{proof}

\begin{definition}
\label{def:gen-zig-num}
Let $\zl_g(\lambda,\mu;\Lambda_{s,t})$ denote the set of isomorphism classes
of generalized zigzag covers of type $(g,\lambda,\mu,\undl z)$,
where the $(s,t)$-distribution $\undl z=\undl x\sqcup\undl y$ is
compatible with the $(s,t)$-tuple $\Lambda_{s,t}$.
The \textit{generalized zigzag number} of type $(g,\lambda,\mu,\undl z)$ is
the number
$$
Z_g(\lambda,\mu;\Lambda_{s,t})=\sum_{[\varphi]\in\zl_g(\lambda,\mu;\Lambda_{s,t})}2^{m(\varphi)}.
$$
Here,
$m(\varphi)=0$ if $\mult^\rb(\varphi,\rho)$ is odd for an effective colouring $\rho$, otherwise, $m(\varphi)=1$.
\end{definition}

\begin{remark}
\label{rem:inv-gen-zig-num}
The number $Z_g(\lambda,\mu;\Lambda_{s,t})$ does not depend on the signed splitting
of $\undl z$.
In fact, the parity of the multiplicity $\mult^\rb(\varphi,\rho)$
depends only on the generalized zigzag cover $\varphi$,
and not on the colouring $\rho$ of $\varphi$ (see Remark \ref{rmk:mult-parity}).
A generalized zigzag cover $\varphi:C\to T\pb^1$ is a resolving tropical cover with
additional combinatorial conditions, namely the three conditions in Definition \ref{def:gen-zig-cov}.
These three conditions are independent of both the colouring $\rho$ and the signed splitting of $\undl z$. 
Therefore, the number $Z_g(\lambda,\mu;\Lambda_{s,t})$ is independent of the signed splitting of $\undl z$.
\end{remark}

\begin{proposition}
\label{thm:low-bou1}
Fix two integers $d\geq1$, $g\geq0$, and fix two partitions $\lambda$, $\mu$ of $d$.
Let $s,t$ be two non-negative integers such that $2s+t=l(\lambda)+l(\mu)+2g-2$.
Assume that $2s+t>0$ and $\{\lambda,\mu\}\not\subset\{(2k),(k,k)\}$.
Let $\Lambda_{s,t}$ be an $(s,t)$-tuple with a signed splitting $(\Lambda_{s,t}^-,\Lambda_{s,t}^+)$.
Then, we have
$$
Z_g(\lambda,\mu;\Lambda_{s,t})\leq H_g^\rb(\lambda,\mu;\Lambda_{s,t}^-,\Lambda_{s,t}^+)\leq H_g^\cb(\lambda,\mu;s,t),
$$
and
$$
Z_g(\lambda,\mu;\Lambda_{s,t})\equiv H_g^\rb(\lambda,\mu;\Lambda_{s,t}^-,\Lambda_{s,t}^+)\mod 2.
$$
\end{proposition}

\begin{proof}
It is straightforward from Proposition \ref{prop:colour-GZC},
Proposition \ref{prop:mult-gen-zig-cov} and Definition \ref{def:gen-zig-num}.
\end{proof}

\section{Asymptotics for real double Hurwitz numbers with only triple ramification}
\label{sec:5}
When $t=0$, 
the number $Z_g(\lambda,\mu;\Lambda_{s,0})$ provides a lower bound for
real double Hurwitz numbers with only triple ramification.
In this section, we study the logarithmic asymptotic growth of
$H_g^\rb(\lambda,\mu;\Lambda_{s,0}^-,\Lambda_{s,0}^+)$ when $s\to\infty$ and $g$ is fixed.

\subsection{Non-vanishing theorem}

We recall the \textit{tail decomposition} of
a partition $\lambda=(\lambda_1,\ldots,\lambda_n)$
from \cite[Section $5$]{rau2019}.
We put $\lambda^2=(\lambda_1,\lambda_1,\ldots,\lambda_n,\lambda_n)$.
The tail decomposition of $\lambda$ is
\begin{equation}
\label{eq:tail-decomp1}
\lambda=(\lambda_{o,o}^2,\lambda_{e,e}^2,\lambda_o,\lambda_e),
\end{equation}
where $\lambda_{o,o}$, $\lambda_{e,e}$, $\lambda_o$ and $\lambda_e$
are partitions such that:
\begin{itemize}
    \item any entry of $\lambda_{o,o}$ is an odd integer;
    \item any entry of $\lambda_{e,e}$ is an even integer;
    \item any entry of $\lambda_o$ is odd and appears only once in $\lambda_o$;
    \item any entry of $\lambda_e$ is even and appears only once in $\lambda_e$.
\end{itemize}
Given a tail decomposition $\lambda=(\lambda_{o,o}^2,\lambda_{e,e}^2,\lambda_o,\lambda_e)$,
we use $\lambda\setminus\lambda_e$ to denote the partition with tail decomposition $(\lambda_{o,o}^2,\lambda_{e,e}^2,\lambda_o)$.

\begin{theorem}
\label{thm:non-vanish-UEC}
Let $g\geq0$ and $d\geq1$ be two integers,
and let $\lambda$, $\mu$ be two partitions of $d$
such that $l(\lambda)\equiv l(\mu)\mod 2$.
Let $\undl z$ be an $(s,0)$-distribution, where $s=\frac{1}{2}(l(\lambda)+l(\mu)+2g-2)$.
Suppose the following conditions hold:
\begin{equation}
\label{eq:condition-partition1}
\begin{aligned}
\min(l(\lambda_{o}),l(\mu_{o}))>0, l(\lambda_e)=l(\mu_e)=0,\\
|l(\lambda_o)-l(\mu_o)|<2\min(l(\lambda_{o,o}),l(\mu_{o,o})).
\end{aligned}
\end{equation}
Then, there exists at least one generalized zigzag cover of
type $(g,\lambda,\mu,\undl z)$.
\end{theorem}

\begin{proof}
Since $l(\lambda_o)+l(\mu_o)\equiv|\lambda|+|\mu|\mod 2$, $l(\lambda_o)+l(\mu_o)$ is even.
Suppose that $l(\lambda_o)-l(\mu_o)=2a$.
We construct a generalized zigzag cover of type $(g,\lambda,\mu,\undl z)$ by considering the three cases:
$a>0,a=0$ and $a<0$.

Case $a>0$: From the assumption $|l(\lambda_o)-l(\mu_o)|<2\min(l(\lambda_{o,o}),l(\mu_{o,o}))$,
it is possible to select $2a$ entries from $(\mu_{o,o}^2)$ and incorporate them into the partition $\mu_o$.
Suppose $(\mu_{o,o}^2)=(\bar\mu_{o,o}^2,\tilde\mu_{o,o}^2)$, and $l(\tilde\mu_{o,o})=a$.
The two partitions $(\lambda_o)$ and $(\mu_o,\tilde\mu_{o,o}^2)$ have the same length,
and we suppose $k=l(\lambda_o)=l(\mu_o,\tilde\mu_{o,o}^2)$.
Assume that the entries $\lambda_o^1,\ldots,\lambda_o^k$
in $\lambda_o$ satisfy $\lambda_o^1<\cdots<\lambda_o^k$,
and the entries $\mu_o^1,\ldots,\mu_o^k$ in $(\mu_o,\tilde\mu_{o,o}^2)$ satisfy
$\mu_o^1\leq\cdots\leq\mu_o^k$.
A generalized zigzag cover of type $(0,\lambda,\mu,\undl z)$ is
constructed as follows.
\begin{enumerate}[$(a)$]
    \item We begin with $k$ strings $S_1,\ldots,S_k$. Label the two leaves of $S_i$ with
    $\lambda_o^{k-i+1}$ and $\mu_o^i$ for $i=1,\ldots,k$.
    The string $S_i$ is oriented from the leaf labelled with $\lambda_o^{k-i+1}$ to the leaf labelled with $\mu_o^i$.
    \begin{figure}[ht]
    \centering
    \begin{tikzpicture}
    \draw[line width=0.3mm] (-6,2)--(-4,1)--(6,1);
    \draw[line width=0.3mm,gray] (-6,1.8)--(-5.8,1.6)--(-6,1.4);
    \draw[line width=0.3mm,gray] (-5.8,1.6)--(-5.2,1.6);
    \draw[line width=0.3mm,gray] (-6,1.3)--(-4.7,1.1)--(-6,0.9);
    \draw[line width=0.3mm,gray] (-4.7,1.1)--(-4.2,1.1);
    \draw[line width=0.3mm,gray] (-4,1)--(-3.5,0.5);
    \draw[line width=0.3mm] (-6,0)--(-3.5,0.5)--(0.1,0.5);
    \draw[line width=0.3mm] (0.5,0.5) +(0:0.4 and 0.2) arc (0:180:0.4 and 0.2);
    \draw[line width=0.3mm,gray] (0.5,0.5) +(0:0.4 and 0.2) arc (0:-180:0.4 and 0.2);
    \draw[line width=0.3mm] (0.9,0.5)--(1.1,0.5);
    \draw[line width=0.3mm] (1.5,0.5) +(0:0.4 and 0.2) arc (0:180:0.4 and 0.2);
    \draw[line width=0.3mm,gray] (1.5,0.5) +(0:0.4 and 0.2) arc (0:-180:0.4 and 0.2);
    \draw[line width=0.3mm] (1.9,0.5)--(6,0.5);
    \draw[line width=0.3mm,gray] (-1.5,0.5)--(-1,0);
    \draw[line width=0.3mm] (-6,-0.5)--(-1,0)--(4,0)--(6,-1);
    \draw[line width=0.3mm,gray] (4,0)--(4.5,0);
    \draw[line width=0.3mm,gray] (6,0.2)--(4.5,0)--(6,-0.2);
    \draw[line width=0.3mm,gray] (5.2,-0.6)--(5.6,-0.6);
    \draw[line width=0.3mm,gray] (6,-0.8)--(5.6,-0.6)--(6,-0.4);
    \draw (-6,-1.5)--(6,-1.5);
    \foreach \Point in {(-5.8,-1.5),(-5.2,-1.5),(-4.7,-1.5),(-4.2,-1.5),(-4,-1.5),(-3.5,-1.5),(-1.5,-1.5),(-1,-1.5),(0.1,-1.5),(0.9,-1.5),(1.1,-1.5),(1.9,-1.5),(4,-1.5),(4.5,-1.5),(5.2,-1.5),(5.6,-1.5)}
    \draw[fill=black] \Point circle (0.03);
    \end{tikzpicture}
    \caption{An example of generalized zigzag cover: even contractible edges and tails with symmetric forks are drawn in gray, strings and odd contractible edges are in black.}
    \label{fig:example-UEC1}
    \end{figure}
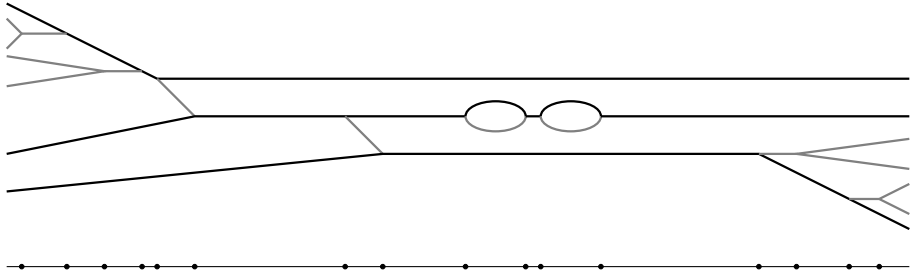
    \item We attach $l(\lambda_{o,o})+l(\lambda_{e,e})$ tails, each with
    symmetric forks (\textit{i.e.} the first and third type in Figure \ref{fig:gzc-tail}),
to the string $S_1$ (See Figure \ref{fig:example-UEC1} for an example).
    Let the resulting graph be denoted as $C_1$.
    The symmetric forks are labelled by the
    entries of $\lambda_{o,o}$ and $\lambda_{e,e}$, with each entry being used exactly once.
    These tails are oriented from their leaves to the
    intersections of the tails with the string $S_1$.
    Then $C_1$ is an oriented graph, and the inner edges adjacent to the symmetric forks are considered as contractible edges.
    The end in $C_1\setminus{\cup_{i=2}^k}S_i$ that is adjacent to a leaf associated with $\lambda_o$ (resp. $\lambda_{o,o}$ or $\lambda_{e,e}$)
    is assigned a weight by $\lambda_o$ (resp. $\lambda_{o,o}$ or $\lambda_{e,e}$).
    For any $3$-valent vertex in $C_1\setminus{\cup_{i=2}^k}S_i$, by applying the balancing condition
    at this vertex, the weight of its outgoing edge
    can be determined by the weights of its incoming edges. Hence, every edge in $C_1\setminus{\cup_{i=2}^k}S_i$ is weighted.
    \item We connect the string $S_i$ to the string $S_{i+1}$ via
    a contractible edge $E_i$, where $1\leq i\leq k-1$.
    Assume the resulting graph is denoted as $C_2$.
    Suppose that $E_i$ intersects with $S_i$ at $v_i$ and
    intersects with $S_{i+1}$ at $v_{i+1}'$. Note that all strings $S_1,\ldots,S_k$ are oriented.
    The contractible edge $E_i$ is chosen such that $v_i$
    lies behind all other inner vertices in $S_i$ and $v_{i+1}'$ precedes all other inner vertices in $S_{i+1}$ according to the orientation.
    The end in $S_1$ adjacent to the leaf associated with $\mu_o^1$ is an outgoing edge of $v_1$,
and let $\mu_o^1$ be the weight of this end.
Suppose that $2|\lambda_{o,o}|+2|\lambda_{e,e}|+\lambda_o^{k}\leq\mu_o^1$.
If $\lambda_o^{k-1}\geq\mu_o^{2}$, from our assumption in the first paragraph of this proof
we have $\lambda_o^k>\lambda_o^{k-1}\geq\mu_o^{2}\geq\mu_o^1$,
which contradicts the inequality $2|\lambda_{o,o}|+2|\lambda_{e,e}|+\lambda_o^{k}\leq\mu_o^1$.
Therefore, we must have $\lambda_o^1<\cdots<\lambda_o^{k-1}<\mu_o^{2}\leq\cdots\leq\mu_o^k$. However, this leads to a contradiction:
    $
    d=2|\lambda_{o,o}|+2|\lambda_{e,e}|+\sum_{j=1}^k\lambda_o^j<
    \sum_{j=1}^{k}\mu_o^j\leq d.
    $
    We conclude that $2|\lambda_{o,o}|+2|\lambda_{e,e}|+\lambda_o^{k}>\mu_o^1$.
The orientation of $E_1$ is chosen such that it is an outgoing edge of $v_1$ (\textit{i.e.}, $E_1$ is oriented from $v_1$ towards $v_2'$),
and $E_1$ is weighted by $2|\lambda_{o,o}|+2|\lambda_{e,e}|+\lambda_o^{k}-\mu_o^1$.
The end in $S_i$, $i=2,\ldots,k$ (resp. $i=2,\ldots,k-1$), adjacent to the leaf associated with $\lambda_o^{k-i+1}$ (resp. $\mu_o^{i}$)
is weighted by $\lambda_o^{k-i+1}$ (resp. $\mu_o^{i}$).
Suppose that all edges in $C_2\setminus(\cup_{j=i+1}^kS_j\cup(\cup_{l=i}^{k-1}E_l))$ are already weighted and oriented,
where $i\in\{1,\ldots,k-1\}$.
We show that the contractible edge
    $E_i$ can be oriented such that it is an outgoing edge of $v_i$ in the following.
    We use the notation $\lambda_o^{k+1}=\mu_o^0=0$.
    Suppose that the following inequality holds:
    \begin{equation}\label{eq:lem-non-vanish}
    (2|\lambda_{o,o}|+2|\lambda_{e,e}|+\sum_{j=k-i+2}^{k+1}\lambda_o^j-\sum_{j=0}^{i-1}\mu_o^j)+\lambda_o^{k-i+1}\leq
    \mu_o^i.
    \end{equation}
    If $\lambda_o^{k-i}\geq\mu_o^{i+1}$, from our assumption in the first paragraph of this proof,
    we have $\lambda_o^j>\mu_o^h$ for any $j\in\{k-i+1,k-i+2,\ldots,k\}$ and any $h\in\{1,2,\ldots,i\}$.
    This would imply the relation $\sum_{j=k-i+1}^{k+1}\lambda_o^j>\sum_{j=0}^{i}\mu_o^j$, which contradicts  equation (\ref{eq:lem-non-vanish}).
    Therefore, we must have $\lambda_o^1<\cdots<\lambda_o^{k-i}<\mu_o^{i+1}\leq\cdots\leq\mu_o^k$.
    This leads to a contradiction:
    $$
    d=2|\lambda_{o,o}|+2|\lambda_{e,e}|+\sum_{j=1}^k\lambda_o^j<
    \sum_{j=1}^{k}\mu_o^j\leq d.
    $$
    We conclude that:
    $$
    2|\lambda_{o,o}|+2|\lambda_{e,e}|+\sum_{j=k-i+1}^{k+1}\lambda_o^j>\sum_{j=0}^{i}\mu_o^j.
    $$
    The balancing condition at $v_i$ implies that the contractible edge
    $E_i$ is an outgoing edge of $v_i$, and the weight of $E_i$ is given by
$2|\lambda_{o,o}|+2|\lambda_{e,e}|+\sum_{j=k-i+1}^{k+1}\lambda_o^j-\sum_{j=0}^{i}\mu_o^j$.
The weight of the outgoing edge of $v_{i+1}'$ is determined by the balancing conditon at $v_{i+1}'$.
Then all edges in $C_2\setminus(\cup_{j=i+2}^kS_j\cup(\cup_{l=i+1}^{k-1}E_l))$ are weighted and oriented.
By repeating this process, we ensure that all edges in $C_2$ are weighted and oriented.
    \item We attach $l(\bar\mu_{o,o})+l(\mu_{e,e})$ tails with
    symmetric forks to the string $S_k$.
Note that the vertex $v_{k}'$ precedes all intersection points of these tails with $S_{k}$.
    Let the resulting graph be denoted as $C$.
The inner edges adjacent to the $l(\bar\mu_{o,o})+l(\mu_{e,e})$ symmetric forks are contractible edges.
    The symmetric forks are labelled by the 
    entries of $\bar\mu_{o,o}$ and $\mu_{e,e}$, with each entry used exactly once.
    These tails are oriented from their
    intersection points with $S_k$ towards their leaves.
Consequently, $C$ becomes an oriented graph.
    The end in $C$ adjacent to a leaf associated with $\bar\mu_{o,o}$ or $\mu_{e,e}$ (resp. $\mu_o^k$)
    is weighted by $\bar\mu_{o,o}$ or $\mu_{e,e}$ (resp. $\mu_o^k$). 
    For any $3$-valent vertex in $C\setminus(\cup_{i=1}^{k-1}(S_i\cup E_i))$, 
by applying the balancing condition at this vertex, the weight function can be extended to all $C$.
Hence, $C$ is a weighted and oriented graph of genus zero.
The orientation of the graph $C$ induces a partial order $\pl$ on its inner vertices.
It follows from \cite[Remark 5.3]{rau2019} that any total order on inner vertices of $C$ that extends $\pl$ can induce a unique tropical cover $\varphi:C\to T\pb^1$ of type $(0,\lambda,\mu,\undl z)$.
To ensure that $\varphi:C\to T\pb^1$ is a generalized zigzag cover,
it suffices to choose a compatible total order on the inner vertices of $C$ such that
the two endpoints of every contractible edge are adjacent to each other according to this total order.
The existence of such a total order is obvious.
Thus, we obtain a generalized zigzag cover $\varphi:C\to T\pb^1$ of type $(0,\lambda,\mu,\undl z)$.
\end{enumerate}
To get a generalized zigzag cover of any genus $g$,
one can appropriately put $g$ contractible circles to an odd end (see Figure \ref{fig:example-UEC1}) for an example).

Case $a=0$: In this case, the two partitions $\lambda_o$ and $\mu_o$ have the same length.
Suppose $k=l(\lambda_o)=l(\mu_o)$.
Assume further that the entries $\lambda_o^1,\ldots,\lambda_o^k$
in $\lambda_o$ satisfy $\lambda_o^1<\cdots<\lambda_o^k$,
and similarly, the entries $\mu_o^1,\ldots,\mu_o^k$ in $\mu_o$ satisfy
$\mu_o^1<\cdots<\mu_o^k$.
The construction outlined for the case $a>0$ can be adapted to the case $a=0$ with straightforward modifications, so we omit the details here.

Case $a<0$: From the assumption $|l(\lambda_o)-l(\mu_o)|<2\min(l(\lambda_{o,o}),l(\mu_{o,o}))$,
it is possible to select $|2a|$ entries from $(\lambda_{o,o}^2)$ and add them to the partition $\lambda_o$.
Suppose $(\lambda_{o,o}^2)=(\bar\lambda_{o,o}^2,\tilde\lambda_{o,o}^2)$, and $l(\tilde\lambda_{o,o})=|a|$.
The two partitions $(\lambda_o,\tilde\lambda_{o,o}^2)$ and $(\mu_o)$ have the same length,
so we suppose $k=l(\lambda_o,\tilde\lambda_{o,o}^2)=l(\mu_o)$.
Assume further that the entries $\lambda_o^1,\ldots,\lambda_o^k$
in $(\lambda_o,\tilde\lambda_{o,o}^2)$ satisfy $\lambda_o^1\leq\cdots\leq\lambda_o^k$,
and similarly, the entries $\mu_o^1,\ldots,\mu_o^k$ in $\mu_o$ satisfy
$\mu_o^1<\cdots<\mu_o^k$.
The construction detailed for the case $a>0$ can be adapted to the case $a<0$ with straightforward modifications, so we omit the specific details here.
\end{proof}

\subsection{Asymptotic growth rate}
In this section, we consider the asymptotic growth of
generalized zigzag numbers when the degree goes to infinity and the genus is fixed.

\begin{lemma}
\label{lem:asymp1}
For any integer $m>3$, there are at least $\lfloor\frac{m-1}{3}\rfloor!$
generalized zigzag covers of type $(0,(1^m),(1^m),\undl z)$,
where $\undl z$ is an $(m-1,0)$-distribution.
\end{lemma}

\begin{proof}
We first consider the case when $m=3n+1$.
We show that there are at least $n!$ generalized zigzag
covers of type $(0,(1^{3n+1}),(1^{3n+1}),\undl z)$.
Let $\varphi_i:C_i\to T\pb^1$ be a generalized zigzag cover
of type $(0,\lambda_i,\mu_i,\undl z_i)$ depicted
in Figure \ref{fig:component-UEC1}, where $i=1,\ldots,n$, and $\lambda_i=\mu_i=(1^4)$.
\begin{figure}[ht]
    \centering
    \begin{tikzpicture}
    \draw[line width=0.3mm] (-6,2)--(-4,1)--(6,1);
    \draw[line width=0.3mm,gray] (-6,1.7)--(-5.5,1.5)--(-6,1.3);
    \draw[line width=0.3mm,gray] (-5.5,1.5)--(-5,1.5);
    \draw[line width=0.3mm,gray] (-4,1)--(-3.5,0.5);
    \draw[line width=0.3mm] (-6,0)--(-3.5,0.5)--(4,0.5)--(6,-0.5);
    \draw[line width=0.3mm,gray] (4,0.5)--(5,0.5);
    \draw[line width=0.3mm,gray] (6,0.7)--(5,0.5)--(6,0.3);
    \draw (-6,-1)--(6,-1);
    \foreach \Point in {(-5.5,-1),(-5,-1),(-4,-1),(-3.5,-1),(4,-1),(5,-1)}
    \draw[fill=black] \Point circle (0.03);
    \draw (-5.2,1.8) node {\tiny $e_i$}  (-5.2,-0.1) node {\tiny $\bar e_i$} (5,1.2) node {\tiny $e_i'$} (5,-0.3) node {\tiny $\bar e_i'$};
    \end{tikzpicture}
    \caption{generalized zigzag cover $\varphi_i:C_i\to T\pb^1$ of type $(0,1^4,1^4,\undl x_i)$.}
    \label{fig:component-UEC1}
\end{figure}
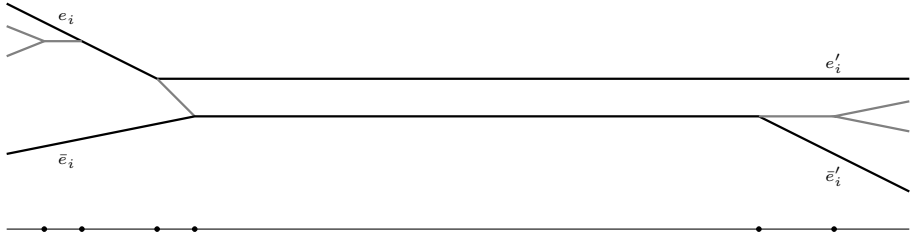
Let $e_i,e_i', \bar e_i,\bar e_i'$ be the four ends of $C_i$ depicted in Figure \ref{fig:component-UEC1}.
From Remark \ref{rmk-combin1}, the map $\varphi_i$ induces an orientation $\ol_i$ on $C_i$.
Let $\pl_i$ be the partial order on inner vertices of $C_i$ determined by $\ol_i$.
The map $\varphi_i$ and $\undl z_i$ also determine a total order of inner vertices of $C_i$.
Let $O(C_i)$ denote this total order of inner vertices of $C_i$. 
Note that the total order is an extension of $\pl_i$.
Every edge of $C_i$ is weighted by the map $\varphi_i$, and
the weight of an edge can be obtained by using the balancing conditions at every vertices of $C_i$.

Let $C$ be a graph obtained by gluing $C_1,\ldots, C_n$ as follows.
For each $i\in\{1,\ldots,n-1\}$, we glue $C_i$ with $C_{i+1}$ via either surgery $(1)$ or surgery $(2)$ as described below:
\begin{itemize}
\item[(1)]  glue the end $e_{i+1}$ of $C_{i+1}$ with
the end $\bar e_i'$ of $C_i$;
\item[(2)] glue the end $e_{i+1}'$ of $C_{i+1}$ with
the end $\bar e_i$ of $C_i$.
\end{itemize}
The gluing procedure is compatible with orientations and weights on $C_1,\ldots,C_n$. 
In other words, if an edge $e$ in $C$ is  is an edge in $C_i$, then $e$ is oriented and weighted according to the 
orientation and weight function on $C_i$.
Hence, the graph $C$ is oriented and weighted. The orientation $\ol$ on $C$
induces a partial order $\pl$ on the inner vertices of $C$.
By remark \cite[Remark 5.3.]{rau2019}, any choice of a total order on inner vertices of $C$ that extends 
$\pl$ gives rise to a tropical cover $\varphi:C\to T\pb^1$ of type $(0,(1^{3n+1}),(1^{3n+1}),\undl z)$.

Let $\sigma$ be a permutation of the set $\{1,\ldots,n\}$.
For any $i\in\{1,\ldots,n\}$, if $i+1$ follows $i$ in the sequence
$\sigma(1),\ldots,\sigma(n)$, we apply type $(1)$ surgery to $C_i$ and $C_{i+1}$.
Otherwise, we apply type $(2)$ surgery to $C_i$ and $C_{i+1}$.
Let $C_\sigma$ denote the graph obtained through this process.
The graph $C_\sigma$ is oriented and weighted as described in the preceding paragraph.
We denote the orientation of $C_\sigma$ by $\ol_\sigma$,
and the partial order on its inner vertices, determined by $\ol_\sigma$, by $\pl_\sigma$.
The total order of the inner vertices of $C_\sigma$, given by $(O(C_{\sigma(1)}),\ldots,O(C_{\sigma(n)}))$,
is an extension of $\pl_\sigma$.
Let $\varphi_\sigma:C_\sigma\to T\pb^1$ be the unique tropical cover of type $(0,(1^{3n+1}),(1^{3n+1}),\undl z)$ that
corresponds to the total order $(O(C_{\sigma(1)}),\ldots,O(C_{\sigma(n)}))$
(see Figure \ref{fig:example-UEC2} for an example).
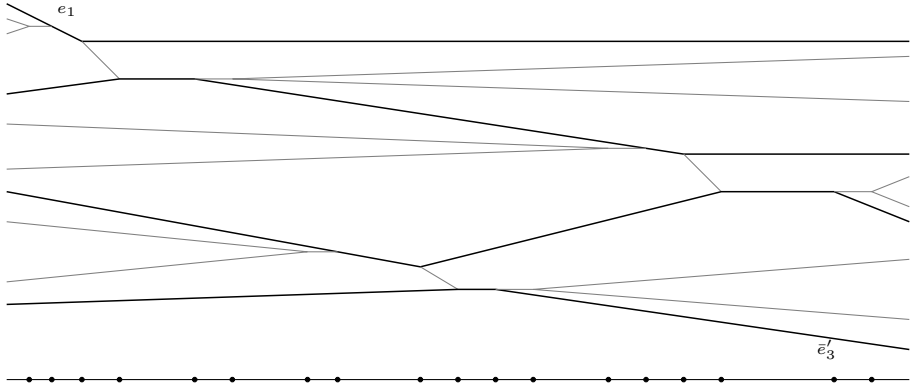
\begin{figure}[ht]
    \centering
    \begin{tikzpicture}
    \draw[line width=0.2mm] (-6,2)--(-5,1.5)--(6,1.5);
    \draw[gray] (-6,1.8)--(-5.7,1.7)--(-6,1.6);
    \draw[gray] (-5.7,1.7)--(-5.4,1.7);
    \draw[gray] (-5,1.5)--(-4.5,1);
    \draw[line width=0.2mm] (-6,0.8)--(-4.5,1)--(-3.5,1)--(3,0)--(6,0);
    \draw[gray] (-3.5,1)--(-3,1);
    \draw[gray] (6,1.3)--(-3,1)--(6,0.7);
    \draw[gray] (2,0.077)--(2.5,0.077);
    \draw[gray] (-6,0.4)--(2,0.077)--(-6,-0.2);
    \draw[gray] (3,0)--(3.5,-0.5);
    \draw[line width=0.2mm] (-6,-0.5)--(-0.5,-1.5)--(3.5,-0.5)--(5,-0.5)--(6,-0.9);
    \draw[gray] (6,-0.3)--(5.5,-0.5)--(6,-0.7);
    \draw[gray] (5,-0.5)--(5.5,-0.5);
    \draw[gray] (-6,-0.9)--(-2,-1.3)--(-6,-1.7);
    \draw[gray] (-2,-1.3)--(-1.6,-1.3);
    \draw[gray] (-0.5,-1.5)--(0,-1.8);
    \draw[line width=0.2mm] (-6,-2)--(0,-1.8)--(0.5,-1.8)--(6,-2.6);
    \draw[gray] (6,-1.4)--(1,-1.8)--(6,-2.2);
    \draw[gray] (1,-1.8)--(0.5,-1.8);
    \draw (-6,-3)--(6,-3);
    \foreach \Point in {(-5.7,-3),(-5.4,-3),(-5,-3),(-4.5,-3),(-3.5,-3),(-3,-3),(-2,-3),(-1.6,-3),(-0.5,-3),(0,-3),(0.5,-3),(1,-3),(2,-3),(2.5,-3),(3,-3),(3.5,-3),(5,-3),(5.5,-3)}
    \draw[fill=black] \Point circle (0.03);
    \draw (-5.2,1.9) node {\tiny $e_1$} (4.9,-2.6) node {\tiny $\bar e_3'$};
    \end{tikzpicture}
    \caption{An example: the permutation is $O(C_1),O(C_3),O(C_2)$.}
    \label{fig:example-UEC2}
\end{figure}
Since $\varphi_1,\ldots,\varphi_n$ are generalized zigzag covers,
the cover $\varphi_\sigma$ satisfies the conditions in Definition \ref{def:gen-zig-cov}.
Thus,  for each permutation $\sigma$ of the set $\{1,\ldots,n\}$, we obtain a generalized zigzag cover $\varphi_\sigma$.
There are $n!$ permutations of the set $\{1,\ldots,n\}$,
so we have $n!$ generalized zigzag covers.
An argument similar to the one presented in the last paragraph of the proof of Lemma \ref{lem:glue1}
shows that if $\sigma_1,\sigma_2$ are two different permutations, then $\varphi_{\sigma_1}$ is not isomorphic to $\varphi_{\sigma_2}$.
We omit the details of this argument here for brevity.

The same argument as in the previous case provides the proof for the lower bounds of the 
numbers of generalized zigzag covers of type
$(0,(1^{3n+2}),(1^{3n+2}),\undl z)$ and $(0,(1^{3n+3}),(1^{3n+3}),\undl z)$ for any $n>0$.
This holds when $\varphi_1:C_1\to T\pb^1$ in the previous case is replaced by the
generalized zigzag covers $\varphi_1':C_1'\to T\pb^1$ of type $(0,(1^5),(1^5),\undl z_1')$
and $\varphi_1'':C_1''\to T\pb^1$ of type $(0,1^6,1^6,\undl z_1'')$, respectively, 
as depicted in Figure \ref{fig:component-UEC2}.
\begin{figure}[ht]
    \centering
    \begin{tikzpicture}
    \draw[line width=0.2mm] (-6,2)--(-5,1.5)--(-2,1.5);
    \draw[gray] (-6,1.8)--(-5.7,1.7)--(-6,1.6);
    \draw[gray] (-5.7,1.7)--(-5.4,1.7);
    \draw[gray] (-5,1.5)--(-4.5,1);
    \draw[line width=0.2mm] (-6,0.8)--(-4.5,1)--(-2,1);
    \draw[gray] (-4,1)--(-3.5,0.5);
    \draw[line width=0.2mm] (-6,0.2)--(-3.5,0.5)--(-3,0.5)--(-2,0);
    \draw[gray] (-2,0.7)--(-2.5,0.5)--(-2,0.3);
    \draw[gray] (-3,0.5)--(-2.5,0.5);
    \draw (-6,-1)--(-2,-1);
    \foreach \Point in {(-5.7,-1),(-5.4,-1),(-5,-1),(-4.5,-1),(-4,-1),(-3.5,-1),(-3,-1),(-2.5,-1)}
    \draw[fill=black] \Point circle (0.03);
    \draw (-5.3,1.9) node {\tiny $e_1$} (-5.1,0) node {\tiny $\bar e_1$} (-2.2,1.6) node {\tiny $e_1'$} (-2.5,0) node {\tiny $\bar e_1'$} (-4,-1.5) node{\tiny $(1)$ $\varphi_1':C_1'\to T\pb^1$};
    \draw[line width=0.2mm] (2,2)--(3,1.5)--(6,1.5);
    \draw[gray] (2,1.8)--(2.3,1.7)--(2,1.6);
    \draw[gray] (2.3,1.7)--(2.6,1.7);
    \draw[gray] (3,1.5)--(3.5,1);
    \draw[line width=0.2mm] (2,0.8)--(3.5,1)--(6,1);
    \draw[gray] (4,1)--(4.5,0.5);
    \draw[line width=0.2mm] (2,0.2)--(4.5,0.5)--(6,0.5);
    \draw[gray] (4.7,0.5)--(5,0.1);
    \draw[line width=0.2mm] (2,-0.1)--(5,0.1)--(5.4,0.1)--(6,-0.4);
    \draw[gray] (6,0.3)--(5.7,0.1)--(6,-0.1);
    \draw[gray] (5.7,0.1)--(5.4,0.1);
    \draw (2,-1)--(6,-1);
    \foreach \Point in {(2.3,-1),(2.6,-1),(3,-1),(3.5,-1),(4,-1),(4.5,-1),(4.7,-1),(5,-1),(5.4,-1),(5.7,-1)}
    \draw[fill=black] \Point circle (0.03);
    \draw (2.6,1.9) node {\tiny $e_1$} (2.6,-0.3) node {\tiny $\bar e_1$} (5.6,1.6) node {\tiny $e_1'$} (5.5,-0.3) node {\tiny $\bar e_1'$} (4,-1.5) node{\tiny $(2)$ $\varphi_1'':C_1''\to T\pb^1$};
    \end{tikzpicture}
    \caption{Generalized zigzag covers $\varphi_1'$ and $\varphi_1''$.}
    \label{fig:component-UEC2}
\end{figure}
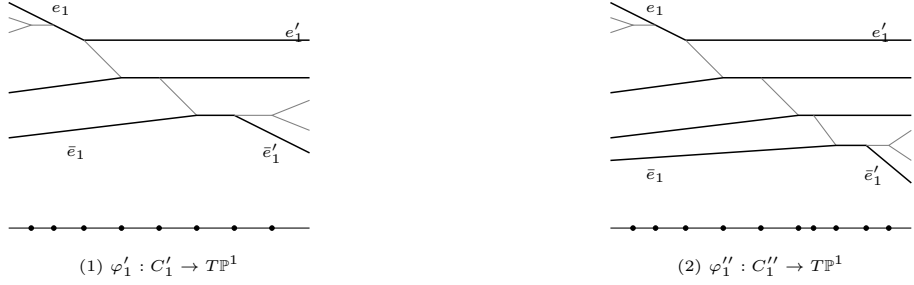
\end{proof}

\begin{theorem}
\label{thm:asymp1}
Given an integer $g\geq0$ and two partitions $\lambda,\mu$ such that $|\lambda|=|\mu|$ and $l(\lambda_e)=l(\mu_e)=0$,
there exists an integer $n_0$ that depends on $\lambda,\mu$ such that for any $h>n_0+3$, we have
\[
|\zl_g((\lambda,1^{h}),(\mu,1^{h});\Lambda_{s,0})|\geq|\zl_g((1^{h-n_0}),(1^{h-n_0});\Lambda_{h-n_0-1,0})|,
\]
where $2s=l(\lambda)+l(\mu)+2g-2+2h$.
\end{theorem}

\begin{proof}
Since $l(\lambda_e)=l(\mu_e)=0$, we conclude that $l(\lambda)\equiv l(\mu)\mod 2$.
Suppose that $|l(\lambda_o)-l(\mu_o)|=2b$ and $\min(l(\lambda_o),l(\mu_o))>0$.
Then the partitions $(\lambda,1^{2a}), (\mu,1^{2a})$,
where $a>b$, satisfy the conditions of Theorem \ref{thm:non-vanish-UEC}.
If, on the other hand, $|l(\lambda_o)-l(\mu_o)|=2b$ and $\min(l(\lambda_o),l(\mu_o))=0$,
then the partitions $(\lambda,1^{2a+1}), (\mu,1^{2a+1})$,
where $a>b$, satisfy the conditions of Theorem \ref{thm:non-vanish-UEC}.
Since the proofs for these two cases are identical, we will only consider the case where $\min(l(\lambda_o),l(\mu_o))>0$.

First, we construct four types of generalized zigzag covers.
Then, we apply Lemma \ref{lem:glue1} to obtain the required type of generalized zigzag covers.
\begin{itemize}
\item Let $\varphi_1:C_1\to T\pb^1$ be a generalized zigzag cover
of type $(g,(\lambda,1^{2a}), (\mu,1^{2a}),\undl z_1)$ constructed by Theorem \ref{thm:non-vanish-UEC}.
Note that, according to the proof of Theorem \ref{thm:non-vanish-UEC}, 
there exists an outward end $l_1$ in the non-empty subgraph $S_1\subset C_1$
which is weighted by $\mu_o^{max}$ and is not in a symmetric fork. Here, $\mu_o^{max}$ denotes the largest entry in $\mu_o$.
\item Let $\varphi_2:C_2\to T\pb^1$ be a generalized zigzag cover
of type $(0,(\mu_o^{max}), (1^{\mu_o^{max}}),\undl z_2)$ constructed according to Figure \ref{fig:asymp1}(1).
In the non-empty subgraph $S_2\subset C_2$, there exists an inward end $l_2'$ weighted by $\mu_o^{max}$ 
and an outward end $l_2$ weighted by $1$. Furtherover,  neither $l_2'$ nor $l_2$ is located in a symmetric fork.
\item Choose an integer $h>2a+\mu_o^{max}+1$.
Let $\varphi_3:C_3\to T\pb^1$ be a generalized zigzag cover 
of type $(0,(1^m),(1^m),\undl z_3)$ obtained by Lemma \ref{lem:asymp1},
where $m=h-2a-\mu_o^{max}+2>3$.
According to the proof of Lemma \ref{lem:asymp1}, in the non-empty subgraph $S_3\subset C_3$,
there are at least two outward ends, $l_3$ and $\bar l_3$, and two inward ends, $l_3'$ and $\bar l_3'$,
each of which is weighted by 1 and is not located in a symmetric fork.
\item Let $\varphi_4:C_4\to T\pb^1$ be a generalized zigzag cover
of type $(0,(1^{\mu_o^{max}}),(\mu_o^{max}), \undl z_4)$ constructed according to Figure \ref{fig:asymp1}(2).
In the non-empty subgraph $S_4\subset C_4$, there exists an inward end $l_4'$ weighted by $1$ and not located in a symmetric fork,
as well as an outward end $l_4$ weighted by $\mu_o^{max}$ and not in a symmetric fork.
\end{itemize}
\begin{figure}[ht]
    \centering
    \begin{tikzpicture}
    \draw[line width=0.3mm] (-1,1)--(3,-1);
    \draw[line width=0.3mm] (-0.5,0.75)--(-0.2,0.75)--(3,1);
    \draw[line width=0.3mm] (-0.2,0.75)--(3,0.5);
    \draw[line width=0.3mm] (2,-0.5)--(2.5,-0.5)--(3,-0.25);
    \draw[line width=0.3mm] (2.5,-0.5)--(3,-0.75);
    \draw[line width=0.3mm,dotted] (2,0.5)--(2,-0.3);
    \draw[line width=0.2mm,gray] (3.2,0.75)--(3.3,0.75)--(3.3,-0.5)--(3.2,-0.5);
    \foreach \Point in {(-0.5,-1.5),(-0.2,-1.5),(2,-1.5),(2.5,-1.5)}
    \draw[fill=black] \Point circle (0.03);
    \draw[line width=0.3mm] (-1.2,1) node{\tiny $\mu_o^{max}$} (3.1,1)
    node{\tiny $1$} (3.1,0.5)  node{\tiny $1$} (3.1,-0.25) node{\tiny $1$}
   (3.1,-0.75) node{\tiny $1$} (3.1,-1) node{\tiny $1$} (1,-1.8) node{\tiny $(1)$ $\varphi_2:C_2\to T\pb^1$};
    \draw[line width=0.3mm,blue] (3,0) node{\tiny $\frac{\mu_o^{max}-1}{2}$ tails};
    \draw[line width=0.3mm] (-1,-1.5)--(3,-1.5);
    \draw[line width=0.3mm] (5,1)--(9,-1);
    \draw[line width=0.3mm] (8.5,-0.75)--(8,-0.75)--(5,-1);
    \draw[line width=0.3mm] (8,-0.75)--(5,-0.5);
    \draw[line width=0.3mm] (6,0.5)--(5.5,0.5)--(5,0.25);
    \draw[line width=0.3mm] (5.5,0.5)--(5,0.75);
    \draw[line width=0.3mm,dotted] (5.5,-0.5)--(5.5,0.3);
    \draw[line width=0.2mm,gray] (4.8,-0.75)--(4.7,-0.75)--(4.7,0.5)--(4.8,0.5);
    \foreach \Point in {(5.5,-1.5),(6,-1.5),(8,-1.5),(8.5,-1.5)}
    \draw[fill=black] \Point circle (0.03);
    \draw[line width=0.3mm] (9.2,-0.8) node{\tiny $\mu_o^{max}$} (4.9,-1)
    node{\tiny $1$} (4.9,-0.5)  node{\tiny $1$} (4.9,0.25) node{\tiny $1$}
   (4.9,0.75) node{\tiny $1$} (4.9,1) node{\tiny $1$} (7,-1.8) node{\tiny $(2)$ $\varphi_4:C_4\to T\pb^1$};
    \draw[line width=0.3mm,blue] (5,0) node{\tiny $\frac{\mu_o^{max}-1}{2}$ tails};
    \draw[line width=0.3mm] (5,-1.5)--(9,-1.5);
    \end{tikzpicture}
    \caption{\small Two generalized zigzag covers $\varphi_2$ and $\varphi_4$.}
    \label{fig:asymp1}
\end{figure}
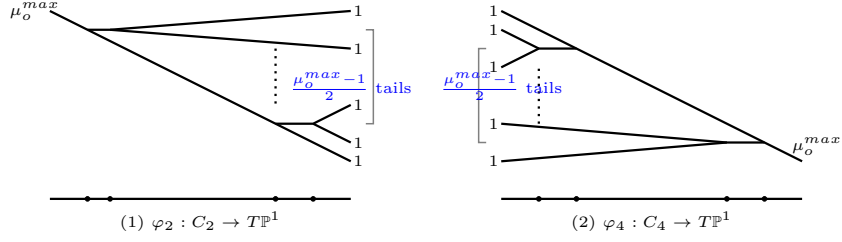

We apply Lemma \ref{lem:glue1} to the generalized zigzag covers $\varphi_1$ and $\varphi_2$.
Then, we obtain a generalized zigzag cover $\varphi':C'\to T\pb^1$ of type
$(g,(\lambda,1^{2a}),(\mu\setminus\mu_o^{max},1^{2a+\mu_o^{max}}),\undl z')$.
There is an outward end $l_2\subset C'$ weighted by $1$ and not located in a symmetric fork.
Next, we glue $\varphi'$ with $\varphi_3$ along $l_2$ and $l_3'$.
By applying Lemma \ref{lem:glue1} again, we obtain a generalized zigzag cover $\varphi'':C''\to T\pb^1$
of type $(g,(\lambda,1^{2a+m-1}),(\mu\setminus\mu_o^{max},1^{2a+\mu_o^{max}+m-1}),\undl z'')$.
In the non-empty subset $S''\subset C''$, there is an outward end $l_3$ weighted by $1$ and not in a symmetric fork.
Finally, we glue $\varphi''$ with $\varphi_4$ along $l_3$ and $l_4'$ to get a generalized zigzag cover $\varphi:C\to T\pb^1$ of type
$(g,(\lambda,1^{2a+m-1+\mu_o^{max}-1}),(\mu,1^{2a+\mu_o^{max}+m-2}),\undl z)$
by Lemma \ref{lem:glue1}.
Since there are at least two outward ends and two inward ends, each weighted by 1 and not located in a symmetric fork, in the non-empty subgraph
$S_3\subset C_3$, it follows that there is also an inward end, weighted by 1 and not located in a symmetric fork, in the non-empty subgraph
$S\subset C$.
See Figure \ref{fig:example-UEC3} for an example.
    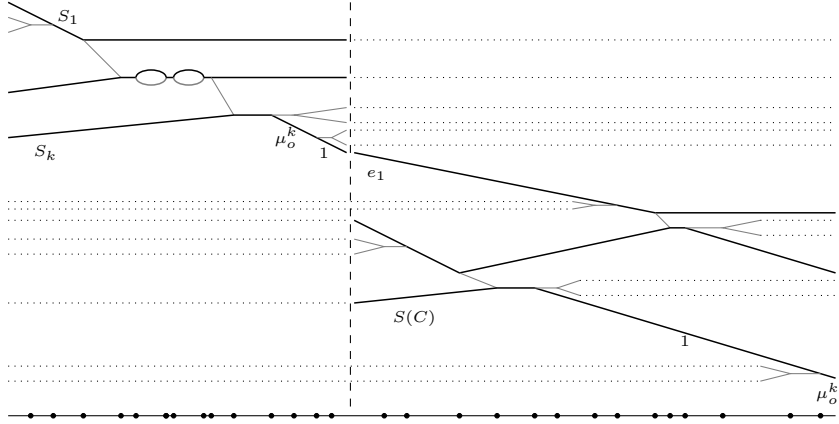
\begin{figure}[ht]
    \centering
    \begin{tikzpicture}
    \draw[line width=0.2mm] (-6,2)--(-5,1.5)--(-1.5,1.5);
    \draw[gray] (-6,1.8)--(-5.7,1.7)--(-6,1.6);
    \draw[gray] (-5.7,1.7)--(-5.4,1.7);
    \draw[gray] (-5,1.5)--(-4.5,1);
    \draw[line width=0.2mm] (-6,0.8)--(-4.5,1)--(-4.3,1);
    \draw[line width=0.2mm] (-4.1,1) +(0:0.2 and 0.1) arc (0:180:0.2 and 0.1);
    \draw[line width=0.2mm,gray] (-4.1,1) +(0:0.2 and 0.1) arc (0:-180:0.2 and 0.1);
    \draw[line width=0.2mm] (-3.9,1)--(-3.8,1);
    \draw[line width=0.2mm] (-3.6,1) +(0:0.2 and 0.1) arc (0:180:0.2 and 0.1);
    \draw[line width=0.2mm,gray] (-3.6,1) +(0:0.2 and 0.1) arc (0:-180:0.2 and 0.1);
    \draw[line width=0.2mm] (-3.4,1)--(-1.5,1);
    \draw[gray] (-3.3,1)--(-3,0.5);
    \draw[line width=0.2mm] (-6,0.2)--(-3,0.5)--(-2.5,0.5)--(-1.5,0);
    \draw[gray] (-1.5,0.6)--(-2.2,0.5)--(-1.5,0.4);
    \draw[gray] (-2.5,0.5)--(-2.2,0.5);
    \draw[gray] (-1.5,0.3)--(-1.7,0.2)--(-1.5,0.1);
    \draw[gray] (-1.9,0.2)--(-1.7,0.2);
    \draw (-2.3,0.2) node{\tiny $\mu_o^k$} (-1.8,0) node{\tiny $1$};
    \draw (-5.2,1.8) node {\tiny $S_1$} (-5.5,0) node {\tiny $S_k$};
    \draw[line width=0.2mm] (-1.4,0)--(2.6,-0.8)--(5,-0.8);
    \draw[gray] (1.8,-0.7)--(2.1,-0.7);
    \draw[gray] (1.5,-0.65)--(1.8,-0.7)--(1.5,-0.75);
    \draw[gray] (2.6,-0.8)--(2.8,-1);
    \draw[line width=0.2mm] (-1.4,-0.9)--(0,-1.6)--(2.8,-1)--(3,-1)--(5,-1.6);
    \draw[gray] (3,-1)--(3.5,-1);
    \draw[gray] (4,-0.9)--(3.5,-1)--(4,-1.1);
    \draw[gray] (-1,-1.25)--(-0.7,-1.25);
    \draw[gray] (-1.4,-1.35)--(-1,-1.25)--(-1.4,-1.15);
    \draw[gray] (0,-1.6)--(0.5,-1.8);
    \draw[line width=0.2mm] (-1.4,-2)--(0.5,-1.8)--(1,-1.8)--(5,-3);
    \draw[gray] (1,-1.8)--(1.3,-1.8);
    \draw[gray] (1.6,-1.7)--(1.3,-1.8)--(1.6,-1.9);
    \draw[gray] (4.8,-2.94)--(4.4,-2.94);
    \draw[gray] (4,-2.84)--(4.4,-2.94)--(4,-3.04);
    \draw (4.9,-3.2) node{\tiny $\mu_o^k$} (3,-2.5) node{\tiny $1$};
    \draw (-1.1,-0.3) node {\tiny $e_1$} (-0.6,-2.2) node {\tiny $S(C)$};
    \draw[dotted] (-1.4,1.5)--(5,1.5);
    \draw[dotted] (-1.4,1)--(5,1);
    \draw[dotted] (-1.4,0.6)--(5,0.6);
    \draw[dotted] (-1.4,0.4)--(5,0.4);
    \draw[dotted] (-1.4,0.3)--(5,0.3);
    \draw[dotted] (-1.4,0.1)--(5,0.1);
    \draw[dotted] (-6,-0.65)--(1.5,-0.65);
    \draw[dotted] (-6,-0.75)--(1.5,-0.75);
    \draw[dotted] (-6,-0.9)--(-1.5,-0.9);
    \draw[dotted] (-6,-1.35)--(-1.5,-1.35);
    \draw[dotted] (-6,-1.15)--(-1.5,-1.15);
    \draw[dotted] (-6,-2)--(-1.5,-2);
    \draw[dotted] (-6,-2.84)--(4,-2.84);
    \draw[dotted] (-6,-3.04)--(4,-3.04);
    \draw[dotted] (4,-0.9)--(5,-0.9);
    \draw[dotted] (4,-1.1)--(5,-1.1);
    \draw[dotted] (1.6,-1.7)--(5,-1.7);
    \draw[dotted] (1.6,-1.9)--(5,-1.9);
    \draw (-6,-3.5)--(5,-3.5);
    \draw[dashed] (-1.45,2)--(-1.45,-3.4);
    \foreach \Point in {(-5.7,-3.5),(-5.4,-3.5),(-5,-3.5),(-4.5,-3.5),(-4.3,-3.5),(-3.9,-3.5),(-3.8,-3.5),(-3.4,-3.5),(-3.3,-3.5),(-3,-3.5),(-3.9,-3.5),(-2.5,-3.5),(-2.2,-3.5),(-1.9,-3.5),(-1.7,-3.5),(2.6,-3.5),(2.8,-3.5),(2.1,-3.5),(1.8,-3.5),(0,-3.5),(0.5,-3.5),(3,-3.5),(3.5,-3.5),(-0.7,-3.5),(-1,-3.5),(1,-3.5),(1.3,-3.5),(4.8,-3.5),(4.4,-3.5)}
    \draw[fill=black] \Point circle (0.03);
    \end{tikzpicture}
    \caption{An example of glued universally enhanced cover.}
    \label{fig:example-UEC3}
\end{figure}
When $\mu_o^{max}=1$, we get a generalized zigzag cover $\varphi$ by
gluing $\varphi_1$ with $\varphi_3$. Let $n_0=2a+\mu_o^{max}-2$.
Since $m=h-2a-\mu_o^{max}+2>3$, we know $h=2a+m+\mu_o^{max}-2>n_0+3$.
Hence, from Lemma \ref{lem:glue1}, the number of generalized zigzag covers
of type $(g,(\lambda,1^h),(\mu,1^h),\undl z)$ satisfies
$$
|\zl_g((\lambda,1^{h}),(\mu,1^{h});\Lambda_{s,0})|\geq |\zl_g((1^{h-n_0}),(1^{h-n_0});\Lambda_{h-n_0-1,0})|.
$$
\end{proof}

We put
$$
z_{g,\lambda,\mu}(h)=Z_g((\lambda,(1^h)),(\mu,(1^h));\Lambda_{s,0}),
$$
where $2s=l(\lambda)+l(\mu)+2g-2+2h$. 
\begin{corollary}
\label{cor:asymp1}
Fix an integer $g\geq0$ and two partitions $\lambda$ and $\mu$ such that $|\lambda|=|\mu|$ and $l(\lambda_e)=l(\mu_e)=0$.
Then, the logarithmic asymptotics for the number $z_{g,\lambda,\mu}(h)$ is at least $\frac{h}{3}\ln h$ as $h\to\infty$.
\end{corollary}

\begin{proof}
If $h$ is large enough, from Lemma \ref{lem:asymp1} and Theorem \ref{thm:asymp1}, the generalized zigzag number
$$
z_{g,\lambda,\mu}(h)\geq \left\lfloor\frac{h-n_0-1}{3}\right\rfloor!,
$$
where $n_0$ is the number determined by Theorem \ref{thm:asymp1}.
Hence, we have
$$
\ln z_{g,\lambda,\mu}(h)\geq \ln\left\lfloor\frac{h}{3}-\frac{n_0+1}{3}\right\rfloor!.
$$
As $\ln x!\sim\ln x^x$ when $x\to\infty$,
we obtain the required lower bound of the logarithmic
asymptotics for $z_{g,\lambda,\mu}(h)$ when $h\to\infty$.
\end{proof}


\section{Lower bounds of real double Hurwitz numbers with triple ramification}\label{sec:6}
When $s\neq0$ and $t\neq0$, the generalized zigzag number 
$Z_g(\lambda,\mu;\Lambda_{s,t})$ depends on the sequence of partitions $\Lambda_{s,t}$.
In this section, we introduce the proper zigzag number 
$Z_g(\lambda,\mu;s,t)\leq Z_g(\lambda,\mu;\Lambda_{s,t})$ 
that does not depend on the sequence $\Lambda_{s,t}$.

\subsection{Proper zigzag numbers}
Let $g\geq0$, $d\geq1$ be two integers,
and suppose that $\lambda$, $\mu$ are two partitions of $d$.
Let $s,t$ be two non-negative integers such that $2s+t=l(\lambda)+l(\mu)+2g-2$.
Choose a set $\undl w=\{w_1,\ldots,w_{2s+t}\}\subset\rb$ of $2s+t$ points whose elements
satisfy $w_1<\cdots<w_{2s+t}$.
Let $\undl z=\undl x\sqcup\undl y$ be the $(s,t)$-distribution given by
$\undl x=\{x_1=(w_{1},w_2),\ldots,x_{s}=(w_{2s-1},w_{2s})\}$ and $\undl y=\{w_{2s+1},\ldots,w_{2s+t}\}$.
We call $\undl z=\undl x\sqcup\undl y$ a trivial $(s,t)$-distribution.
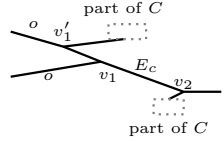
\begin{figure}[ht]
    \centering
    \begin{tikzpicture}
    \draw[line width=0.3mm] (-1.3,0.3)--(-0.6,0.1)--(-0.1,-0.1)--(1,-0.5)--(1.5,-0.5);
    \draw[line width=0.3mm] (-0.6,0.1)--(0.2,0.2);
    \draw[line width=0.3mm] (-1.3,-0.3)--(-0.1,-0.1);
    \draw[line width=0.3mm] (1,-0.5)--(0.8,-0.6);
    \draw[line width=0.3mm,gray,dotted] (0,0.2)--(0.5,0.2)--(0.5,0.4)--(0,0.4)--(0,0.2);
    \draw[line width=0.3mm,gray,dotted] (0.6,-0.6)--(1,-0.6)--(1,-0.8)--(0.6,-0.8)--(0.6,-0.6);
    \draw[line width=0.3mm] (-1,0.4) node{\tiny $o$} (-0.8,-0.3)
    node{\tiny $o$} (0.5,-0.15) node{\tiny $E_c$} (0.2,0.6) node{\tiny part of $C$}
 (0.8,-1) node{\tiny part of $C$} (1,-0.4) node{\tiny $v_2$} (0,-0.3) node{\tiny $v_1$} (-0.6,0.3) node{\tiny $v_1'$};
    \end{tikzpicture}
    \caption{\small Characteristic edge $E_c$ of a properly mixed generalized zigzag cover.}
    \label{fig:refine-zig}
\end{figure}

\begin{definition}
\label{def:ref-zig-cov}
Let $\undl z=\undl x\sqcup\undl y$ be a trivial $(s,t)$-distribution.
A generalized zigzag cover $\varphi:C\to T\pb^1$ of type $(g,\lambda,\mu,\undl x\sqcup\undl y)$
is called \textit{properly mixed} if it satisfies the following conditions.
\begin{enumerate}[$(1)$]
    \item There exists an odd inner edge $E_c\subset S$ such that
$C\setminus E_c^\circ$ consists of two connected components.
Moreover, vertices of one component of $C\setminus E_c^\circ$ are mapped onto $\undl x$ by $\varphi$,
and vertices of the other component of $C\setminus E_c^\circ$ are mapped onto $\undl y$ by $\varphi$.
The inner edge $E_c$ is called the {\it characteristic} edge.
    \item Let $v_1, v_2$ be the two endpoints of $E_c$, and suppose that $E_c$ is oriented from $v_1$ to $v_2$ by $\varphi$.
There exists a vertex $v_1'\in C$ such that $(v_1',v_1)$ is mapped to the pair $x_1\in\undl x$ by $\varphi$.
\item The local picture of $C$ at $(v_1',v_1)$ is depicted in Figure $\ref{fig:pair-EC}(\textrm{xii})$.
Moreover, the two odd incoming edges of $(v_1',v_1)$ are ends of $C$ with the same weight.
\item
The vertex $v_2$ is mapped to $w_{2s+t}$ by $\varphi$, and $v_2$ has two odd incoming edges.
(see Figure \ref{fig:refine-zig}).
\end{enumerate}
\end{definition}

\begin{definition}
\label{def:ref-zig-num}
Let $\zl_g(\lambda,\mu;s,t)$ be the set of properly
mixed generalized zigzag covers of type $(g,\lambda,\mu,\undl x\sqcup\undl y)$.
The number of properly mixed generalized zigzag covers of type $(g,\lambda,\mu,\undl x\sqcup\undl y)$
is called the \textit{properly zigzag number}, and it is denoted by $Z_g(\lambda,\mu;s,t)$.
\end{definition}

\subsection{Wall-crossing and lower bounds for generalized zigzag numbers}
We study the wall-crossing phenomenon that occurs when a simple branch point crosses a triple branch point. 
Then we show that properly zigzag numbers
are lower bounds for generalized zigzag numbers for any $(s,t)$-tuple $\Lambda_{s,t}$.

\begin{theorem}
\label{thm:lower-bounds}
Let $g\geq0$ be an integer,
and let $\lambda$, $\mu$ be two partitions such that $|\lambda|=|\mu|$.
Let $s,t$ be two non-negative integers satisfying $2s+t=l(\lambda)+l(\mu)+2g-2$.
Then, for any $(s,t)$-tuple $\Lambda_{s,t}$, we have the inequality
\[
Z_g(\lambda,\mu;s,t)\leq Z_g(\lambda,\mu;\Lambda_{s,t}).
\]
\end{theorem}

\begin{proof}
We suppose that $\zl_g(\lambda,\mu;s,t)\neq\emptyset$.
To establish our claim, it suffices to show that for any $(s,t)$-tuple $\Lambda_{s,t}$, there exists an injective map
$$
\Phi:\zl_g(\lambda,\mu;s,t)\to\zl_g(\lambda,\mu;\Lambda_{s,t}).
$$
Let $\varphi:C\to T\pb^1\in\zl_g(\lambda,\mu;s,t)$ be a properly mixed generalized zigzag cover
of type $(g,\lambda,\mu,\undl x\sqcup\undl y)$.
In this paragraph, we recall some properties of $\varphi$.
Let $\undl w=\{w_1,\ldots,w_{2s+t}\}\subset\rb$ be the set that determines the trivial $(s,t)$-distribution $\undl z=\undl x\sqcup\undl y$.
Let $v_1'$, $v_1$ and $v_2$ be the vertices of $C$ as given in Definition \ref{def:ref-zig-cov}.
Then $(v_1',v_1)=\varphi^{-1}(x_1)$ and $v_2=\varphi^{-1}(w_{2s+t})$.
The tropical curve $C$ is oriented and weighted by $\varphi$ according to
Remark \ref{rmk-combin1}.
The orientation $\ol$ of $C$ induces a partial order $\pl$ on the inner vertices of $C$.
Let $W$ be the weight function on $C$.
From Remark \ref{rmk-combin1}, the cover $\varphi$ and $\undl z$ induce a total order $O(C)$ on
the inner vertices of $C$. More precisely, $O(C)$ is the sequence $\varphi^{-1}(x_1),\ldots,\varphi^{-1}(x_s),\varphi^{-1}(w_{2s+1}),\ldots,\varphi^{-1}(w_{2s+t})$.
The total order $O(C)$ is an extension of the partial order $\pl$.
Denote by $\undl v_y$ the set of all inner vertices of $C$ that are mapped onto the set $\undl y$,
and by $\undl v_x$ the set of all pairs $(v,v')$ of inner vertices of $C$ that are mapped onto $\undl x$.
Let $O(C)|_{\undl v_y}$ and $O(C)|_{\undl v_x}$
be the restricted total orders of $O(C)$ on $\undl v_y$ and $\undl v_x$, respectively.
Note that $O(C)|_{\undl v_y}$ is the sequence $\varphi^{-1}(w_{2s+1}),\ldots,\varphi^{-1}(w_{2s+t})$,
and $O(C)|_{\undl v_x}$ is the sequence $\varphi^{-1}(x_1),\ldots,\varphi^{-1}(x_s)$.

Let $\Lambda=(\Lambda_1',\ldots,\Lambda_{s+t}')$ be
an $(s,t)$-tuple such that $\Lambda_i'=(2,1,\ldots,1)$ for $i=1,\dots,t$
and $\Lambda_j'=(3,1,\ldots,1)$ for $j=t+1,\ldots,t+s$.
We first consider the case where $\Lambda_{s,t}\neq\Lambda$.
In what follows, we construct a new total order $O_{\Lambda_{s,t}}(C)$ on the inner vertices of $C$
such that $O_{\Lambda_{s,t}}(C)$ is an extension of $\pl$ and corresponds to a generalized
zigzag cover in the set $\zl_g(\lambda,\mu;\Lambda_{s,t})$.
Let $\undl z'=\undl x'\sqcup\undl y'$ be an $(s,t)$-distribution that is compatible with $\Lambda_{s,t}$,
\textit{i.e.}, $\Lambda_i=(3,1,\ldots,1)$ iff $z_i'\in\undl x'$, and $\Lambda_j=(2,1,\ldots,1)$ iff $z_j'\in\undl y'$.
Suppose that the elements of $\undl z'$ satisfy $z'_1<\cdots<z'_{s+t}$ 
({\it i.e.}, if $z_i'=(x_{2i}',x_{2i+1}')$ is a pair, then $z_i'<z_j'$ means $x_{2i}'<x_{2i+1}'<z_j'$ in the natural order of real numbers).
Let $O_1(C)=u_1,\ldots,u_{s+t}$ be the total order determined as follows.
If $z_1'\in\undl x'$, let $u_1=\varphi^{-1}(x_1)$. Otherwise, let $u_1=\varphi^{-1}(w_{2s+1})$.
Assume that $u_i$ is already defined for some $i\in\{1,\ldots,s+t-1\}$.
Suppose that $\varphi^{-1}(x_1),\ldots,\varphi^{-1}(x_{i_0}),\varphi^{-1}(w_{2s+1}),\ldots,\varphi^{-1}(w_{2s+i_1})$
have already been used to define $u_1,\ldots,u_i$.
We set $u_{i+1}=\varphi^{-1}(x_{i_0+1})$ if $z_{i+1}'\in\undl x'$. Otherwise,
we set $u_{i+1}=\varphi^{-1}(w_{2s+i_1+1})$.
By induction, we construct a total order $O_{\Lambda_{s,t}}(C)$ on the inner vertices of $C$.

We now prove that the total order $O_{\Lambda_{s,t}}(C)$ is an extension of $\pl$,
and corresponds to a generalized
zigzag cover in the set $\zl_g(\lambda,\mu;\Lambda_{s,t})$.
Since $O_{\Lambda_{s,t}}(C)|_{\undl v_x}=O(C)|_{\undl v_x}$ and $O_{\Lambda_{s,t}}(C)|_{\undl v_y}=O(C)|_{\undl v_y}$,
the total order $O_{\Lambda_{s,t}}(C)$ is an extension of $\pl|_{\undl v_x}$ and $\pl|_{\undl v_y}$.
From Definition \ref{def:ref-zig-cov}(2) and (4),
the restriction $\pl|_{\{\varphi^{-1}(x_i),\varphi^{-1}(w_{2s+j})\}}$ of $\pl$ to 
the set $\{\varphi^{-1}(x_i),\varphi^{-1}(w_{2s+j})\}$ is a total order, if $i=1$ and $j=t$.
Otherwise, $\pl|_{\{\varphi^{-1}(x_i),\varphi^{-1}(w_{2s+j})\}}$ is an empty order.
Suppose that the $(s,t)$-tuple $\Lambda_{s,t}=(\Lambda_1,\ldots,\Lambda_{s+t})$.
Since $\Lambda_{s,t}\neq\Lambda$, 
there exists a pair $(k,l)$ such that $1\leq k<l\leq s+t$ and $\Lambda_k=(3,1,\ldots,1), \Lambda_l=(2,1,\ldots,1)$.
From the construction of $O_{\Lambda_{s,t}}(C)$ we know that
$\varphi^{-1}(x_1)$ precedes $\varphi^{-1}(w_{2s+t})$ in the total order $O_{\Lambda_{s,t}}(C)$.
Hence, $O_{\Lambda_{s,t}}(C)$ is indeed an extension of $\pl$.
Let $\varphi_{\Lambda_{s,t}}:C\to T\pb^1$ be the unique tropical cover corresponding to the
total order $O_{\Lambda_{s,t}}(C)$ according to \cite[Remark 5.3]{rau2019}.
Since $\varphi:C\to T\pb^1$ is a generalized zigzag cover,
the cover $\varphi_{\Lambda_{s,t}}:C\to T\pb^1$ is also a generalized zigzag cover by definition.
We define the map $\Phi$ by setting $\Phi(\varphi)=\varphi_{\Lambda_{s,t}}$.
The image $\varphi_{\Lambda_{s,t}}$ of $\Phi$ is a generalized zigzag cover of type
$(g,\lambda,\mu;\undl z'=\undl x'\sqcup\undl y')$.

In this paragraph, we prove that $\Phi$ is an injective map.
Let $\varphi_1:C_1\to T\pb^1$ and $\varphi_2:C_2\to T\pb^1$ be two non-isomorphic, properly mixed generalized zigzag covers.
If there is no isomorphism between the tropical curves $C_1$ and $C_2$,
then the images $\Phi(\varphi_1)$ and $\Phi(\varphi_2)$ are obviously not isomorphic.
Suppose that there is an isomorphism between the tropical curves $C_1$ and $C_2$.
Let $\undl v_y'$ (resp. $\undl v_x'$) be the set of vertices (resp. pairs of vertices) 
that are mapped onto $\undl y$ (resp. $\undl x$) by $\varphi_1$ or $\varphi_2$.
Then, the total order $O_1(C_1)=(O_1(C_1)|_{\undl v_x'}, O_1(C_1)|_{\undl v_y'})$ induced by $\varphi_1$ is not
the same as the total order $O_2(C_2)=(O_2(C_2)|_{\undl v_x'}, O_2(C_2)|_{\undl v_y'})$ induced by $\varphi_2$.
By definitions, the total orders $O_{\Lambda_{s,t}}(C_1)$ and $O_{\Lambda_{s,t}}(C_2)$ are different.
Therefore, the two covers $\Phi(\varphi_1)$ and $\Phi(\varphi_2)$ are not isomorphic.
Hence, the map $\Phi$ is an injective map.

When $\Lambda_{s,t}=\Lambda$, we define the map $\Phi$ by
considering wall-crossing behavior of properly mixed generalized zigzag covers.
Let $E_c$ be the characteristic odd edge with endpoints $v_1$ and $v_2$ as given in Definition \ref{def:ref-zig-cov} (see Figure \ref{fig:refine-zig}).
Let $E$ be the contractible edge of $C$ adjacent to $v_1$.
When the lengths of the edges $E_c$ and $E$ are shrunk to $0$,
the edges $E_c$ and $E$ degenerate to a $5$-valent vertex $v'$
in the shrunk tropical curve $C_0$ (see Figure \ref{fig:degenerate-curve}).
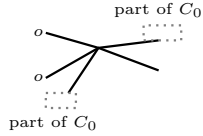
\begin{figure}[ht]
    \centering
    \begin{tikzpicture}
    \draw[line width=0.3mm] (-1.3,0.3)--(-0.6,0.1)--(0.2,0.2);
    \draw[line width=0.3mm] (-1.3,-0.3)--(-0.6,0.1)--(0.2,-0.2);
    \draw[line width=0.3mm] (-1,-0.5)--(-0.6,0.1);
    \draw[line width=0.3mm,gray,dotted] (0,0.2)--(0.5,0.2)--(0.5,0.4)--(0,0.4)--(0,0.2);
    \draw[line width=0.3mm,gray,dotted] (-1.3,-0.5)--(-0.9,-0.5)--(-0.9,-0.7)--(-1.3,-0.7)--(-1.3,-0.5);
    \draw[line width=0.3mm] (-1.4,0.3) node{\tiny $o$} (-1.4,-0.3)
    node{\tiny $o$} (0.2,0.6) node{\tiny part of $C_0$} (-1.2,-0.9) node{\tiny part of $C_0$};
    \end{tikzpicture}
    \caption{\small Degenerate curve $C_0$.}
    \label{fig:degenerate-curve}
\end{figure}
Since every edge in $C_0$ is also an edge in $C$,
the restrictions $\ol|_{C_0}$ and $W|_{C_0}$ of the orientation $\ol$
and weight function $W$ on $C$ provide an orientation and a weight function on $C_0$.
The balancing conditions at the vertices $v_1',v_1,v_2$ of $C$ imply the balancing condition at $v'$ of $C_0$.
Let $\undl v_y'=\undl v_y\setminus\{v_2\}$ and $\undl v_x'=\undl v_x\setminus\{(v_1',v_1)\}$.

Now, we construct an oriented and weighted graph with only $1$-valent and $3$-valent vertices that resolves the degenerate curve $C_0$.
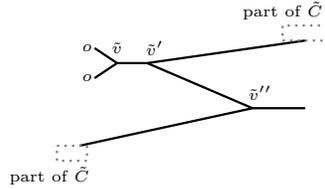
\begin{figure}[ht]
    \centering
    \begin{tikzpicture}
    \draw[line width=0.3mm] (-1.3,0.3)--(-1,0.1)--(-0.6,0.1)--(1.5,0.4);
    \draw[line width=0.3mm] (-1.3,-0.1)--(-1,0.1);
    \draw[line width=0.3mm] (-0.6,0.1)--(0.8,-0.5)--(1.5,-0.5);
    \draw[line width=0.3mm] (-1.5,-1)--(0.8,-0.5);
    \draw[line width=0.3mm,gray,dotted] (1.2,0.4)--(1.8,0.4)--(1.8,0.6)--(1.2,0.6)--(1.2,0.4);
    \draw[line width=0.3mm,gray,dotted] (-1.8,-1)--(-1.4,-1)--(-1.4,-1.2)--(-1.8,-1.2)--(-1.8,-1);
    \draw[line width=0.3mm] (-1.4,0.3) node{\tiny $o$} (-1.4,-0.1)
    node{\tiny $o$} (-1,0.3) node{\tiny$\tilde v$} (-0.5,0.3) node{\tiny$\tilde v'$}
    (0.9,-0.3) node{\tiny$\tilde v''$} (1.2,0.8) node{\tiny part of $\tilde C$} (-1.9,-1.4) node{\tiny part of $\tilde C$};
    \end{tikzpicture}
    \caption{\small tropical curve $\tilde C$.}
    \label{fig:wall-crossing}
\end{figure}
Let $l_1,l_2$ be the two inward ends of the same odd weight in $C_0$ that are adjacent to $v'$.
First, we replace $l_1$ and $l_2$ with
an inward tail $T$ that has a symmetric fork of the first type in Figure \ref{fig:gzc-tail}.
Let $\tilde v$ denote the inner vertex of the tail $T$ (see Figure \ref{fig:wall-crossing}).
The $5$-valent vertex $v'\in C_0$ is resolved to become a $4$-valent vertex.
Next, we resolve the $4$-valent vertex as illustrated in Figure \ref{fig:pair-EC}(\textrm{vii}),
and consider the two resolving $3$-valent vertices $(\tilde v',\tilde v'')$ as a pair of vertices of the new graph.
Let $\tilde C$ be the graph obtained through this process (see Figure \ref{fig:wall-crossing}).
In the above procedure, all vertices and edges of $C_0$, except for $E, E_c, (v_1',v_1)$ and $v_2$, remain changed.
We show below that the graph $\tilde C$ is oriented and weighted by $\ol$ and $W$.
The edge $E'$ with endpoints $\tilde v',\tilde v''$ is oriented from $\tilde v'$ to $\tilde v''$, and it is a contractible edge.
Note that the inward tail $T$ is oriented from its leaves to $\tilde v'$.
The restricted orientation $\ol|_{\tilde C\setminus(T\cup E')}$
and the restricted weight function $W|_{\tilde C\setminus(T\cup E')}$ provide
an orientation and a weight function on $\tilde C\setminus(T\cup E')$.
The graph $\tilde C$ is oriented, and let $\tilde\ol$ denote the orientation on $\tilde C$.
By applying the balancing conditions at $\tilde v,\tilde v', \tilde v''$,
we can determine the weights of the two new inner edges adjacent to $\tilde v'$ in $\tilde C$.
Therefore, every edge of $\tilde C$ is weighted.
Let $\tilde\pl$ denote the partial order on inner vertices of $\tilde C$ that is determined by $\tilde\ol$.

We construct a total order $O_{\Lambda}(\tilde C)$ on the inner vertices of
$\tilde C$ such that $O_{\Lambda}(\tilde C)$ is an extension of $\tilde\pl$ and corresponds to a generalized
zigzag cover in the set $\zl_g(\lambda,\mu;\Lambda)$.
Let $O_{\Lambda}(\tilde C)$ be the sequence $O(C)|_{\undl v_y'}, \tilde v, \tilde v',\tilde v'', O(C)|_{\undl v_x'}$.
This sequence defines a total order on the inner vertices of $\tilde C$.
Note that $\tilde\pl|_{\undl v_x'}=\pl|_{\undl v_x'}$ and $\tilde\pl|_{\undl v_y'}=\pl|_{\undl v_y'}$,
so we can conclude that the total order $O_{\Lambda}(\tilde C)$ is an extension of $\tilde\pl$,
as illustrated in Figure \ref{fig:wall-crossing}.
Let $\tilde\varphi: \tilde C\to T\pb^1$ be the unique tropical cover corresponding to the total order $O_{\Lambda}(\tilde C)$ according to \cite[Remark 5.3]{rau2019}.
From the construction, it is easy to see that $\tilde\varphi$ is a generalized zigzag cover
of type $(g,\lambda,\mu;\Lambda)$.
We define the map $\Phi$ as follows:
$$
\begin{aligned}
\Phi:\zl_g(\lambda,\mu;s,t)&\to\zl_g(\lambda,\mu;\Lambda)\\
\varphi&\mapsto\tilde\varphi.
\end{aligned}
$$
Let $\varphi_1:C_1\to T\pb^1$ and $\varphi_2:C_2\to T\pb^1$ be two covers in $\zl_g(\lambda,\mu;s,t)$ such that
$\Phi(\varphi_1)$ is isomorphic to $\Phi(\varphi_2)$.
Then the tropical curve $\tilde C_1$ is isomorphic to $\tilde C_2$,
and the total order $O_{\Lambda}(\tilde C_1)$ is identical to $O_{\Lambda}(\tilde C_2)$.
From the construction and Defintion \ref{def:ref-zig-cov}, it is clear that
$\tilde C_1\setminus(T_1\cup E'_1)$ is isomorphic to $\tilde C_2\setminus(T_2\cup E'_2)$
and $O(\tilde C_1\setminus(T_1\cup E'_1))=O(\tilde C_2\setminus(T_2\cup E'_2))$,
where $T_i$ and $E_i'$ are the tails and contractible edges, respectively, in the previous construction.
Consequently, we deduce that $C_1$ is isomorphic to $C_2$ and $O(C_1)=O(C_2)$.
By \cite[Remark 5.3]{rau2019}, the two covers $\varphi_1$ and $\varphi_2$ are isomorphic.
Therefore, the map $\Phi$ is injective when $\Lambda_{s,t}=\Lambda$.
Hence, we have completed the proof.
\end{proof}

\section{Uniform asymptotics}
\label{sec:7}

We apply proper zigzag numbers to investigate the uniform asymptotic growth of
real double Hurwitz numbers with triple ramification,
as both the degree and the genus tend to infinity, while adding either simple or triple branch points.

\subsection{Uniform asymptotics for (real) double Hurwitz numbers}
In this section, we study the uniform bound for the logarithmic asymptotics of real double Hurwitz numbers
as both the degree and the genus tend to infinity.
\begin{lemma}
\label{lem:asymp2}
Let $m\geq1$ and $g\geq0$ be two integers.
Then, the zigzag number
$$
Z_g((1^{2m+1}),(1^{2m+1});\Lambda_{0,4m+2g})\geq
\frac{(2g)!}{\left((2\left\lfloor\frac{g}{m}\right\rfloor+2)!\right)^{m}}\cdot(m!)^4.
$$
\end{lemma}

\begin{proof}
We first construct a zigzag cover $\varphi:C\to T\pb^1$ of type $(g,(1^{2m+1}),(1^{2m+1}),\undl z)$ in Figure \ref{fig:zig1},
which is defined as \textit{unmixed} according to \cite[Definition 5.5]{rau2019}.
\begin{figure}[ht]
    \centering
    \begin{tikzpicture}
    \draw[line width=0.4mm] (-1,-2)--(7,-2.5);
    \draw[line width=0.4mm] (7,-2.5)--(4,-3)--(5,-3.5)--(4,-4);
    \draw[line width=0.4mm,dotted] (4,-4)--(3,-4.5);
    \draw[line width=0.4mm] (3,-4.5)--(2,-5)--(5.5,-5.5)--(3.5,-6)--(11,-6.5);
\draw[line width=0.3mm,gray] (10.3,-2.5)--(11,-2.2);
\draw[line width=0.3mm,gray] (7,-2.5)--(10.3,-2.5)--(11,-2.8);
\draw[line width=0.3mm,gray] (4,-3)--(1.8,-3);
\draw[line width=0.3mm,gray] (1.2,-3)--(-0.5,-3)--(-1,-2.7);
\draw[line width=0.3mm,gray] (-0.5,-3)--(-1,-3.3);
\draw[line width=0.3mm,gray] (5,-3.5)--(10.5,-3.5)--(11,-3);
\draw[line width=0.3mm,gray] (10.5,-3.5)--(11,-4);
\draw[line width=0.3mm,gray] (2,-5)--(1.5,-5);
\draw[line width=0.3mm,gray] (0.9,-5)--(-0.8,-5)--(-1,-4.5);
\draw[line width=0.3mm,gray] (-0.8,-5)--(-1,-5.5);
\draw[line width=0.3mm,gray] (5.5,-5.5)--(10,-5.5)--(11,-5);
\draw[line width=0.3mm,gray] (10,-5.5)--(11,-6);
\draw[line width=0.3mm,gray] (3.5,-6)--(-0.3,-6)--(-1,-5.8);
\draw[line width=0.3mm,gray] (-0.3,-6)--(-1,-6.2);
\draw[line width=0.3mm,dotted] (5.4,-3.8)--(5.4,-5.3);
\draw[line width=0.2mm,dotted] (0.5,-4.7)--(0.5,-3.2);
\draw[line width=0.2mm,gray] (-1.2,-3)--(-1.3,-3)--(-1.3,-6)--(-1.2,-6);
\draw[line width=0.2mm,blue] (-1.2,-5) node{\tiny $m$ tails};
\draw[line width=0.2mm,gray] (11.2,-2.5)--(11.3,-2.5)--(11.3,-5.5)--(11.2,-5.5);
\draw[line width=0.2mm,blue] (11.2,-4.5) node{\tiny $m$ tails};
\draw[line width=0.3mm] (-1,-7)--(11,-7);
    \foreach \Point in {(-0.3,-7),(-0.5,-7),(-0.8,-7),(0.9,-7),(1.2,-7),(1.5,-7),(1.8,-7),(2,-7),(3.5,-7),(4,-7),(5,-7),(5.5,-7),(7,-7),(10,-7),(10.3,-7),(10.5,-7)}
    \draw[fill=black] \Point circle (0.03);
\draw (1.5,-3) ellipse (0.3 and 0.17);
\draw (1.2,-5) ellipse (0.3 and 0.17);
\draw[line width=0.3mm] (11.2,-2.2) node{\tiny $1$} (11.2,-2.8) node{\tiny $1$} (11.2,-3) node{\tiny $1$} (11.2,-4) node{\tiny $1$} (11.2,-5) node{\tiny $1$} (11.2,-6) node{\tiny $1$}(11.2,-6.5) node{\tiny $1$} (-1.2,-2) node{\tiny $1$}  (-1.2,-2.7) node{\tiny $1$} (-1.2,-3.3) node{\tiny $1$} (-1.2,-4.5) node{\tiny $1$} (-1.2,-5.5) node{\tiny $1$} (-1.2,-5.8) node{\tiny $1$} (-1.2,-6.2) node{\tiny $1$};
    \end{tikzpicture}
    \caption{\small zigzag cover of type $(g,(1^{2m+1}),(1^{2m+1}),\undl z)$.}
    \label{fig:zig1}
\end{figure}
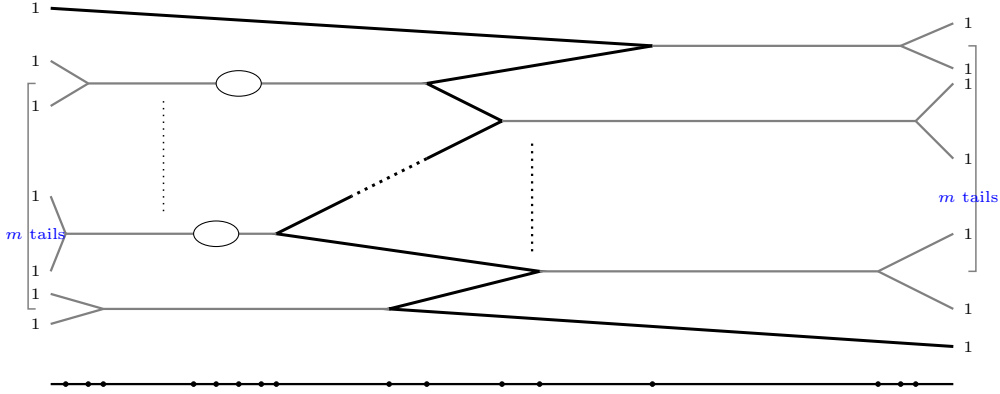
We now describe the characteristics of this unmixed zigzag cover $\varphi:C\to T\pb^1$ depicted in Figure \ref{fig:zig1}.
The string $S\subset C$ is depicted in black, with each edge in $S$ having a weight of $1$.
All symmetric cycles are positioned on inward tails.
There are $m$ inward tails and $m$ outward tails, both of type $1,1$.
There are $m$ bends in $S$ with their peaks pointing to the left, and $m$ bends in $S$ with their peaks pointing to the right.
Suppose that the elements of $\undl z\subset\rb$ satisfy $z_1<\cdots<z_{4m+2g}$.
The inner vertices of $C$ are mapped to $\undl z$ as follows:
\begin{itemize}
\item $m$ inward (resp. outward) symmetric fork vertices are mapped onto the set $\{z_1,\ldots,z_m\}$
(resp. $\{z_{3m+2g+1},\ldots,z_{4m+2g}\}$);
\item $m$ bends with peaks pointing to the left (resp. right) are mapped onto the set $\{z_{m+2g+1},\ldots,z_{2m+2g}\}$
(resp. $\{z_{2m+2g+1},\ldots,z_{3m+2g}\}$);
\item the vertices of symmetric cycles are mapped onto the set $\{z_{m+1},\ldots,z_{m+2g}\}$.
\end{itemize}
By division algorithm, there exists a unique integer $r$ such that
$g=\left\lfloor\frac{g}{m}\right\rfloor m+r$ and $0\leq r<m$.
We select $r$ inward tails and place $\left\lfloor\frac{g}{m}\right\rfloor+1$ symmetric cycles on each of these inward tails.
On the remaining $m-r$ inward tails, we place $\left\lfloor\frac{g}{m}\right\rfloor$ symmetric cycles on each.

From Remark \ref{rmk-combin1}, the tropical curve $C$ is oriented and weighted.
The orientation $\ol$ on $C$ induces a partial order $\pl$ on the inner vertices of $C$.
Furthermore, any total order on the inner vertices of $C$ that extends this partial order $\pl$
determines a unique tropical cover according to \cite[Remark 5.3]{rau2019}.
To find a lower bound for the number of tropical covers with the same source curve $C$,
we only need to determine a lower bound for the number of total orders on the inner vertices of $C$ that extend the partial order $\pl$.
The partial order $\pl$ restricts to the empty order on each of the four sets $W_1=\varphi^{-1}(\{z_1,\ldots,z_m\})$,
$W_2=\varphi^{-1}(\{z_{m+2g+1},\ldots,z_{2m+2g}\})$, 
$W_3=\varphi^{-1}(\{z_{2m+2g+1},\ldots,z_{3m+2g}\})$ and 
$W_4=\varphi^{-1}(\{z_{3m+2g+1},\ldots,z_{4m+2g}\})$.
Hence, any permutation of the elements within one of these  four sets provides a valid total order on $\cup_{i=1}^4W_i$ 
(as proved in \cite[Lemma 5.6]{rau2019}).
Let $V_i$ be the set of symmetric cycle vertices located on a tail, where $i=1,\ldots, \min(g,m)$.
When the partial order $\pl$ restricts to $V_i$, it becomes a total order.
Let $v_i\in V_i$ be a vertex in $V_i$, $i=1,\ldots, \min(g,m)$.
The partial order $\pl$ restricts to the empty order on $\{v_i,v_j\}$ if $i\neq j$.
We suppose that $V_i$ for $i=1,\ldots,r$ contains $2\left\lfloor\frac{g}{m}\right\rfloor+2$ vertices,
and $V_j$ for $j=r+1,\ldots,m$ contains $2\left\lfloor\frac{g}{m}\right\rfloor$ vertices.
The vertices of $V_1$ can be arranged in $\binom{2g}{2\left\lfloor\frac{g}{m}\right\rfloor+2}$ ways within the sequence $1,\ldots,2g$.
Once the vertices of $V_1$ are placed, the vertices of $V_2$ can be arranged in $\binom{2g-2\left\lfloor\frac{g}{m}\right\rfloor-2}{2\left\lfloor\frac{g}{m}\right\rfloor+2}$ ways,
and so on for the remaining sets $V_i$.
We get at least
$$
\begin{aligned}
&\binom{2g}{2\left\lfloor\frac{g}{m}\right\rfloor+2}\cdots\binom{2g-2(r-1)\left\lfloor\frac{g}{m}\right\rfloor-2(r-1)}{2\left\lfloor\frac{g}{m}\right\rfloor+2}
\cdot\binom{2g-2r\left\lfloor\frac{g}{m}\right\rfloor-2r}{2\left\lfloor\frac{g}{m}\right\rfloor}\cdots
\binom{2\left\lfloor\frac{g}{m}\right\rfloor}{2\left\lfloor\frac{g}{m}\right\rfloor}\\
&=\frac{(2g)!}{\left((2\left\lfloor\frac{g}{m}\right\rfloor+2)!\right)^{r}\cdot\left((2\left\lfloor\frac{g}{m}\right\rfloor)!\right)^{m-r}}
\end{aligned}
$$
total orders on $W_5=\varphi^{-1}(\undl z)\setminus(\cup_{i=1}^4W_i)$.
Hence, there are at least
$\frac{(2g)!}{\left((2\left\lfloor\frac{g}{m}\right\rfloor+2)!\right)^{m}}\cdot(m!)^4$
total orders on inner vertices of $C$.
\end{proof}

\begin{theorem}
\label{thm:zigzag1}
Let $\lambda$ and $\mu$ be two partitions such that $|\lambda|=|\mu|$ and $l(\lambda_o,\mu_o)\leq2$.
Then, there exists a fixed integer $N_0$ that depends on $\lambda$ and $\mu$ such that, for any $m>N_0$, we have
$$
Z_g((\lambda,1^{2m}),(\mu,1^{2m});\Lambda_{0,r_0})\geq Z_g((1^{2m-2N_0+1}),(1^{2m-2N_0+1});\Lambda_{0,r_1}),
$$
where $r_0=l(\lambda)+l(\mu)+2g-2+4m$, $r_1=4m-4N_0+2g$.
\end{theorem}

\begin{proof}
From \cite[Proposition 5.2]{rau2019}, there exists a zigzag cover $\varphi_1:C_1\to T\pb^1$
of type $(0,(\lambda,1,1), (\mu,1,1);\undl x_1)$,
where $\undl x_1$ is a set of $l(\lambda)+l(\mu)+2$ simple branch points.
According to the construction detailed in the proof of \cite[Proposition 5.2]{rau2019}, we have the following cases.
\begin{enumerate}[$(1)$]
\item $l(\lambda_o)=l(\mu_o)=1$: The string $S_1\subset C_1$ has an inward end $l_1$ weighted by $\lambda_o$
and an outward end $l_1'$ weighted by $\mu_o$, both of which are not part of any symmetric fork.
In the case $\mu_o\neq1$, we construct the following three types of zigzag covers.
\begin{itemize}
\item According to \cite[Proposition 5.2]{rau2019},
we have at least one zigzag cover $\varphi_2$ of type $(0,\mu_o,(1^{\mu_o}),\undl x_2)$,
where the string has an inward end $l_2$ weighted by $\mu_o$ and an outward end $l_2'$ weighted by $1$,
both of which are not part of any symmetric fork.
Let $N_0=\frac{\mu_o+1}{2}$.
\item From Lemma \ref{lem:asymp2}, there exist $Z_g((1^{2m-2N_0+1}),(1^{2m-2N_0+1});\Lambda_{0,r_1})$
zigzag covers $\varphi_3$ of type $(g,(1^{2m-2N_0+1}),(1^{2m-2N_0+1}),\undl x_3)$.
Furthermore, the string of every zigzag cover has an inward end $l_3$ and an outward end $l_3'$,
both weighted by 1 and not part of any symmetric fork.
\item From \cite[Proposition 5.2]{rau2019}, we also have at least one zigzag cover $\varphi_4$ of type $(0,(1^{\mu_o}),\mu_o,\undl x_4)$
whose string has an inward end $l_4$ weighted by $1$ and an outward end $l_4'$ weighted by $\mu_o$,
both of which are not part of any symmetric fork.
\end{itemize}
Let $\varphi_1'$ be the zigzag cover obtained by gluing $\varphi_1$ with $\varphi_2$ along $l_1'$ and $l_2$, according to Lemma \ref{lem:glue1}.
Denote by $\varphi_2'$ the zigzag cover obtained by gluing $\varphi_1'$ with $\varphi_3$ along $l_2'$ and $l_3$.
At last, we glue $\varphi_2'$ with $\varphi_4$ along $l_3'$ and $l_4$.
By applying Lemma \ref{lem:glue1}, we obtain $Z_g((1^{2m-2N_0+1}),(1^{2m-2N_0+1});\Lambda_{0,r_1})$
zigzag covers of type $(g,(\lambda,1^{2m}),(\mu,1^{2m}),\undl x)$.
In the case $\mu_o=1$, let $N_0=1$. We then glue $\varphi_1$ with zigzag cover $\varphi_3$ along $l_1'$ and $l_3$,
which results in the required inequality.
\item $l(\mu_o)=2$: Suppose that $\mu_o=(\mu_o^1,\mu_o^2)$. The string $S_1\subset C_1$ has two outward ends $E_1$, $E_2$ 
which are weighted by $\mu_o^1$ and $\mu_o^2$, respectively, and are not part of any symmetric fork. 
In this case, the required inequality can be obtained using an argument similar to that of case $(1)$, and thus we omit the details here.
\item $l(\lambda_o,\mu_o)=0$: The string $S_1\subset C_1$ has two outward ends $E_1$, $E_2$ which are both weighted by $1\in(\mu,1,1)$
and not part of any symmetric fork.
An argument similar to that of the case $(1)$ implies the required inequality, so we omit the details here.
\item $l(\lambda_o)=2$: Suppose that $\lambda_o=(\lambda_o^1,\lambda_o^2)$. 
The string $S_1\subset C_1$ has two inward ends $E_1$, $E_2$ which are weighted by $\lambda_o^1$ and $\lambda_o^2$, respectively,
and are not part of any symmetric fork.
An argument similar to that of the case $(1)$ implies the required inequality, so we omit the details here.
\end{enumerate}
\end{proof}

\begin{proof}[Proof of Theorem $\ref{thm:main2}$]
From Theorem \ref{thm:zigzag1}, there exists a fixed integer $N_0$ depending on $\lambda$ and $\mu$ such that
for any $m>N_0$, we have
$$
z_{\lambda,\mu}(g,m)\geq Z_g((1^{2m-2N_0+1}),(1^{2m-2N_0+1});\Lambda_{0,t}),
$$
where $t=4m-4N_0+2g$.  
From Lemma \ref{lem:asymp2} and Lemma \ref{lem:limit1}, we obtain that
$$
\liminf_{g,m\to\infty}\frac{\ln z_{\lambda,\mu}(g,m)}{2g\ln m+4m\ln m}
\geq1.
$$
%
%
%
\end{proof}

\begin{proof}[Proof of Theorem $\ref{coro:main3}(2)$]
From equation $(3)$ in \cite{dyz-2017} we know
$H_g^\cb(1^{d},1^d)\sim\frac{2}{d!^2}\binom{d}{2}^{2g+2d-2}$ as $g\to\infty$,
so $\ln H_g^\cb(1^{d},1^d)\sim2g\ln\frac{d(d-1)}{2}$ when $g$ tends to infinity.
It follows from the proof of \cite[Theorem 5.10]{rau2019} that
$H_g^\cb((\lambda,1^{2m}),(\mu,1^{2m}))\leq H_g^\cb(1^{|\lambda|+2m},1^{|\mu|+2m})$.
From Lemma \ref{lem:asymp2}, Lemma \ref{lem:limit1}(1) and Theorem \ref{thm:zigzag1}, we obtain  
$$
\liminf_{g\to\infty}\frac{\ln z_{\lambda,\mu}(g,m)}{\ln h^\cb_{\lambda,\mu}(g,m)}\geq \frac{\ln (m-N_0)}{\ln(\frac{(|\lambda|+2m)(|\lambda|+2m-1)}{2})}.
$$
\end{proof}

\subsection{Uniform asymptotics for real double Hurwitz numbers with triple ramification}

Let $\lambda$ and $\mu$ be two partitions such that $|\lambda|=|\mu|$.
Suppose that there exists an odd number $o\neq1$ that appears at least twice in $\lambda$.
Let $2c(\lambda,\mu)=\big||\lambda_e|+2o-|\mu_e|\big|$,
$\lambda'=\lambda\setminus(\lambda_e,o,o)$ and $\mu'=\mu\setminus\mu_e$.
We put
$$
\begin{aligned}
z_{\lambda,\mu}(g,h,m)=Z_g((\lambda,1^{2c+h+2m-1}),(\mu,1^{2c+h+2m-1});s(h),t(m)),\\
h_{\lambda,\mu}(g,h,m)=H^\cb_g((\lambda,1^{2c+h+2m-1}),(\mu,1^{2c+h+2m-1});s(h),t(m)),
\end{aligned}
$$
where $s(h)=\frac{l(\lambda')+l(\mu')}{2}+c+h$, $t(m)=l(\lambda_e)+l(\mu_e)+2g-2+4m+2c$.

\begin{theorem}
\label{thm:asymp2}
Let $g\geq0$ be an integer,
and let $\lambda$ and $\mu$ be two partitions such that $|\lambda|=|\mu|$.
Suppose that there exists an odd integer $o\neq1$ that appears at least twice in $\lambda$
and there exists an even integer $e\geq2o$ in $\mu$.
Then, there exists a fixed integer $n_0$ that depends on $\lambda$ and $\mu$ such that
for any $h>n_0+3$ and $m>1$, we have
$$
z_{\lambda,\mu}(g,h,m)\geq |\zl_0((1^{h-n_0}),(1^{h-n_0});\Lambda_{h-n_0-1,0})|\cdot Z_g((1^{2m-1}),(1^{2m-1});\Lambda_{0,4m+2g-4}).
$$
\end{theorem}

\begin{proof}
Let $\lambda'=\lambda\setminus(\lambda_e,(o,o))$ and $\mu'=\mu\setminus\mu_e$,
then $l(\lambda'_e)=l(\mu'_e)=0$.
Suppose that $||\lambda_e|+2o-|\mu_e||=2c$.
There are three cases: $|\lambda_e|+2o-|\mu_e|=2c>0$, $|\lambda_e|+2o-|\mu_e|=-2c<0$
or $|\lambda_e|+2o-|\mu_e|=0$.
We prove Theorem \ref{thm:asymp2} by considering each case separately.

Case $(1)$: $|\lambda_e|+2o-|\mu_e|=2c>0$.
We consider three pairs of partitions:
\begin{itemize}
\item $(\lambda',1^{2c})$ and $\mu'$;
\item $(o,o)$ and $(2o-1,1)$;
\item $(\lambda_e,2o-1,1^{2m})$ and $(\mu_e,1^{2m+2c-1})$.
\end{itemize}
Let $n_0$ be the integer obtained from Theorem \ref{thm:asymp1} for the partitions $(\lambda',1^{2c})$ and $\mu'$.
Let $h$ be a positive integer satisfying $h>n_0+3$.
In the following, we construct three types of tropical covers.
\begin{enumerate}[$(a)$]
\item Let $\varphi_1:C_1\to T\pb^1$
be a generalized zigzag cover of type $(0,(\lambda',1^{2c+h}), (\mu',1^{h}),\undl z_1)$
obtained from Theorem \ref{thm:asymp1}.
Let $l_1$ be a weighted 1 inward end that is not part of any symmetric fork in the non-empty subgraph $S_1\subset C_1$.
The existence of such an end is guaranteed by the proof of Theorem \ref{thm:asymp1}
(see also the construction in Lemma \ref{lem:asymp1}).
From Theorem \ref{thm:asymp1}, we have at least $|\zl_0((1^{h-n_0}),(1^{h-n_0});\Lambda_{h-n_0-1,0})|$
generalized zigzag covers of type $(0,(\lambda',1^{2c+h}), (\mu',1^{h}),\undl z_1)$.
\item A generalized zigzag cover $\varphi_2:C_2\to T\pb^1$ of type $(0,(o,o), (2o-1,1),\undl z_2)$
is given in Figure \ref{fig:proper-zig1}.
\begin{figure}[ht]
    \centering
    \begin{tikzpicture}
    \draw[line width=0.3mm] (-1,0.8)--(0,0.5)--(1,0.8);
    \draw[line width=0.3mm] (0,0.5)--(0.6,0);
    \draw[line width=0.3mm] (-1,-0.4)--(0.6,0)--(1,-0.2);
    \draw[line width=0.3mm] (-1,-1)--(1,-1);
    \foreach \Point in {(0,-1),(0.6,-1)}
    \draw[fill=black] \Point circle (0.03);
    \draw[line width=0.3mm] (-0.8,0.9) node{\tiny $o$} (-0.8,-0.2)
    node{\tiny $o$} (-0.1,0.25)  node{\tiny $o-1$} (1,-0.4) node{\tiny $2o-1$} (0.9,1) node{\tiny $1$};
    \end{tikzpicture}
    \caption{\small Tropical cover $\varphi_2:C_2\to T\pb^1$ .}
    \label{fig:proper-zig1}
\end{figure}
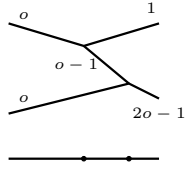
The weight of every edge of $C_2$ is marked in Figure \ref{fig:proper-zig1}.
The orientation on $C_2$ is induced by the covering map $\varphi_2$,
so the weight $1$ end and weight $2o-1$ end are oriented outward.
\item We construct zigzag covers $\varphi_3:C_3\to T\pb^1$ of type
$(g,(\lambda_e,2o-1,1^{2m}), (\mu_e,1^{2m+2c-1}),\undl z_3)$ here.
Let $\varphi_3':C_3'\to T\pb^1$ be the zigzag cover of type
$(0,(\lambda_e,2o-1,1), (\mu_e,1^{2c}),\undl z_3')$ given in Figure \ref{fig:proper-zig2}.
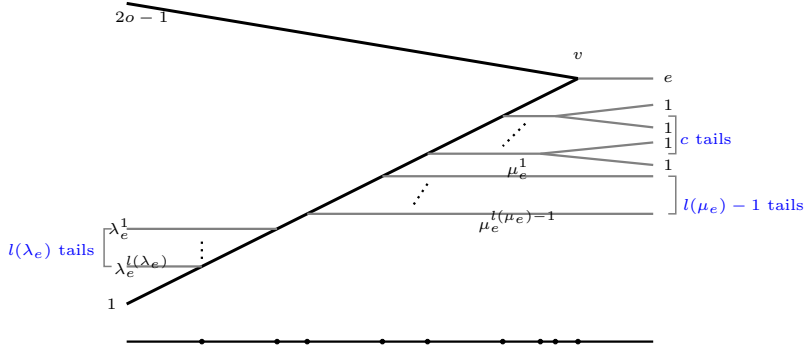
\begin{figure}[ht]
    \centering
    \begin{tikzpicture}
    \draw[line width=0.4mm] (4,1)--(10,0);
    \draw[line width=0.4mm] (10,0)--(4,-3);
\draw[line width=0.3mm,gray] (10,0)--(11,0);
\draw[line width=0.3mm,gray] (9,-0.5)--(9.7,-0.5)--(11,-0.35);
\draw[line width=0.3mm,gray] (9.7,-0.5)--(11,-0.65);
\draw[line width=0.3mm,gray] (8,-1)--(9.5,-1)--(11,-0.85);
\draw[line width=0.3mm,gray] (9.5,-1)--(11,-1.15);
\draw[line width=0.3mm,gray] (7.4,-1.3)--(11,-1.3);
\draw[line width=0.3mm,gray] (6.4,-1.8)--(11,-1.8);
\draw[line width=0.3mm,gray] (5,-2.5)--(4,-2.5);
\draw[line width=0.3mm,gray] (6,-2)--(4,-2);
\draw[line width=0.3mm,dotted] (9.3,-0.6)--(9,-0.9);
\draw[line width=0.3mm,dotted] (8,-1.4)--(7.8,-1.7);
\draw[line width=0.3mm,dotted] (5,-2.4)--(5,-2.1);
\draw[line width=0.2mm,gray] (11.2,-0.5)--(11.3,-0.5)--(11.3,-1)--(11.2,-1);
\draw[line width=0.2mm,blue] (11.7,-0.8) node{\tiny $c$ tails};
\draw[line width=0.2mm,gray] (11.2,-1.3)--(11.3,-1.3)--(11.3,-1.8)--(11.2,-1.8);
\draw[line width=0.2mm,blue] (12.2,-1.7) node{\tiny $l(\mu_e)-1$ tails};
\draw[line width=0.2mm,gray] (3.8,-2)--(3.7,-2)--(3.7,-2.5)--(3.8,-2.5);
\draw[line width=0.2mm,blue] (3,-2.3) node{\tiny $l(\lambda_e)$ tails};
\draw[line width=0.3mm] (4,-3.5)--(11,-3.5);
    \foreach \Point in {(5,-3.5),(6,-3.5),(6.4,-3.5),(7.4,-3.5),(8,-3.5),(9.5,-3.5),(9,-3.5),(9.7,-3.5),(10,-3.5)}
    \draw[fill=black] \Point circle (0.03);
    \draw[line width=0.3mm] (4.2,0.8) node{\tiny $2o-1$} (3.9,-2)
    node{\tiny $\lambda_e^1$} (4.2,-2.5)  node{\tiny $\lambda_e^{l(\lambda_e)}$} (11.2,0) node{\tiny $e$} (11.2,-0.35) node{\tiny $1$}
(11.2,-0.65) node{\tiny $1$}  (11.2,-0.85) node{\tiny $1$} (11.2,-1.15) node{\tiny $1$} (9.2,-1.2) node{\tiny $\mu_e^1$} (9.2,-1.9) node{\tiny $\mu_e^{l(\mu_e)-1}$} (3.8,-3) node{\tiny $1$} (10,0.3) node{\tiny $v$};
    \end{tikzpicture}
    \caption{\small zigzag cover of type $(0,(\lambda_e,2o-1,1), (\mu_e,1^{2c}),\undl z_3')$.}
    \label{fig:proper-zig2}
\end{figure}
From Theorem \ref{thm:zigzag1}, there are at least $Z_g((1^{2m-1}),(1^{2m-1});\Lambda_{0,4m+2g-4})$
zigzag covers of type $(g,(1^{2m}), (1^{2m}),\undl z_3'')$.
Note that in this case, the integer $N_0$ given by Theorem \ref{thm:zigzag1} is equal to $1$.
Moreover, every such zigzag cover has a string with two weighted $1$ outward ends that are not part of any symmetric fork.
From Lemma \ref{lem:glue1}, we obtain at least $Z_g((1^{2m-1}),(1^{2m-1});\Lambda_{0,4m+2g-4})$
zigzag covers $\varphi_3:C_3\to T\pb^1$ of type $(g,(\lambda_e,2o-1,1^{2m}), (\mu_e,1^{2m+2c-1}),\undl z_3)$.
\end{enumerate}

We first glue the weighted $1$ outward end in $C_2$ with the weighted 1 inward end $l_1\subset C_1$.
From Lemma \ref{lem:glue1}, we get a generalized zigzag cover $\varphi_1':C_1\to T\pb^1$ of type
$(0,(\lambda',o,o,1^{2c+h-1}),(\mu',2o-1,1^h),\undl z_1')$.
Furthermore, there exists a weighted $2o-1$ outward end in the non-empty subgraph $S_1'\subset C_1'$.
Then we glue the weighted $2o-1$ outward end in $C_1'$ with the weighted $2o-1$ end in the string of $C_3$.
By Lemma \ref{lem:glue1} again, we know that there are at least
$$
|\zl_0((1^{h-n_0}),(1^{h-n_0});\Lambda_{h-n_0-1,0})|\cdot Z_g((1^{2m-1}),(1^{2m-1});\Lambda_{0,4m+2g-4})
$$
generalized zigzag covers of type $(g,(\lambda,1^r),(\mu,1^r);s,t)$, where
$$
\begin{aligned}
r&=2c+h-1+2m, s=\frac{1}{2}(l(\lambda')+l(\mu'))+c+h,\\
t=l(\lambda)+l(\mu)&+2g-2+2r-2s=l(\lambda_e)+l(\mu_e)+2g-2+2c+4m.
\end{aligned}
$$
Note that from Lemma \ref{lem:glue1}, all the glued generalized zigzag covers are properly mixed.

Case $|\lambda_e|+2o-|\mu_e|=-2c<0$:
We consider three pairs of partitions:  $\lambda'$ and $(\mu',1^{2c})$, $(o,o)$ and $(2o-1,1)$,
$(\lambda_e,2o-1,1^{2m+2c})$ and $(\mu_e,1^{2m-1})$.
Denote by $n_0$ the integer obtained from Theorem \ref{thm:asymp1} for the partitions $\lambda'$ and $(\mu',1^{2c})$.
Let $h$ be a positive integer satisfying $h>n_0+3$.
The construction in the first case also gives us three types of tropical covers.
\begin{itemize}
\item According to Theorem \ref{thm:asymp1}, there are at least
$|\zl_0((1^{h-n_0}),(1^{h-n_0});\Lambda_{h-n_0-1,0})|$ generalized zigzag covers
$\varphi_1:C_1\to T\pb^1$ of type $(0,(\lambda',1^{h}), (\mu',1^{2c+h}),\undl z_1)$.
\item A generalized zigzag cover $\varphi_2:C_2\to T\pb^1$ of type $(0,(o,o), (2o-1,1),\undl z_2)$.
\item $Z_g((1^{2m-1}),(1^{2m-1});\Lambda_{0,4m+2g-4})$ zigzag covers $\varphi_3:C_3\to T\pb^1$ of type $(g,(\lambda_e,2o-1,1^{2m+2c}), (\mu_e,1^{2m-1}),\undl z_3)$.
\end{itemize}
Since the rest of the proof is nearly identical to that of the first case, we omit it for brevity.

The case that $-|\lambda_e|-2o+|\mu_e|=2c=0$ can be proved in a similar way,
so we omit the details here.
\end{proof}

\begin{proof}[Proof of Theorem $\ref{coro:main2}$]
From the proof of \cite[Theorem 5.10]{rau2019}, 
the number
$$
h^\cb_{\lambda,\mu}(g,h,m)\leq H^\cb_g(1^{|\lambda|+2c+h+2m-1},1^{|\mu|+2c+h+2m-1}).
$$
Let $c_0=n_0+3$, where $n_0$ is the fixed integer determined in Theorem $\ref{thm:asymp2}$.
Since $\ln x!\sim x\ln x$ as $x\to\infty$,
Theorem $\ref{coro:main2}$ follows from Theorem \ref{thm:asymp2}, Lemma \ref{lem:asymp2}
and equation $(5)$ in \cite{dyz-2017}.
\end{proof}

\begin{proof}[Proof of Theorem $\ref{coro:main3}(1)$]
Let $c_0=n_0+3$, where $n_0$ is the fixed integer determined in Theorem $\ref{thm:asymp2}$.
From Theorem \ref{thm:asymp2},
Lemma \ref{lem:asymp2}, Lemma \ref{lem:limit1}(1)
and equation $(3)$ in \cite{dyz-2017} we know that
if $h>c_0$, $m>1$ are fixed, we have
$$
\liminf_{g\to\infty}\frac{\ln z_{\lambda,\mu}(g,h,m)}{\ln h^\cb_{\lambda,\mu}(g,h,m)}\geq
\frac{\ln(m-1)}{ \ln(\frac{(|\lambda|+2c+h+2m-1)(|\lambda|+2c+h+2m-2)}{2})}.
$$
\end{proof}



\begin{proof}[Proof of Theorem $\ref{thm:main1}$]
From Theorem \ref{thm:asymp2}, Theorem \ref{thm:asymp1}, Lemma \ref{lem:asymp1} and Lemma \ref{lem:asymp2}, the number
$z_{\lambda,\mu}(g,h,m)\geq \frac{(2g)!}{\left((2\left\lfloor\frac{g}{m-1}\right\rfloor+2)!\right)^{m-1}}\cdot(m-1)!^4\cdot\left\lfloor\frac{h}{3}-\frac{n_0+1}{3}\right\rfloor!$.
Let $S(g,h,m)$ be defined in equation $(\ref{eq:asymp-main2})$,
then we have $\frac{\ln z_{\lambda,\mu}(g,h,m)}{2g\ln m+4m\ln m+\frac{h}{3}\ln h}\geq S(g,h,m)$.
From Lemma \ref{lem:limit2}, we obtain that $\lim\limits_{g,h,m\to\infty}S(g,h,m)=1$.
We have completed the proof of Theorem \ref{thm:main1}.
\end{proof}

\appendix
\section{Proofs of two uniform limits}
\label{sec:a}
In this appendix, we prove two lemmas about uniform limits.

\begin{lemma}\label{lem:limit1}
Let $m, g$ be two integers with $m>1$ and $g\geq0$.
Then the following statements are true.
\begin{enumerate}[$(1)$]
\item If $m$ is fixed and $g$ tends to infinity, we have
$$
\ln\left(\frac{(2g)!}{\left((2\left\lfloor\frac{g}{m}\right\rfloor+2)!\right)^{m}}\cdot(m!)^4\right)
\sim 2g\ln m.
$$
\item As $g$ and $m$ both tend to infinity we have
$$
\ln\left(\frac{(2g)!}{\left((2\left\lfloor\frac{g}{m}\right\rfloor+2)!\right)^{m}}\cdot(m!)^4\right)
\sim 2g\ln m+4m\ln m.
$$
\end{enumerate}
\end{lemma}

\begin{proof}
It is easy to see that
$$
\ln\left(\frac{(2g)!}{\left((2\left\lfloor\frac{g}{m}\right\rfloor+2)!\right)^{m}}\cdot(m!)^4\right)
=\ln (2g)!+4\ln m!-m\ln\left (2\left\lfloor\frac{g}{m}\right\rfloor+2\right)!.
$$
Note that $\frac{g}{m}+\theta=\left\lfloor\frac{g}{m}\right\rfloor+1$, where $0<\theta\leq1$.
From the Stirling's formula (see equation (1.4) on page 2 of \cite{wong-2001} for example) we know
$$
\ln n!=(n+\frac{1}{2})\ln n-n+\frac{1}{2}\ln 2\pi +o(1).
$$
Then
$$
\begin{aligned}
&\lim_{g\to\infty}\frac{\ln\left(\frac{(2g)!}{\left((2\left\lfloor\frac{g}{m}\right\rfloor+2)!\right)^{m}}\cdot(m!)^4\right)}{2g\ln m}\\
=&\lim_{g\to\infty}\frac{(2g+\frac{1}{2})\ln (2g)-2g-m\left (2\left\lfloor\frac{g}{m}\right\rfloor+2+\frac{1}{2}\right)\ln\left (2\left\lfloor\frac{g}{m}\right\rfloor+2\right)+m\left (2\left\lfloor\frac{g}{m}\right\rfloor+2\right)+O(1)}{2g\ln m}\\
=&\lim_{g\to\infty}\frac{2g\ln (2g)-(2g+2m\theta)\ln\left (\frac{2g}{m}+2\theta\right)+\frac{1}{2}\ln(2g)-\frac{m}{2}\ln\left (\frac{2g}{m}+2\theta\right)+2m\theta+O(1)}{2g\ln m}\\
=&\lim_{g\to\infty}\frac{2g\ln (2g)-2g\ln\left (\frac{2g}{m}+2\theta\right)}{2g\ln m}+0\\
=&1+\lim_{g\to\infty}\frac{\ln \frac{2g}{m}-\ln\left (\frac{2g}{m}+2\theta\right)}{\ln m}\\
=&1+\lim_{g\to\infty}\frac{\ln \frac{1}{1+\frac{2m\theta}{2g}}}{\ln m}\\
=&1.
\end{aligned}
$$
We have completed the proof of the statement (1).

Suppose that the conclusion in statement (2) is not right. 
Then there exist a positive constant $\delta$ and two sequences $\{g_i\}_{i=1}^\infty$, $\{m_i\}_{i=1}^\infty$ 
such that
$\lim\limits_{i\rightarrow\infty}g_i=\lim\limits_{i\rightarrow\infty}m_i=\infty$ and
$$
\left|\frac{\ln S(g_i,m_i)}{2g_i\ln m_i+4m_i\ln m_i}-1\right|\geq\delta,~ \forall i\in\nb^+ ,
$$
where $S(g_i,m_i)=\frac{(2g_i)!}{\left((2\left\lfloor\frac{g_i}{m_i}\right\rfloor+2)!\right)^{m_i}}\cdot(m_i!)^4$.

Then we have
$$
\begin{aligned}
\frac{\ln S(g_i,m_i)}{2g_i\ln m_i+4m_i\ln m_i}
=&\frac{\ln (2g_i)!+4\ln m_i!-m_i\ln\left (2\left\lfloor\frac{g_i}{m_i}\right\rfloor+2\right)!}{2g_i\ln m_i+4m_i\ln m_i}\\
=&\frac{\frac{\ln (2g_i)!}{2g_i\ln m_i}+\frac{4\ln m_i!}{4m_i\ln m_i}\frac{4m_i}{2g_i}
-\frac{m_i\ln\left(2\left\lfloor\frac{g_i}{m_i}\right\rfloor+2\right)!}{2g_i\ln m_i}}{1+\frac{4m_i}{2g_i}}.
\end{aligned}
$$

If $\{\frac{g_i}{m_i}\}_{i=1}^\infty$ is unbounded, up to a subsequence, 
we may assume that $\lim\limits_{i\rightarrow\infty}\frac{g_i}{m_i}=\infty$.
Then we obtain
$$
\begin{aligned}
&\lim_{i\to\infty}\frac{\ln S(g_i,m_i)}{2g_i\ln m_i+4m_i\ln m_i}\\
=&\lim_{i\to\infty}\left(\frac{\ln (2g_i)!}{2g_i\ln m_i}+\frac{4\ln m_i!}{4m_i\ln m_i}\frac{4m_i}{2g_i}
-\frac{m_i\ln\left(2\left\lfloor\frac{g_i}{m_i}\right\rfloor+2\right)!}{2g_i\ln m_i}\right)\\
=&\lim_{i\to\infty}\frac{\ln (2g_i)!-m_i\ln \left(2\left\lfloor\frac{g_i}{m_i}\right\rfloor+2\right)!}{2g_i\ln m_i}.
\end{aligned}
$$
From the Stirling's formula we obtain that
$$
\begin{aligned}
&\lim_{i\to\infty}\frac{\ln S(g_i,m_i)}{2g_i\ln m_i+4m_i\ln m_i}\\
=&\lim_{i\to\infty}\frac{(2g_i+\frac{1}{2})\ln (2g_i)-m_i\left(2\left\lfloor\frac{g_i}{m_i}\right\rfloor+2+\frac{1}{2}\right)
\ln\left(2\left\lfloor\frac{g_i}{m_i}\right\rfloor+2\right)+O(g_i)+O(\left\lfloor\frac{g_i}{m_i}\right\rfloor)}{2g_i\ln m_i}\\
=&\lim_{i\to\infty}\frac{2g_i\ln (2g_i)-(2m_i\left\lfloor\frac{g_i}{m_i}\right\rfloor+2m_i)\ln\left(2\left\lfloor\frac{g_i}{m_i}\right\rfloor+2\right)}{2g_i\ln m_i}.
\end{aligned}
$$
Note that $\frac{g_i}{m_i}+\theta_i=\left\lfloor\frac{g_i}{m_i}\right\rfloor+1$, where $0<\theta_i\leq1$.
Then
$$
\begin{aligned}
&\lim_{i\to\infty}\frac{\ln S(g_i,m_i)}{2g_i\ln m_i+4m_i\ln m_i}\\
=&\lim_{i\to\infty}\frac{2g_i\ln (2g_i)-(2g_i+2m_i\theta_i)\ln\left (\frac{2g_i}{m_i}+2\theta_i\right)}{2g_i\ln m_i}\\
=&\lim_{i\to\infty}\left(\frac{2g_i\ln (2g_i)-2g_i\ln\left (\frac{2g_i}{m_i}+2\theta_i\right)}{2g_i\ln m_i}+\frac{2m_i\theta_i\ln\left (\frac{2g_i}{m_i}+2\theta_i\right)}{2g_i\ln m_i}\right)\\
=&1+\lim_{i\to\infty}\frac{\theta_i\ln\left (\frac{2g_i}{m_i}+2\theta_i\right)}{\frac{g_i}{m_i}\ln m_i}\\
=&1.
\end{aligned}
$$
This is a contradiction.

If $\{\frac{g_i}{m_i}\}_{i=1}^\infty$ is bounded, then up to a subsequence, 
we may assume that $\lim\limits_{i\rightarrow\infty}\frac{g_i}{m_i}=b$ for some $b\in[0,\infty)$.
We first consider the case that $b\in(0,\infty)$. In this case,
we have $\lim\limits_{i\to\infty}\frac{\ln g_i}{\ln m_i}=1$ and
$
\lim\limits_{i\to\infty}\left(\frac{m_i\ln (2\left\lfloor\frac{g_i}{m_i}\right\rfloor+2)!}{2g_i\ln m_i}\right)=0.
$
Then 
$$
\lim_{i\to\infty}\frac{\ln S(g_i,m_i)}{2g_i\ln m_i+4m_i\ln m_i}
=\lim_{i\to\infty}\frac{\frac{\ln (2g_i)!}{2g_i\ln g_i}\frac{\ln g_i}{\ln m_i}+\frac{4\ln m_i!}{4m_i\ln m_i}\frac{4m_i}{2g_i}}{1+\frac{4m_i}{2g_i}}=1.
$$
The last equality follows from the fact that $\ln x!\sim x\ln x$ when $x\to\infty$.
This is a contradiction.


Now we consider the case $b=0$. 
Up to a subsequence, 
we may assume $g_i<m_i$.
Then $S(g_i,m_i)=\frac{(2g_i)!}{2^{m_i}}\cdot (m_i!)^4$.
A direct calculation shows that
$$
\begin{aligned}
\lim_{i\to\infty}\frac{\ln S(g_i,m_i)}{2g_i\ln m_i+4m_i\ln m_i}&=\lim_{i\to\infty}\frac{\ln (2g_i)!+4\ln m_i!-m_i\ln2}{2g_i\ln m_i+4m_i\ln m_i}\\
&=\lim_{i\to\infty}\frac{\frac{\ln (2g_i)!}{2g_i\ln g_i}\frac{\ln g_i}{\ln m_i}\frac{2g_i}{4m_i}+\frac{4\ln m_i!}{4m_i\ln m_i}-\frac{m_i\ln2}{4m_i\ln m_i}}{\frac{2g_i}{4m_i}+1}=1.
\end{aligned}
$$
This is a contradiction.
Therefore, we have completed the proof of statement (2).
\end{proof}

\begin{lemma}\label{lem:limit2}
Let
\begin{equation}\label{eq:asymp-main2}
S(g,h,m):=\frac{\ln (2g)!+4\ln(m-1)!-(m-1)\ln (2\left\lfloor\frac{g}{m-1}\right\rfloor+2)!+\ln\left\lfloor\frac{h}{3}-\frac{n_0+1}{3}\right\rfloor!}{2g\ln m+4m\ln m+\frac{h}{3}\ln h}.
\end{equation}
Then $\lim\limits_{g,h,m\to\infty}S(g,h,m)=1$.
\end{lemma}

\begin{proof}
The proof of Lemma \ref{lem:limit2} is similar to that of Lemma \ref{lem:limit1}(2).
Suppose that the conclusion is not right. Then there exist a positive constant $\delta$ 
and three sequences $\{g_i\}_{i=1}^\infty$, $\{m_i\}_{i=1}^\infty$, $\{h_i\}_{i=1}^\infty$ 
such that
$\lim\limits_{i\rightarrow\infty}g_i=\lim\limits_{i\rightarrow\infty}m_i=\lim\limits_{i\rightarrow\infty}h_i=\infty$ and
$$
|S(g_i,h_i,m_i)-1|\geq\delta,~ \forall i\in\nb^+.
$$
We put 
$$
\begin{aligned}
&a(h)=\frac{\ln\left\lfloor\frac{h}{3}-\frac{n_0+1}{3}\right\rfloor!}{\frac{h}{3}\ln h},~~
b(h,m)=\frac{\frac{h}{3}\ln h}{2m\ln m},\\
S(g,m)=&\frac{\ln (2g)!}{2g\ln g}\frac{\ln g}{\ln m}+\frac{4\ln(m-1)!}{4m\ln m}\frac{4m}{2g}
-\frac{(m-1)\ln (2\left\lfloor\frac{g}{m-1}\right\rfloor+2)!}{2g\ln m}.
\end{aligned}
$$
Then 
$$
S(g_i,h_i,m_i)=\frac{S(g_i,m_i)-a(h_i)-\frac{4m_i}{2g_i}a(h_i)}{1+\frac{4m_i}{2g_i}+b(h_i,m_i)\frac{m_i}{g_i}}+a(h_i).
$$

If $\{\frac{g_i}{m_i}\}_{i=1}^\infty$ is unbounded, up to a subsequence, 
we may assume that $\lim\limits_{i\rightarrow\infty}\frac{g_i}{m_i}=\infty$.
We have
$$
\begin{aligned}
&\lim_{g_i,h_i,m_i\to\infty}(S(g_i,m_i)-a(h_i)-\frac{4m_i}{2g_i}a(h_i))\\
=&\lim_{g_i,h_i,m_i\to\infty}\left(\frac{\ln (2g_i)!}{2g_i\ln m_i}-\frac{(m_i-1)\ln (2\left\lfloor\frac{g_i}{m_i-1}\right\rfloor+2)!}{2g_i\ln m_i}\right)-1.
\end{aligned}
$$
From the calculation in the proof of Lemma \ref{lem:limit1}, we know that 
$$
\lim\limits_{g_i,h_i,m_i\to\infty}(S(g_i,m_i)-a(h_i)-\frac{4m_i}{2g_i}a(h_i))=0.
$$
Since $1+\frac{4m_i}{2g_i}+b(h_i,m_i)\frac{m_i}{g_i}\geq1$,
we get 
$$
\lim_{g_i,h_i,m_i\to\infty}S(g_i,h_i,m_i)=
\lim_{g_i,h_i,m_i\to\infty}\frac{S(g_i,m_i)-a(h_i)-\frac{4m_i}{2g_i}a(h_i)}{1+\frac{4m_i}{2g_i}+b(h_i,m_i)\frac{m_i}{g_i}}+1=1.
$$
This is a contradiction.

If $\{\frac{g_i}{m_i}\}_{i=1}^\infty$ is bounded, up to a subsequence, 
we may assume that $\lim\limits_{i\rightarrow\infty}\frac{g_i}{m_i}=c\in[0,\infty)$.
In the case that $c\neq0$, one has $\lim\limits_{i\rightarrow\infty}\frac{\ln g_i}{\ln m_i}=1$.
Hence, we obtain
$$
\lim_{g_i,h_i,m_i\to\infty}(S(g_i,m_i)-a(h_i)-\frac{4m_i}{2g_i}a(h_i))
=1+\frac{2}{c}-0-1-\frac{2}{c}=0.
$$
It follows from $1+\frac{4m_i}{2g_i}+b(h_i,m_i)\frac{m_i}{g_i}\geq1$ that 
$\lim\limits_{g_i,h_i,m_i\to\infty}S(g_i,h_i,m_i)=1$. This is a contradiction.

In the case $c=0$, up to a subsequence, we may assume that $g_i<m_i-1$ for all $i$.
Then 
$$
\begin{aligned}
S(g_i,h_i,m_i)=&\frac{\ln (2g_i)!+4\ln(m_i-1)!-(m_i-1)\ln 2+\ln\left\lfloor\frac{h_i}{3}-\frac{n_0+1}{3}\right\rfloor!}{2g_i\ln m+4m_i\ln m+\frac{h_i}{3}\ln h_i}\\
=&\frac{\frac{\ln(2g_i)!}{2g_i\ln g_i}\frac{2g_i\ln g_i}{4m_i\ln m_i}+\frac{4\ln (m_i-1)!}{4m_i\ln m_i}
-\frac{(m_i-1)\ln 2}{4m_i\ln m_i}+\frac{\ln\left\lfloor\frac{h_i}{3}-\frac{n_0+1}{3}\right\rfloor!}{\frac{h_i}{3}\ln h_i}\frac{\frac{h_i}{3}\ln h_i}{4m_i\ln m_i}}{\frac{\frac{h_i}{3}\ln h_i}{4m_i\ln m_i}+1+\frac{2g_i\ln m_i}{4m_i\ln m_i}}.
\end{aligned}
$$
Let 
$$
\begin{aligned}
a(h_i)=&\frac{\ln\left\lfloor\frac{h_i}{3}-\frac{n_0+1}{3}\right\rfloor!}{\frac{h_i}{3}\ln h_i},~~
b(h_i,m_i)=\frac{\frac{h_i}{3}\ln h_i}{4m_i\ln m_i}\\
S(g_i,m_i)=&\frac{\ln(2g_i)!}{2g_i\ln g_i}\frac{2g_i\ln g_i}{4m_i\ln m_i}+\frac{4\ln (m_i-1)!}{4m_i\ln m_i}
-\frac{(m_i-1)\ln 2}{4m_i\ln m_i}.
\end{aligned}
$$
Then
$$
S(g_i,h_i,m_i)=\frac{S(g_i,m_i)-a(h_i)-\frac{g_i}{2m_i}a(h_i)}{1+\frac{g_i}{2m_i}+b(h_i,m_i)}+a(h_i).
$$
Note that
$$
\lim_{g_i,h_i,m_i\to\infty}(S(g_i,m_i)-a(h_i)-\frac{g_i}{2m_i}a(h_i))=0+1-0-1-0=0.
$$
We obtain that $\lim\limits_{g_i,h_i,m_i\to\infty}S(g_i,h_i,m_i)=1$ from $1+\frac{g_i}{2m_i}+b(h_i,m_i)\geq1$.
This is again a contradiction. We have completed the proof.
\end{proof}


\section*{Acknowledgements}
The first author is grateful to Ilia Itenberg for
pointing out the problem of asymptotic behavior of real double Hurwitz numbers
when non-simple branch points are added, during his visit to IMJ-PRG in 2020.
The authors thank Di Yang for suggesting the problem of uniform asymptotics of double Hurwitz numbers,
as well as for valuable discussions and suggestions.
The authors would like to thank Chenglang Yang and Zhiyuan Wang for their helpful discussions,
and they also thank Jianfeng Wu for carefully reading the manuscript.
Y. Ding was supported by the National Natural Science Foundation of China (No.12101565), and
the Natural Science Foundation of Henan (No. 212300410287).
H. Liu was supported by the National Natural Science Foundation of China (No.12171439).


\end{document}